\documentclass[dvipsnames,onefignum,onetabnum]{siamart171218}
\usepackage[top=1.4in,bottom=1.275in,left=1.4in,right=1.4in]{geometry}

\usepackage[T1]{fontenc}
\usepackage[english]{babel}%
\usepackage{csquotes}%
\usepackage{hyphenat}%
\usepackage{siunitx}
\usepackage{placeins}
\usepackage{cite}%
\usepackage[begintext=``, endtext='']{quoting}
\SetBlockEnvironment{quoting}
\SetBlockThreshold{1}

\newcommand{\revised}[1]{{#1}}
\newcommand{\changed}[1]{{#1}}

\usepackage{amsfonts}
\usepackage{amssymb}
\usepackage{amscd}
\usepackage{mathtools}
\usepackage[mathscr]{eucal}
\usepackage{calligra}
\DeclareMathAlphabet{\mathcalligra}{T1}{calligra}{m}{n}
\DeclareFontShape{T1}{calligra}{m}{n}{<->s*[1.1]callig15}{}
\usepackage{bm}
\usepackage{autonum}%
\usepackage[nice]{nicefrac}

\usepackage{tikz}
\usepackage{pgf}
\usepackage{pgfplots}
\usepackage{pgfplotstable}
\usepackage{pgf-pie}
\usepackage{shellesc}
\pgfplotsset{compat=newest}
\usetikzlibrary{spy}
\usetikzlibrary{shapes}
\usetikzlibrary{backgrounds}
\usetikzlibrary{shadows}
\usetikzlibrary{matrix}
\usetikzlibrary{external}
\usepgfplotslibrary{fillbetween}
\usepackage{comment}
\usepackage{graphicx}
\usepackage[colorinlistoftodos,prependcaption,textsize=small]{todonotes}
\usepackage{cleveref}

\crefname{equation}{}{}
\usepackage{csvsimple}
\usepackage{booktabs}
\usepackage{environ}
\usepackage{adjustbox}
\usepackage{relsize}
\usepackage[font=small,figurewithin=section]{caption}%
\usepackage[subrefformat=parens,labelformat=simple]{subcaption}%

\usepackage{afterpage}

\newtheorem{remark}{Remark}[section]

\let\min\relax \DeclareMathOperator*\min{\vphantom{p}min}
\let\max\relax \DeclareMathOperator*\max{\vphantom{p}max}
\let\subset\relax \DeclareMathOperator{\subset}{\subseteq}
 
\let\tilde\widetilde
\let\hat\widehat

\DeclarePairedDelimiter\floor{\lfloor}{\rfloor}

\DeclareMathOperator*{\supp}{supp}

\newcommand{\sspace}{\hspace{0.25pt}}

\newcommand{\approxsol}[1]{{\tilde{#1}}_{h}}

\renewcommand{\a}{a}

\newcommand{\R}{\mathbb{R}} %
\newcommand{\N}{\mathbb{N}} %
\newcommand{\dd}{\,\mathrm{d}} %
\newcommand{\bdry}{\partial} %
\newcommand{\tr}{\mathrm{tr}} %
\newcommand{\spann}{\mathrm{span}} %
\newcommand{\adj}{\mathrm{adj}} %

\newcommand{\upperone}[2]{\raisebox{-0.2\height}{$#1{}^1\!$}}
\newcommand{\upone}{\mathpalette\upperone\relax}
\newcommand{\lowertwo}[2]{\raisebox{-0.1\height}{$#1{}_{\!2}$}}

\newcommand{\lowtwo}{\mathpalette\lowertwo\relax}

\newcommand{\frachalf}[2]{\raisebox{0.0\height}{$#1{}\upone/\lowtwo$}}

\newcommand{\onehalf}{\mathpalette\frachalf\relax}

\newcommand{\balpha}{{\bm{\alpha}}}

\newcommand{\bdelta}{{\bm{\delta}}}

\newcommand{\bxi}{\bm{\xi}}

\newcommand{\bvarphi}{\bm{\varphi}}

\newcommand{\bme}{\bm{e}}

\newcommand{\bmi}{\bm{i}}
\newcommand{\bmj}{\bm{j}}

\newcommand{\bmn}{\bm{n}}

\newcommand{\bmp}{\bm{p}}

\newcommand{\bmx}{\bm{x}}
\newcommand{\bmy}{\bm{y}}

\newcommand{\bmJ}{\bm{J}}

\newcommand{\sfc}{\mathsf{c}}

\newcommand{\sff}{\mathsf{f}}

\newcommand{\sfu}{\mathsf{u}}
\newcommand{\sfv}{\mathsf{v}}
\newcommand{\sfw}{\mathsf{w}}

\newcommand{\sfA}{\mathsf{A}}
\newcommand{\sfB}{\mathsf{B}}

\newcommand{\sfJ}{\mathsf{J}}

\newcommand{\sfM}{\mathsf{M}}
\newcommand{\sfN}{\mathsf{N}}

\newcommand{\mcD}{\mathcal{D}}

\newcommand{\mcI}{\mathcal{I}}

\newcommand{\mcO}{\mathcal{O}}
\newcommand{\mcP}{\mathcal{P}}
\newcommand{\mcQ}{\mathcal{Q}}

\newcommand{\scD}{\mathscr{D}}

\newcommand{\bbN}{\mathbb{N}}

\newcommand{\bbX}{\mathbb{X}}

\newcommand{\bff}{\mathbf{f}}
\newcommand{\bfg}{\mathbf{g}}

\newcommand{\bfu}{\mathbf{u}}
\newcommand{\bfv}{\mathbf{v}}

\newcolumntype{?}{!{\vrule width 1.2pt}}

\makeatletter
\newsavebox{\measure@tikzpicture}
\NewEnviron{scaletikzpicturetowidth}[1]{%
	\tikzifexternalizingnext{%
		\def\tikz@width{#1}%
		\def\tikzscale{1}\begin{lrbox}{\measure@tikzpicture}%
			\tikzset{external/export next=false,external/optimize=false}%
			\BODY
		\end{lrbox}%
		\pgfmathparse{#1/\wd\measure@tikzpicture}%
		\edef\tikzscale{\pgfmathresult}%
		\BODY
	}{%
		\BODY
	}%
}
\makeatother

\definecolor{color1}{rgb}{0, 0.4470, 0.7410}
\definecolor{color2}{rgb}{0.8500, 0.3250, 0.0980}
\definecolor{color3}{rgb}{0.9290, 0.6940, 0.1250}
\definecolor{color4}{rgb}{0.7060, 0.3840, 0.7650}
\definecolor{color5}{rgb}{0.4660, 0.6740, 0.1880}
\definecolor{color6}{rgb}{0.3010, 0.7450, 0.9330}
\definecolor{color7}{rgb}{0.6350, 0.0780, 0.1840}

\newcommand{\logLogSlopeTriangle}[6]
{

	\pgfplotsextra
	{
		\pgfkeysgetvalue{/pgfplots/xmin}{\xmin}
		\pgfkeysgetvalue{/pgfplots/xmax}{\xmax}
		\pgfkeysgetvalue{/pgfplots/ymin}{\ymin}
		\pgfkeysgetvalue{/pgfplots/ymax}{\ymax}

		\pgfmathsetmacro{\xArel}{#1-#2}
		\pgfmathsetmacro{\yArel}{#3}
		\pgfmathsetmacro{\xBrel}{#1}
		\pgfmathsetmacro{\yBrel}{\yArel}
		\pgfmathsetmacro{\xCrel}{\xArel}

		\pgfmathsetmacro{\lnxB}{\xmin*(1-(#1-#2))+\xmax*(#1-#2)} %
		\pgfmathsetmacro{\lnxA}{\xmin*(1-#1)+\xmax*#1} %
		\pgfmathsetmacro{\lnyA}{\ymin*(1-#3)+\ymax*#3} %
		\pgfmathsetmacro{\lnyC}{\lnyA+#5/#4*(\lnxA-\lnxB)}
		\pgfmathsetmacro{\yCrel}{\lnyC-\ymin)/(\ymax-\ymin)} %

		\coordinate (A) at (rel axis cs:\xArel,\yArel);
		\coordinate (B) at (rel axis cs:\xBrel,\yBrel);
		\coordinate (C) at (rel axis cs:\xCrel,\yCrel);

		\draw[gray!75!black, dashed, line width=1.0pt, #6]   (A)-- node[pos=0.5,anchor=north] {#4}
		(B)-- 
		(C)-- node[pos=0.5,anchor=east] {#5}
		cycle;
	}
}

\pgfplotsset{
  log x ticks with fixed point/.style={
      xticklabel={
        \pgfkeys{/pgf/fpu=true}
        \pgfmathparse{exp(\tick)}%
        \pgfmathprintnumber[fixed relative, precision=3]{\pgfmathresult}
        \pgfkeys{/pgf/fpu=false}
      }
  },
  log y ticks with fixed point/.style={
      yticklabel={
        \pgfkeys{/pgf/fpu=true}
        \pgfmathparse{exp(\tick)}%
        \pgfmathprintnumber[fixed relative, precision=3]{\pgfmathresult}
        \pgfkeys{/pgf/fpu=false}
      }
  }
}

\tikzset{
  ashadow/.style={opacity=.25, shadow xshift=0.07, shadow yshift=-0.07},
}
\def\arrow{
  (-0.3,0.0) [rounded corners=0.5] -- (-0.1,0.2) [rounded corners=0] -- (-0.1,0.1) [rounded corners=0.5] -- (0.3,0.1) -- (0.3,-0.1) [rounded corners=0] -- (-0.1,-0.1) [rounded corners=0.5] -- (-0.1,-0.2) [rounded corners=0.8] -- cycle
}

\definecolor{CustomGreen}{RGB}{65,169,50}

\pgfplotstableread{./Results/MatrixSparsity/A_p2.dat}\mytable
 \graphicspath{{./Figures/}}

\title{The surrogate matrix methodology: Low-cost assembly for isogeometric analysis}

\author{Daniel~Drzisga\thanks{Lehrstuhl f\"ur Numerische Mathematik, Fakult\"at f\"ur Mathematik (M2), Technische Universit\"at M\"unchen, Garching bei M\"unchen (\email{drzisga@ma.tum.de}, \email{keith@ma.tum.de}, \email{wohlmuth@ma.tum.de})}
\and Brendan~Keith\footnotemark[1]
\and Barbara~Wohlmuth\footnotemark[1]}

\begin{document}

\maketitle

\begin{abstract}
A new methodology in isogeometric analysis (IGA) is presented.
This methodology delivers low-cost variable-scale approximations (surrogates) of the matrices which IGA conventionally requires to be computed from element-scale quadrature formulas.
To generate surrogate matrices, quadrature must only be performed on certain elements in the computational domain.
This, in turn, determines only a subset of the entries in the final matrix.
The remaining matrix entries are computed by a simple B-spline interpolation procedure.
Poisson's equation, membrane vibration, plate bending, and Stokes' flow problems are studied.
In these problems, the use of surrogate matrices has a negligible impact on solution accuracy.
Because only a small fraction of the original quadrature must be performed, we are able to report beyond a fifty-fold reduction in overall assembly time in the same software.
The capacity for even further speed-ups is clearly demonstrated.
The implementation used here was achieved by a small number of modifications to the open-source IGA software library GeoPDEs.
Similar modifications could be made to other present-day software libraries.
\end{abstract} 

\section{Introduction} %
\label{sec:introduction}

Avoiding unnecessary work is of utmost importance when computing at the frontiers of contemporary research.
To frame a workable definition, recall that practical simulations in science and engineering involve a large number of possible sources of error.
For instance, we highlight the categories of modeling error, numerical error, and data error, each of which have many subcategories.
The total error in a simulation is controlled by the aggregate of each relevant source of error.
In this paper, ``unnecessary work'' \textemdash{} or, more precisely, \emph{over-computation} \textemdash{} is any machine expense used to drive one source of error in a problem far below the total error.
It cannot be overstated that removing sources of over-computation can have an outsized influence on the computational cost of getting an accurate solution.

In some instances, circumventing over-computation is the simplest way to accelerate a numerical algorithm.
For example, in the use of iterative methods, for both linear and non-linear problems, it has long been acknowledged that \emph{over-solving} a discretized problem is a negligent expense.
Relaxing iterative solver errors usually reduces to just adjusting the tolerances naturally built into established algorithms.
In other instances, sources of over-computation are less conspicuous and avoiding them requires the development of new algorithms.
For example, in the field of uncertainty quantification, it has recently come to light that sampling error can be relaxed \textemdash{} and, in turn, computational cost can be significantly reduced \textemdash{} by the use of a tunable \emph{surrogate response surface} \cite{butler2012posteriori,MATTIS201836}.

The focus of this article is the Galerkin form of isogeometric analysis (IGA) \cite{hughes2005isogeometric,cottrell2009isogeometric}.
At first sight, in view of the long list of computer methods which rose beforehand, the Galerkin isogeometric method may be seen as a rather paradigmatic approach to the discretization of partial differential equations (PDEs).
Indeed, Galerkin IGA methods are little more than finite element methods which employ non-uniform rational B-spline (NURBS) bases \cite{hollig2003finite}.
Although it was immediately shown by Hughes et al. \cite{hughes2005isogeometric} that the use of such a basis improves the interoperability between computer-aided design (CAD) and PDE analysis, many other benefits of the IGA approach were also demonstrated early on in the IGA literature.
Of particular note, the arbitrary smoothness of NURBS bases generally improves the accuracy per degree of freedom and lends itself to convenient techniques for the discretization of high-order PDEs \cite{hughes2018mathematics,hollig2003finite}.
It is these and other serendipitous features of IGA which have attributing to its truly meteoric success in modern computational science and engineering research.

It is well-established that traditional isogeometric methods face a great computational burden at the point of matrix assembly.
This is due, in part, to the large support of the basis functions.
Although many other common concerns are naturally alleviated by the IGA paradigm, this particular challenge is clearly evidenced by the expansive literature on quadrature rules and accelerated assembly algorithms \cite{hiemstra2017optimal,bressan2018sum,sangalli2018matrix,mantzaflaris2017low,hofreither2018black,ANTOLIN2015817,mantzaflaris2015integration,Mantzaflaris2014,Calabro2017,Fahrendorf2018,hughes2010efficient,auricchio2012simple,hiemstra2019fast}.
Indeed, we may further accentuate this remark with the following quote from the 2014 review article \cite{da2014mathematical}:
\blockquote{\ldots at the moment the assembly of the matrix is the most time-consuming part of isogeometric codes. The development of optimal assembly procedures is an important task required to render isogeometric methods a competitive technology.}

In this article, we present a simple methodology to avoid \emph{over-assembling} matrices in IGA.
Roughly speaking, it requires performing quadrature for only a small fraction of the trial and test basis function interactions and then {approximating the rest through}, for example, {interpolation}.
This leads to a large sparse matrix where the majority of entries have not been computed using any quadrature at all.
\changed{Usually}, such matrices will not coincide with the ones generated by performing quadrature for every non-zero entry (cf. \Cref{sub:polynomial_reproduction}), but they can be interpreted as surrogates for those matrices.

The main idea \changed{used here} was first introduced in the context of first-order finite elements by Bauer et al. in \cite{bauer2017two}.
Thereafter, applications to peta-scale geodynamical simulations were presented in \cite{bauer2018large,bauer2018new} and a theoretical analysis was given in \cite{drzisga2018surrogate}.
In the massively parallel applications \cite{bauer2017two,bauer2018large,bauer2018new}, it was natural to work with so-called ``macro-meshes'' as well as a piecewise polynomial space for resolving the surrogate matrices.
This choice was motivated by a low communication cost across the faces of the macro-elements, a convenient cache-aware implementation, and the fact that a hybrid mesh structure allowed for extremely fast evaluation of the three-dimensional polynomials; see \cite{bauer2018large,bauer2018new} for further details.
In contrast, the surrogate matrices in this paper are computed using a B-spline interpolation space.
With this particular strategy, we demonstrate that the cost of matrix assembly in conventional IGA codes can be reduced by an order of magnitude.

Our approach bears some similarities to the integration by interpolation and lookup (IIL) approach proposed in \cite{Mantzaflaris2014,mantzaflaris2015integration}.
\changed{In those two works, an integrand factor from the weak form, composed of both the coefficients of the underlying PDE as well as the geometry mapping, is approximated.}
In this work, the actual entries of the final matrix are shown to be related to a small number of smooth so-called \emph{stencil functions}; instead of a factor in the integrand, it is these stencil functions which are approximated.

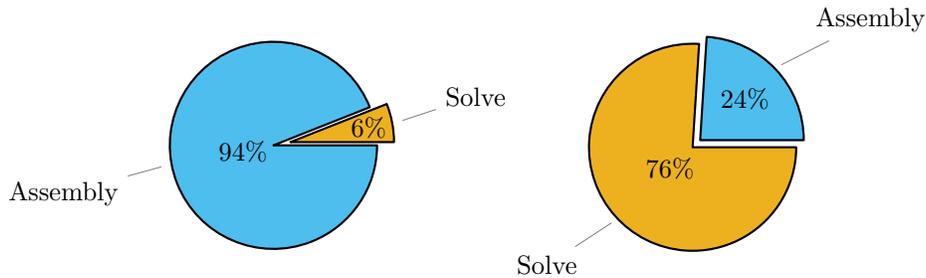
\begin{figure}
\centering
\begin{scaletikzpicturetowidth}{0.7\textwidth}
\begin{tikzpicture}[scale=\tikzscale]
\pie[color={color3,color6}, text = pin, explode=0.25]{6/Solve, 94/Assembly}
\pie[color={color6,color3}, text = pin, explode=0.15, pos={12,0}]{24/Assembly, 76/Solve}
\end{tikzpicture}
\end{scaletikzpicturetowidth}
\caption{\label{fig:PieCharts} Left: Distribution of computational cost with standard IGA. Right: Distribution of computational cost with a surrogate IGA strategy. (Timings taken from the experiment presented in \Cref{fig:dynamic_sampling_p2_timing,fig:dynamic_sampling_p2} for $c = 3$. Both experiments ran on a single compute core and the default MATLAB backslash operator has been used as a solver.)}
\end{figure}

An advantage of the IIL approach is that, in theory, it does not require a uniform knot vector assumption (cf. \Cref{sec:preliminaries}).
However, in practice, this assumption is necessary in order to obtain compact lookup tables \cite{mantzaflaris2015integration}.
On the other hand, one advantage of our approach is that it can be easily implemented using existing IGA assembly paradigms.
Another advantage is that the implementation is identical whether using a B-spline or a NURBS basis.

Before moving on, some other important remarks deserve to be emphasized:
\begin{itemize}
    \item
    The methodology we propose for IGA applications is essentially independent of the quadrature rule used at the individual element or basis function level.
    \changed{This lays bare the possibility for it to be used in conjunction with many other cutting edge techniques for accelerated IGA assembly.}
    \item
    \changed{For our new surrogate methods, it would be most efficient if the matrix entries which require quadrature were to be computed based on individual basis function interactions.
    This leaves out standard element-by-element assembly strategies, but would work well with the control point pairwise method proposed in \cite{Karatarakis2014} or, ideally, with a row-by-row approach; e.g., \cite{hiemstra2019fast,Calabro2017,sangalli2018matrix}.
    Nevertheless, one should still expect to see significant speed-ups with surrogate methods in standard element-by-element codes, at least for moderate polynomial orders.}
    In order to underscore this fact, we did not develop a stand-alone code.
    Instead, we implemented \changed{our} surrogate methods by \changed{simply} modifying the assembly routines in the open-source library GeoPDEs \cite{de2011geopdes,vazquez2016new}, \changed{leaving ever other aspect of the code fixed}.
    For illustration, the reader may refer to the left and right sides of \cref{fig:PieCharts} to compare the relative timings before and after some relatively minor changes were made to this software (cf. \Cref{app:computing_surrogates_with_existing_iga_codes} and \cite{drzisga2019igasurrogateimpl}).
    In both cases, the differences in solving time and solution accuracy were {negligible}.
    We expect that most other element-by-element IGA codes should be easy to modify in a similar manner.
    \item
    Many efficient assembly strategies for IGA see their performance advantage only in the high polynomial order regime.
    \changed{Here, the performance usually grows with each $h$-refinement.
    Indeed, at just over one million degrees of freedom, our experiments demonstrate assembly speed-ups beyond \emph{\textbf{fifty times}}, in the exact same code, with a simple second-order NURBS basis (see \Cref{ssub:dynamic_sampling_length}).}
\end{itemize}

In our experiments, we analyze surrogate IGA methods for Poisson's equation, membrane vibration, plate bending, and Stokes' flow problems.
The Poisson case is analyzed in detail and the additional experiments are provided in order to motivate further study.
It is our eventual goal to adapt our methods to a matrix-free framework, similar to what has been used recently in low-order settings \cite{drzisga2018surrogate,bauer2017two,bauer2018large,bauer2018new}.
This would certainly be helpful in order to reach the full potential of IGA in extreme scale computations.

In the next section, we take stock of the majority of mathematical notation used in the remainder of the paper.
In \Cref{sec:surrogate_matrices_exploiting_basis_structure}, we introduce the notion of a stencil function in the IGA context.
In \Cref{sec:surrogate_matrices_interpolation_of_stencil_functions}, we investigate the accuracy of B-spline interpolation with regard to stencil functions.
In \Cref{sec:surrogate_matrices_preserving_structure}, we use interpolants of the stencil functions (i.e., surrogate stencil functions) to define surrogate matrices for IGA.
\Cref{app:computing_surrogates_with_existing_iga_codes} consists of a brief description of our software implementation.
A more complete description is provided in \cite{drzisga2019igasurrogateimpl}.
In \Cref{sec:Poissons_equation,sec:transverse_vibrations_of_an_isotropic_membrane,sec:the_biharmonic_equation,sec:stokes_equation}, we examine surrogate IGA methods for Poisson's equation, membrane vibration, plate bending, and Stokes' flow problems, respectively.
\Cref{app:marsden_s_identity} is included to support some of the analysis carried out in \Cref{sec:surrogate_matrices_interpolation_of_stencil_functions}.

\section{Preliminaries} %
\label{sec:preliminaries}

In this section, we lay out the principal mathematical focus and notation of the paper.

\subsection{Model problems and notation} %
\label{sub:notation_and_model_problems}
Let $\Omega\subset\R^n$ be a domain, $n=2,3$.
Let $V = V(\Omega)$ be a Hilbert space over $\R$, the field of real numbers, and let $V^\ast$ denote its topological dual.
For historical reasons, we proceed by adopting notation from the $h$-version of the finite element method and thus let $V_h$ denote a finite-dimensional subspace of $V$.
Although we also deal with a number of important alternatives (see, e.g., \Cref{sec:transverse_vibrations_of_an_isotropic_membrane,sec:stokes_equation}), we are chiefly interested in the following three problems:
\begin{subequations}
\label{eq:VariationalFormulations}
\begin{alignat}{3}
    &
    \text{Find } u \in V \text{ satisfying }
    &
    \quad
    a(u,v)
    &=
    F(v)
    \quad
    &&\text{for all } v\in V
    \,.
\label{eq:ContinousVF}
    \\
    &
    \text{Find } u_h \in V_h \text{ satisfying }
    &
    \quad
    a(u_h,v_h)
    &=
    F(v_h)
    \quad
    &&\text{for all } v_h\in V_h
    \,.
\label{eq:DiscreteVF}
    \\
    &
    \text{Find } \tilde{u}_h \in V_h \text{ satisfying }
    &
    \quad
    \tilde{a}(\tilde{u}_h,v_h)
    &=
    F(v_h)
    \quad
    &&\text{for all } v_h\in V_h
    \,.
\label{eq:SurrogateVF}
\end{alignat}
\end{subequations}
Here and throughout, $a:V\times V \to \R$ is a continuous and coercive bilinear form, $\tilde{a}:V_h\times V_h \to \R$ is an approximation of $a|_{V_h\times V_h}$, hereby deemed the \emph{surrogate} for $a(\cdot,\cdot)$, and $F\in V^\ast$ is a bounded linear functional.

\begin{figure}
\centering
    \includegraphics[width=0.45\textwidth]{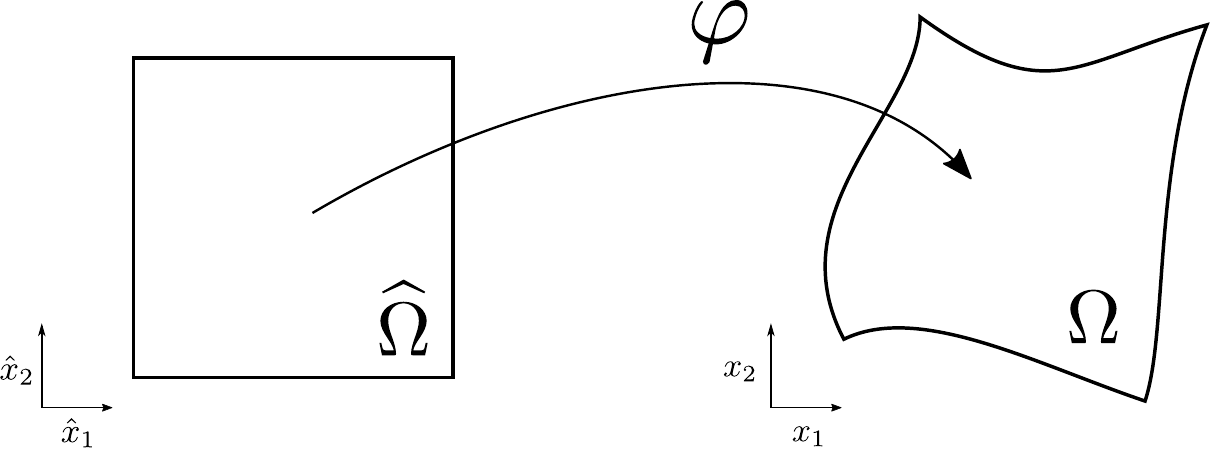}
    \caption{\label{fig:NURBSmapping} Illustration of a smooth transformation from a parametric domain $\hat{\Omega}\subset\R^2$ to a physical domain $\Omega\subset\R^2$.}
\end{figure}

To keep the exposition simple and to the point, we will assume that the physical domain of every problem $\Omega\subset\R^n$ is defined as the image of a single parametric domain $\hat{\Omega} = (0,1)^n$.
This leaves all of our analysis in the {single} patch geometry setting, $\Omega = \bvarphi(\hat{\Omega})$, for some diffeomorphism $\bvarphi:\hat{\Omega}\to \R^n$ of sufficient regularity; see \Cref{fig:NURBSmapping}.
The single patch setting is by no means a necessary assumption.
The entirety of the analysis considered here can easily be generalized to the {multi}-patch setting (cf. \Cref{sub:the_multi_patch_setting}).
However, in order to stay in the isogeometric setting, we assume that $\bvarphi(\hat{\bmx}) = \sum_i \sfc_i \hat{N}_i(\hat{\bmx})$, where each $\sfc_i\in \R^n$ is a {control point} vector and each $\hat{N}_i(\hat{\bmx})$ is a NURBS basis function on the parametric domain $\hat{\Omega}$.
Here, NURBS basis functions are defined in the standard way, as described in \Cref{sub:nurbs}.

For matrices $\sfM \in \R^{l\times m}$, define the ${\max}$-norm, $\|\sfM\|_{\max} = \max_{i,j}|\sfM_{ij}|$.
For any function $v:\Omega\to \R$, we will use the notation, $\|v\|_0$, $\|v\|_1$, and $\|v\|_2$, for the canonical \mbox{$L^2(\Omega)$-,} \mbox{$H^1(\Omega)$-,} and \mbox{$H^2(\Omega)$-norms,} respectively.
\changed{
Moreover, if $v$ is smooth, we define its support as $\supp(v) = \{\bmx\in\Omega \,:\, v(\bmx) \neq 0\}$.
}
When dealing with a domain $\mcD\subset\Omega$, denote the related $L^2(\mcD)$, $H^1(\mcD)$, and $H^2(\mcD)$ norms by $\|v\|_{0,\mcD}$, $\|v\|_{1,\mcD}$, and $\|v\|_{2,\mcD}$, respectively.
Denote the space of univariate polynomials of degree at most $q$ by $\mcP_q$.
Likewise, denote the space of multivariate polynomials of degree at most $q$, in each Cartesian direction $\bme_i$, by $\mcQ_{q} = [\mcP_{q}]^n$ and denote $\mcQ_{q}(\mcD) = \{f|_\mcD \,:\, f\in \mcQ_{q}\}$.
We will often deal with Cartesian subdomains $\mcD= \mcD_1\times\cdots\times \mcD_n$.
In this case, it is natural to deal with Cartesian--Sobolev seminorms; e.g., $[f]_{W^{r,\infty}(\mcD)} = \sum_{i=1}^n \|D^{r\cdot\bme_i} f\|_{L^\infty(\mcD)}$.
Note that $[f]_{W^{r,\infty}(\mcD)} \leq |f|_{W^{r,\infty}(\mcD)} = \sum_{|\bm{\alpha}|=r} \|D^{\balpha} f\|_{L^\infty(\mcD)}$.
All remaining notation will be defined as it arises.

\subsection{Cardinal B-splines and NURBS} %
\label{sub:nurbs}

Let $m\geq 2p+1$ and $\{b_k\}_{k=1}^m$ be an order $p$ B-spline basis on the unit interval $(0,1)$.
Let $N = m^n$ and let $\{\hat{N}_i\}_{i=1}^N$ be a corresponding NURBS basis on $\hat{\Omega}$.
Namely,
\begin{equation}
  \hat{N}_i(\hat{\bmx})
  =
  \frac{w_i \hat{B}_i(\hat{\bmx})}{\sum_j w_j \hat{B}_j(\hat{\bmx})}
  \,,
  \qquad
  \hat{\bmx}=(\hat{x}_1,\ldots,\hat{x}_n)\in\hat{\Omega}
  \,,
\label{eq:NURBS}
\end{equation}
where each $\hat{B}_i(\hat{\bmx}) = {b}_{i_1}(\hat{x}_1) \cdots {b}_{i_n}(\hat{x}_n)$ is a multivariate B-spline of uniform order $p$ and each $w_i>0$ is a fixed weight parameter.
Here and from now on, we identify every global index $1\leq i \leq N$ with a multi-index $\bmi = (i_1,\ldots,i_n)$, $1\leq i_k\leq m$, through the colexicographical relationship $i = i_1 + (i_2 -1) m + \cdots + (i_n -1) m^{n-1}$.

Generally, a univariate B-spline basis $\{{b}_k\}_{k=1}^{m}$ is defined by an ordered multiset, or \emph{knot vector}, $\Xi = \{\xi_1,\ldots,\xi_{m+p+1}\}$.
In this paper, we deal only with \emph{open uniform knot vectors}; i.e., $\xi_1,\ldots,\xi_{p+1}=0$, $\xi_{m+1},\ldots,\xi_{m+p+1}=1$, and $\xi_{k+1} - \xi_k = \frac{1}{m-p}$, otherwise.
The quantity $h = \max_{1\leq k\leq m-1}|\xi_{k+1}-\xi_{k}| = \frac{1}{m-p}$ will be an important parameter for us, which we hereby refer to as the \emph{mesh size}.
Clearly, we could consider NURBS spaces with different orders $p_1,\ldots,p_n$ in each Cartesian direction \cite{piegl2012nurbs,hughes2005isogeometric}.
In order to simplify the exposition, we avoid this complication.

\begin{figure}
  \centering
  \includegraphics[trim=0.2cm 0cm 1cm 0cm,clip=true,height=3.1cm]{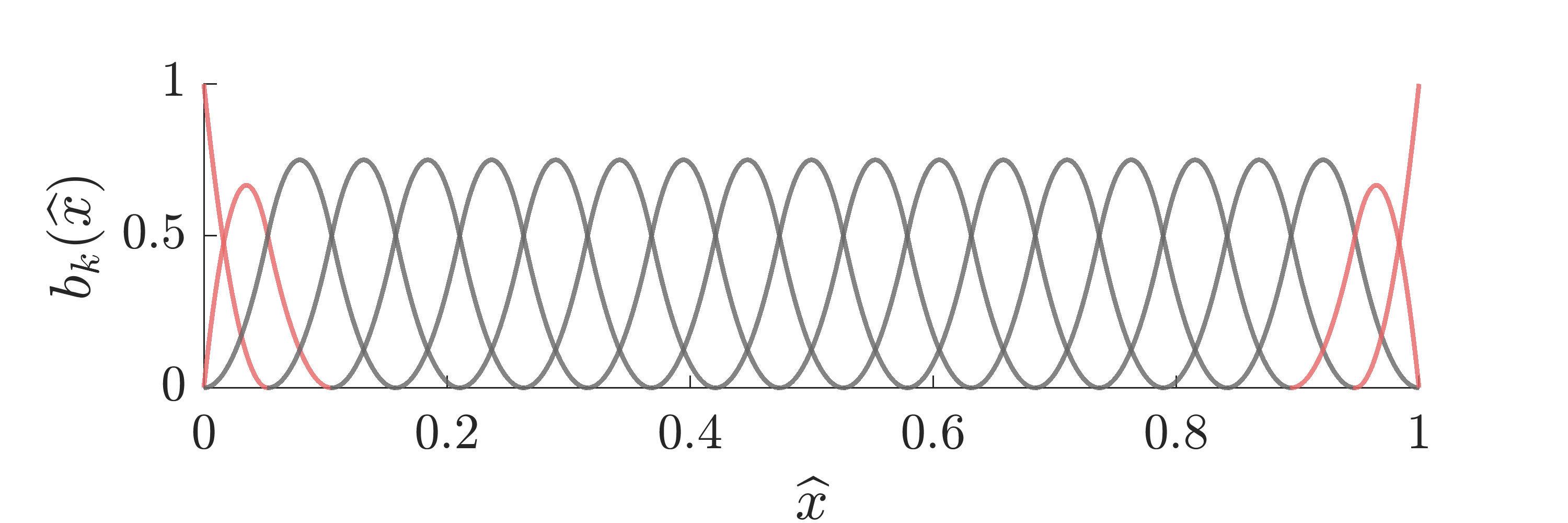}
  \hspace{0.5cm}
  \includegraphics[trim=2cm 0.2cm 2cm -0.2cm,clip=true,height=3.1cm]{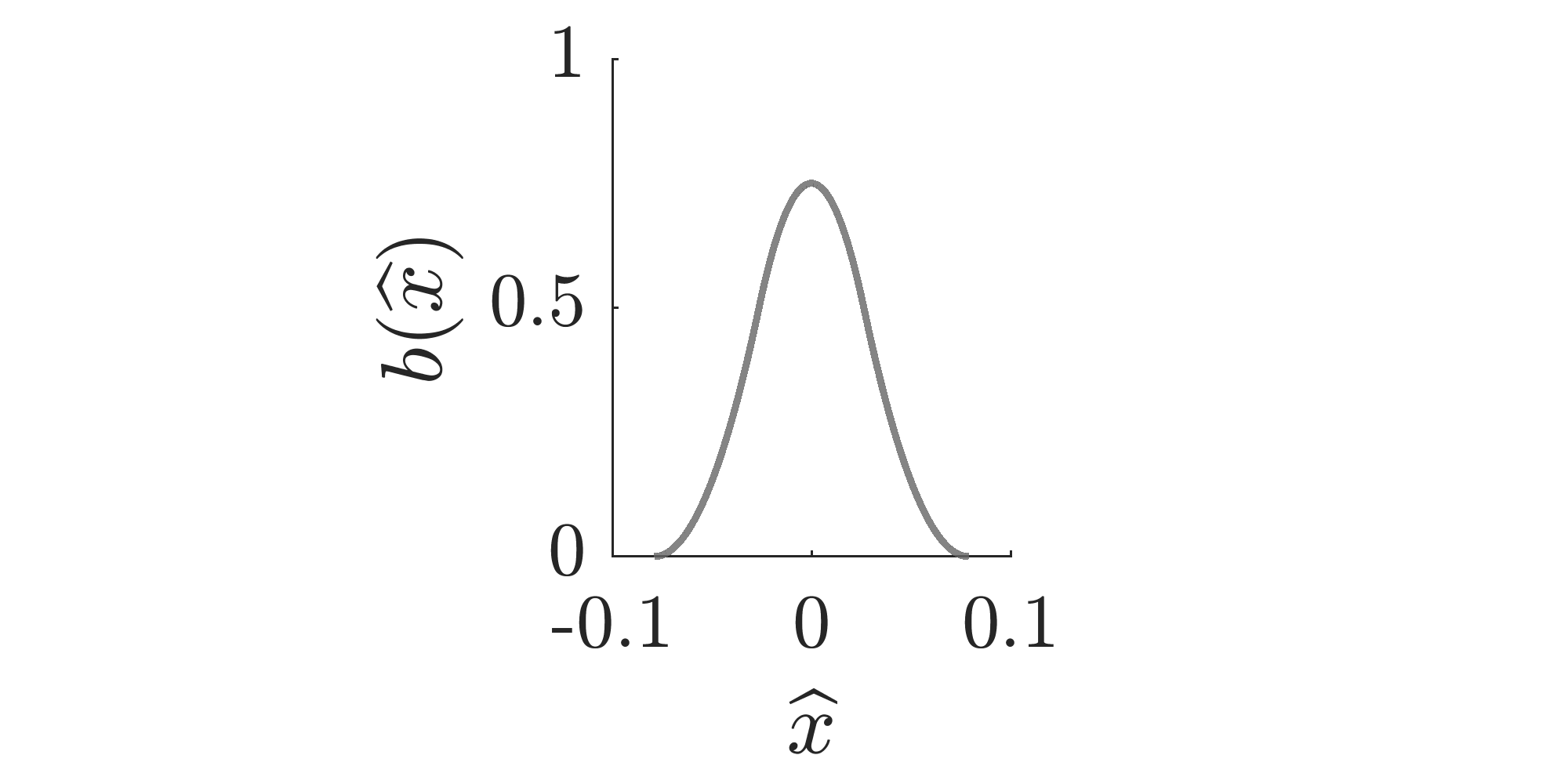}
\caption{\label{fig:B-splines} Left: $1$D B-spline basis functions $\{b_k\}_{k=1}^{21}$ generated by the open uniform knot vector $\Xi^{(2)}$, defined in~\cref{eq:OpenKnotVectorEg}. Right: Each gray basis function is equivalent, up to translation, to the function $b(\hat{x})$. The red functions are obviously not.} 
\end{figure}

For an explicit example of an open uniform knot vector, consider
\begin{equation}
  \Xi^{(p)}
  =
  \{\underbrace{0,\cdots,0}_{\text{$p$ times}}\}\cup\Big\{\frac{k}{19}\Big\}_{k=0}^{19}\cup\{\underbrace{1,\cdots,1}_{\text{$p$ times}}\}
  \,.
\label{eq:OpenKnotVectorEg}
\end{equation}
The corresponding $p=2$ B-spline basis is depicted in \cref{fig:B-splines}.
Observe that all but four of the basis functions (highlighted in red) are identical, up to an equally spaced set of translations.
These functions are called \emph{cardinal B-splines} \cite{schoenberg1973cardinal,schoenberg1946contributionsA,schoenberg1946contributionsB}.

Let $\tilde{x}^{(k)} = (k - \frac{p+1}{2})\cdot h$, for each ${k} = p+1,\ldots,m-p$.
In general, there are always $m-2p$ univariate cardinal B-spline basis functions which can each be expressed ${b}_{k}(\hat{x}) = {b}(\hat{x} - \tilde{x}^{(k)})$, for some function, ${b}(\hat{x})$, centered at the origin \revised{with support in $(-\frac{p+1}{2} \cdot{h}, \frac{p+1}{2} \cdot{h})$} (see, e.g., \cref{fig:B-splines}).
\revised{The $\tilde{x}^{(k)}$ correspond to the midpoints of each function ${b}_{k}$.}
Likewise, there are $(m-2p)^n$ multivariate cardinal B-splines.
That is $\hat{B}_i(\hat{\bmx}) = \hat{B}(\hat{\bmx} - \tilde{\bmx}_i)$, where $\tilde{\bmx}_i = \big(\tilde{x}^{(i_1)},\ldots, \tilde{x}^{(i_n)})$ and $\hat{B}(\hat{\bmx}) = b({\hat{x}_{1}})\cdots b({\hat{x}_{n}})$.
For future reference, define the set of all such $\tilde{\bmx}_i$ as $\tilde{\bbX}$.
Also, notice that the ratio of cardinal B-splines basis functions to total B-spline basis functions quickly tends to unity, $\big(\frac{m-2p}{m}\big)^n\to 1$, as $m$ increases.

\section{Surrogate matrices: Exploiting basis structure} %
\label{sec:surrogate_matrices_exploiting_basis_structure}

Equations~\cref{eq:DiscreteVF,eq:SurrogateVF}, respectively, induce two related matrix equations,
\begin{equation}
  \sfA\sfu = \sff\qquad \text{ and } \qquad\tilde{\sfA}\tilde{\sfu} = \sff,
\label{eq:DiscreteVariationalProblems}
\end{equation}
for basis function coefficients $\sfu,\tilde{\sfu}\in\R^N$.
As mentioned previously, the key idea in this paper is constructing the majority of the surrogate stiffness matrix $\tilde{\sfA}$ via \emph{interpolation} of the true stiffness matrix $\sfA$.
In this section, we first describe exactly what is meant by this statement and then demonstrate how isogeometric analysis makes it possible.

\subsection{Stencil functions} %
\label{sub:exploiting_basis_structure1}

Recall~\cref{eq:DiscreteVF}.
Generally, every function $v_h\in V_h$ can be identified with a unique function on the domain $\hat{\Omega}$ through a suitable pushforward operator $\bvarphi_\ast$.
Namely, $v_h = \bvarphi_\ast \hat{v}_h$.
Define $\hat{V}_h$ be the set of all such $\hat{v}_h$, which is a discrete space in the parametric domain $\hat{\Omega}$.
Accordingly, the bilinear form $a:V_h\times V_h\to \R$ can be identified with a parametric domain bilinear form $\hat{a}:\hat{V}_h\times \hat{V}_h\to \R$ in such a way that $a(w_h,v_h) = \hat{a}(\hat{w}_h,\hat{v}_h)$, for all $\hat{w}_h,\hat{v}_h\in \hat{V}_h$.

Let $\{\phi_i\} = \{\bvarphi_\ast \hat{\phi}_i\}$ be a basis for $V_h$ with $\{\hat{\phi}_i\}$ the corresponding basis for $\hat{V}_h$.
The fundamental observation in the surrogate matrix methodology now follows.
If, $\hat{\phi}$ is some fixed reference function and, for a set of indices $i,j$, $\hat{\phi}_i(\hat{\bmx}) = \hat{\phi}(\hat{\bmx}-\tilde{\bmx}_i)$ and $\hat{\phi}_j(\hat{\bmx}) = \hat{\phi}(\hat{\bmx}-\tilde{\bmx}_j)$, then
\begin{equation}
  \hat{a}(\hat{\phi}_j,\hat{\phi}_i)
  =
  \hat{a}(\hat{\phi}(\cdot-\tilde{\bmx}_j),\hat{\phi}(\cdot-\tilde{\bmx}_i))
  :=
  \Phi(\tilde{\bmx}_j,\tilde{\bmx}_i)
  \,.
\label{eq:TemporaryStencilFunction}
\end{equation}
Here, in the rightmost equality, the definition of a new scalar-valued function $\Phi(\cdot,\cdot)$ has been made, wherein any dependence on the mesh size $h$ has been implicitly assumed. 
This function, $\Phi(\tilde{\bmx}_j,\tilde{\bmx}_i)$, may also be expressed in terms of $\tilde{\bmx}_i$ and a translation $\bdelta = \tilde{\bmx}_j - \tilde{\bmx}_i$.
In this alternative characterization, after denoting $\sfA_{ij} = a(\phi_j,\phi_i)$, we may write
\begin{equation}
  \sfA_{ij}
  =
  \Phi_\bdelta(\tilde{\bmx}_i)
\label{eq:StencilRoughDefinition}
\end{equation}
and $\bdelta$ may be treated as a parameter.
\revised{We define} $\Phi_\bdelta(\tilde{\bmx}_i) = \Phi(\tilde{\bmx}_i + \bdelta,\tilde{\bmx}_i) = \Phi(\tilde{\bmx}_j,\tilde{\bmx}_i)$.

For a fixed number of translations $\bdelta$, these so-called \emph{stencil functions}, $\Phi_\bdelta(\cdot)$, can be identified with the majority of entries in many IGA stiffness matrices.
In many circumstance, each $\Phi_\bdelta(\cdot)$ is smooth and may, therefore, be interpolated after only being evaluated at small number of points in the parametric domain $\tilde{\bmx}_i\in {\hat{\Omega}}$.
After denoting the interpolants \textemdash{} i.e., the \emph{surrogate stencil functions} \textemdash{} by $\tilde{\Phi}_\bdelta(\cdot)$, we simply define
\begin{equation}
  \tilde{\sfA}_{ij}
  =
  \tilde{\Phi}_\bdelta(\tilde{\bmx}_i)
  \,.
\label{eq:SurrogateRoughDefinition}
\end{equation}

\begin{remark}
In some cases, the stencil functions $\Phi_\bdelta$ are themselves polynomials (see, e.g., \Cref{prop:PolynomialReproduction,lem:MassMatrixReproduction,lem:stokesreproduction}).
Therefore, if polynomial interpolation of sufficiently high order is used, the true stiffness matrix be generated exactly \textemdash{} i.e., $\tilde{\Phi}_\bdelta = \Phi_\bdelta$ and thus $\tilde{\sfA} = \sfA$, up to round-off error \textemdash{} in significantly less time than with a traditional assembly algorithm.
Otherwise, in many scenarios, a sufficiently accurate approximation of the stiffness matrix $\tilde{\sfA} \approx \sfA$ will be generated.
\end{remark}

\subsection{B-spline basis functions} %
\label{sub:exploiting_basis_structure2}

Fix $V= H^1(\Omega)$, $\bvarphi:\hat{\Omega}\to\Omega$, $\hat{V}_h = \spann\{\hat{B}_i\}$, and, accordingly, $V_h = \spann\{B_i\}$, where each $B_i = \hat{B}_i\circ \bvarphi^{-1}$.
Consider the bilinear form $a(u,v) = \int_\Omega \nabla u \cdot \nabla v \dd x$.
It is easy to verify that $a(\cdot,\cdot)$ pulls back to
\begin{equation}
    \hat{a}(\hat{w},\hat{v})
    =
    \int_{\hat{\Omega}} \hat{\nabla} \hat{w}(\hat{\bmx})^\top {K}(\hat{\bmx})\, \hat{\nabla} \hat{v}(\hat{\bmx}) \dd \hat{\bmx}
    \,,
    \qquad
    \text{where}
    \qquad
    {K} = \frac{D\bvarphi^{-1}\sspace\sspace D\bvarphi^{-\top}}{|\det{\left(D\bvarphi^{-1}\right)}|}
    \,,
\label{eq:BilinearFormPoissonTensor}
\end{equation}
with arguments $\hat{w},\hat{v}\in \hat{V} = H^1(\hat{\Omega})$.

Recall \Cref{sub:nurbs}.
Assume that $\{\hat{B}_i\}$ is generated by an open uniform knot vector $\Xi$ with $p>0$ fixed.
Obviously, $\sfA_{ij} = \hat{a}(\hat{B}_j,\hat{B}_i)$.
In the cardinal B-spline setting, $\hat{B}_i(\hat{\bmx}) = \hat{B}(\hat{\bmx}-\tilde{\bmx}_i)$ and $\hat{B}_j(\hat{\bmx}) = \hat{B}(\hat{\bmx}-\tilde{\bmx}_j)$.
Therefore, by a simple change of variables,
\begin{equation}
  \begin{aligned}
    \sfA_{ij}
    =
    \int_{\hat{\Omega}} \hat{\nabla} \hat{B}(\hat{\bmx}-\tilde{\bmx}_i)^\top {K}(\hat{\bmx})\, \hat{\nabla} \hat{B}(\hat{\bmx}-\tilde{\bmx}_j) \dd \hat{\bmx}
    &=
    \int_{\hat{\omega}_{\bdelta}} \hat{\nabla} \hat{B}(\hat{\bmy})^\top {K}(\tilde{\bmx}_i + \hat{\bmy})\, \hat{\nabla} \hat{B}_{\bdelta}(\hat{\bmy}) \dd \hat{\bmy}
    \,,
  \end{aligned}
\label{eq:ChangeOfVars}
\end{equation}
where $\bdelta = \tilde{\bmx}_j - \tilde{\bmx}_i$, $\hat{B}_\bdelta(\hat{\bmy}) = \hat{B}(\hat{\bmy}-\bdelta)$, and $\hat{\omega}_\bdelta = \supp(\hat{B})\cap\supp(\hat{B}_{\bdelta})$.

There is a natural correspondence between the density of the matrix $\sfA$ and the set of translations $\bdelta$ such that $\hat{\omega}_\bdelta \neq \emptyset$.
Consequently, the cardinality of the set of relevant translations, $\scD = \{ \bdelta = \tilde{\bmx}_j - \tilde{\bmx}_i \,:\, \hat{\omega}_\bdelta \neq \emptyset\}$, is fixed for all sufficiently large $m$.
Namely, $|\scD| = (2p+1)^n$.
Now, for each $\bdelta \in \scD$, we may define the stencil function
\begin{equation}
  \Phi_\bdelta(\tilde{\bmx})
  =
  \int_{\hat{\omega}_{\bdelta}} \hat{\nabla} \hat{B}(\hat{\bmy})^\top {K}(\tilde{\bmx} + \hat{\bmy})\, \hat{\nabla} \hat{B}_\bdelta(\hat{\bmy}) \dd \hat{\bmy}
  \,.
\label{eq:StencilFunction}
\end{equation}

\begin{figure}
  \centering
  \begin{tikzpicture}
  \begin{scope}[spy using outlines={circle,black,magnification=4,size=4cm, connect spies}]
    \node {\pgfimage[interpolate=true,width=0.7\textwidth]{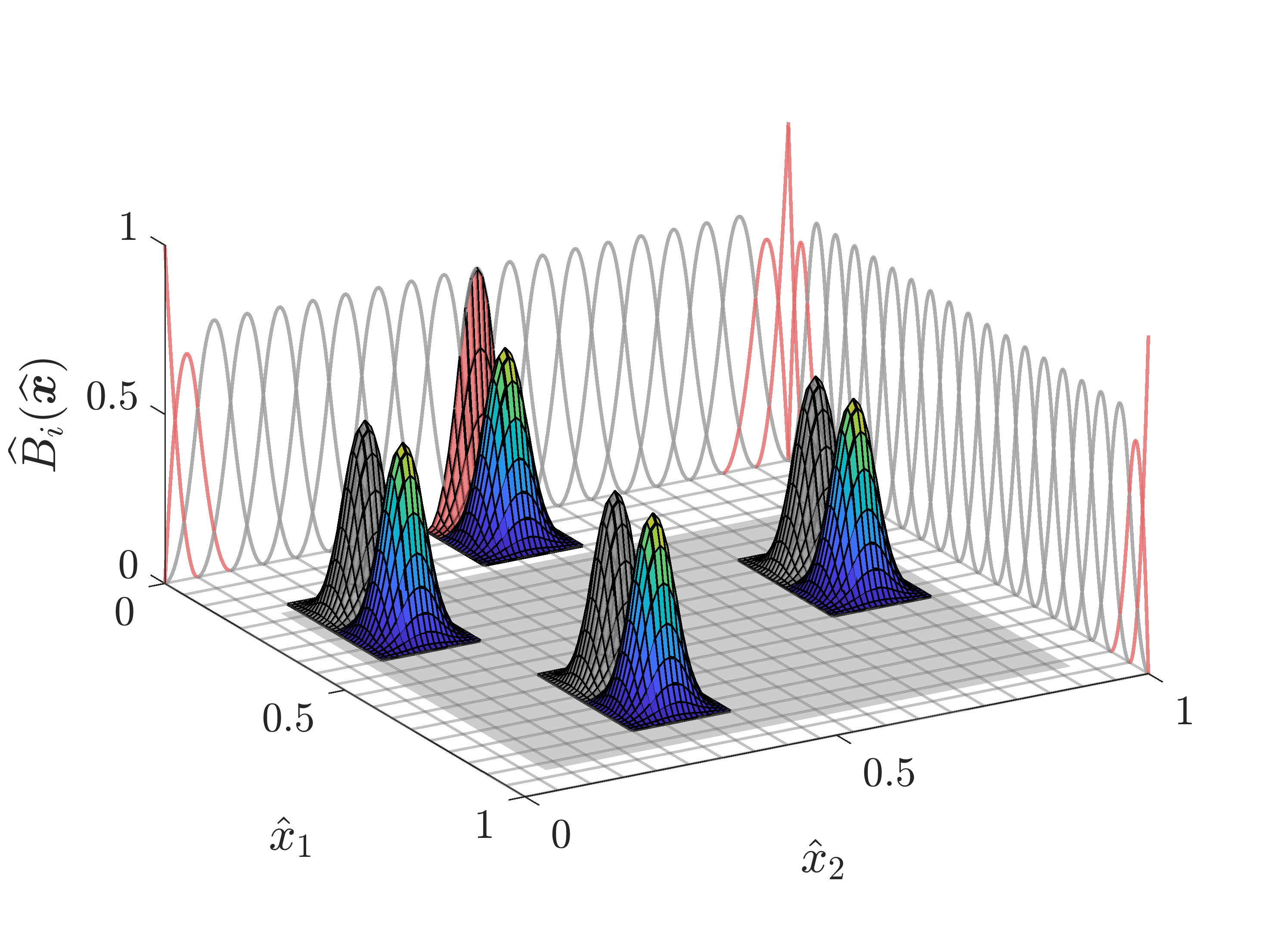}};
    \spy [black, size=5cm, magnification=2] on (0.05,-1.05) in node [left] at (9,1.25);
  \end{scope}

  \draw[left color=red, right color=red!60, opacity=0.8, drop shadow={ashadow, color=red!60!black}, rotate=-30, shift={(5,3.65)}, scale=1.60] \arrow;
  \end{tikzpicture}
\caption[Representative B-spline basis functions]{\label{fig:BasisFunctions_delta-20} Second-order cardinal B-spline basis functions $\hat{B}_i(\hat{\bmx}) = \hat{B}(\hat{\bmx}-\tilde{\bmx}_i)$ (color gradient).
After translation by $\bdelta = \begin{psmallmatrix}-2h\\0\end{psmallmatrix}$ (represented by the red arrow), most of these basis functions are equal to another cardinal basis function $\hat{B}_j(\hat{\bmx}) = \hat{B}_i(\hat{\bmx}-\bdelta)$ (gray).
For every such function, $\tilde{\bmx}_i\in\tilde{\Omega}_\bdelta$ (this subset is shaded in gray on the mesh).
Clearly, this property does not hold for every cardinal basis function, as indicated by the cardinal B-spline neighboring the boundary and the nearby (non-cardinal) basis function (red).
} 
\end{figure}

Let $\mathrm{conv}(\tilde{\bbX})$ denote the convex hull of $\tilde{\bbX}$.
Such functions are defined at any point $\tilde{\bmx} \in \mathrm{conv}(\tilde{\bbX})$ where $\tilde{\bmx} + \bdelta \in \mathrm{conv}(\tilde{\bbX})$.
Ultimately, this means that the domain of $\Phi_\bdelta$, which we will denote \revised{by} $\tilde{\Omega}_\bdelta$, depends on $\bdelta$ (see, e.g., \Cref{fig:BasisFunctions_delta-20}).
Clearly, we always have $\bm{0}\in \scD$.
The reader may easily verify that $\tilde{\Omega}_{\bm{0}} = \mathrm{conv}(\tilde{\bbX}) = \bigcup_{\bdelta\in\scD} {\tilde{\Omega}_\bdelta}$ and $\tilde{\Omega}_\bdelta + \bdelta \subset \tilde{\Omega}_{\bm{0}}$, for each $\bdelta \in \scD$.

\subsection{NURBS basis functions} %
\label{sub:basis_structure_nurbs}

The principal difference between the treatment of a NURBS basis $\{\hat{N}_i\}$ and the related B-spline basis $\{\hat{B}_i\}$ is that a NURBS basis cannot be assumed to have the translation invariance property which leads directly to~\cref{eq:TemporaryStencilFunction}.
Fortunately, as we now demonstrate, this property is not entirely necessary to define a useful stencil function.

Define $W(\hat{\bmx}) = \sum_j w_j\hat{B}_j(\hat{\bmx})$, where $\{w_j\}$ are the weight parameters appearing in~\cref{eq:NURBS}.
It is known that $W(\hat{\bmx})$ is unchanged under mesh refinements.
Therefore, employing a similar change of variables argument as used in~\cref{eq:ChangeOfVars}, it holds that
\begin{equation}
  \begin{aligned}
    \sfA_{ij}
    =
    \revised{\hat{a}}(\hat{N}_j,\hat{N}_i)
    &=
    w_iw_j
    \int_{\hat{\omega}_{\bdelta}} \hat{\nabla} \bigg(\frac{\hat{B}(\hat{\bmy})}{W(\tilde{\bmx}_i + \hat{\bmy})}\bigg)^\top {K}(\tilde{\bmx}_i + \hat{\bmy})\, \hat{\nabla} \bigg(\frac{\hat{B}_\bdelta(\hat{\bmy})}{W(\tilde{\bmx}_i + \hat{\bmy})}\bigg) \dd \hat{\bmy}
    \,,
  \end{aligned}
\label{eq:ChangeOfVarsNURBS}
\end{equation}
where, as in~\cref{eq:ChangeOfVars}, $\bdelta = \tilde{\bmx}_j - \tilde{\bmx}_i$.
At this point, it may be natural to divide by $w_iw_j$ and define the stencil function $\Phi_\bdelta(\tilde{\bmx})$ from the resulting expression on the right-hand side of~\cref{eq:ChangeOfVarsNURBS}.
Instead, we pause to consider the regularity of $W(\hat{\bmx})$.

Recall~\cref{eq:SurrogateRoughDefinition}.
Since each $\hat{B}_i(\hat{\bmx})$ is only piecewise polynomial, it is clear that, in general, $W\not\in C^{q}(\hat{\Omega})$, for any $q\geq p$.
This fact could significantly limit the accuracy of an interpolant $\tilde{\Phi}_\bdelta = \Pi\Phi_\bdelta$ which we may wish to construct.
Therefore, we restrict our attention to $W\in C^{q}(\hat{\Omega})$, where $q\geq p$.
It turns out that this set of functions, $\spann\{\hat{B}_i\} \cap C^{q}(\hat{\Omega})$, is equal to the polynomial space $\mcQ_{p}(\hat{\Omega})$.
Moreover, restricting to the subset $\tilde{\Omega}_{\bm{0}}$, each weight parameter can be expressed \revised{as} $w_i = w(\tilde{\bmx}_i)$, where $w(\hat{\bmx})$ is a polynomial in $\mcQ_p(\tilde{\Omega}_{\bm{0}})$.
(See \Cref{app:marsden_s_identity} for details.)
Therefore, for any $W\in \mcQ_p(\hat{\Omega})$, we may define
\begin{equation}
  \Phi_\bdelta(\tilde{\bmx})
  =
  w(\tilde{\bmx})w(\tilde{\bmx}+\bdelta)
  \int_{\hat{\omega}_{\bdelta}} \hat{\nabla} \bigg(\frac{\hat{B}(\hat{\bmy})}{W(\tilde{\bmx} + \hat{\bmy})}\bigg)^\top {K}(\tilde{\bmx} + \hat{\bmy})\, \hat{\nabla} \bigg(\frac{\hat{B}_\bdelta(\hat{\bmy})}{W(\tilde{\bmx} + \hat{\bmy})}\bigg) \dd \hat{\bmy}
  ,
\label{eq:StencilFunctionNURBS}
\end{equation}
for each $\tilde{\bmx}\in\tilde{\Omega}_\bdelta$ and $\bdelta \in \scD$.
Clearly, $\sfA_{ij} = \Phi_\bdelta(\tilde{\bmx}_i)$ for each corresponding $\bdelta \in \scD$.
An illustration of a stencil function coming from an IGA discretization with a NURBS basis is presented in \Cref{fig:BasisFunctions_delta00}.

\begin{remark}
  Notably, when $W=1$, then $w=1$ also.
  Therefore,~\cref{eq:StencilFunctionNURBS} is consistent with~\cref{eq:StencilFunction}.
  \changed{Obviously, in practice, neither of these expressions needs to be used in order to evaluate $\Phi_\bdelta(\cdot)$ at any point $\tilde{\bmx}_i$.
  Indeed, since $\Phi_\bdelta(\tilde{\bmx}_i) = \sfA_{ij}$, any existing IGA code already has a mechanism to compute $\Phi_\bdelta(\tilde{\bmx}_i)$ using quadrature (cf. \Cref{app:computing_surrogates_with_existing_iga_codes}).
  Nevertheless, these expressions are important for analysis.}
\end{remark}

\begin{figure}
  \centering
  \includegraphics[trim=2.5cm 0cm 2.5cm 0cm,clip=true,height=6cm]{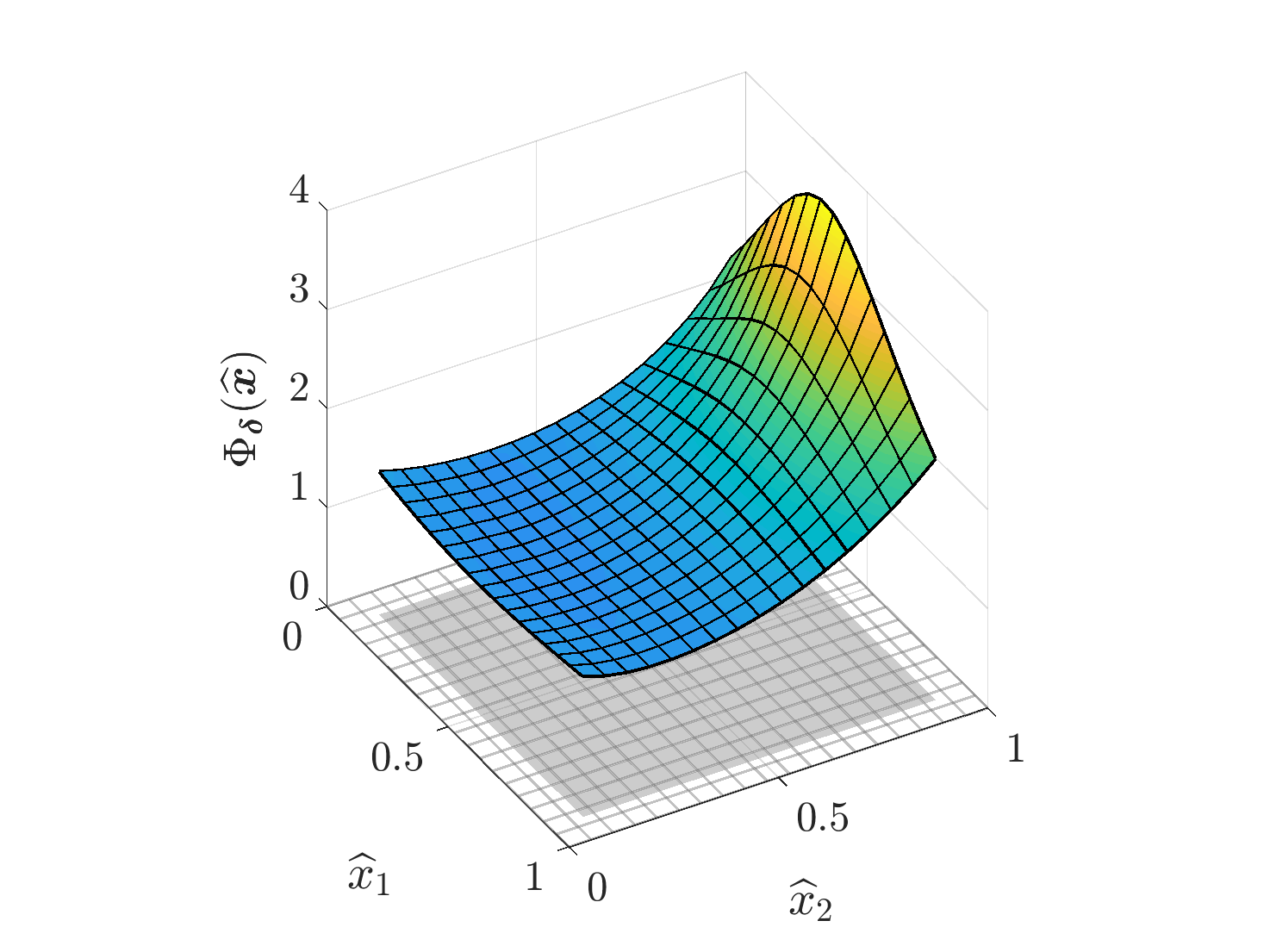}
  \caption{\label{fig:BasisFunctions_delta00}
  The graph of a stencil function $\Phi_\bdelta:\tilde{\Omega}_\bdelta\to\R$ defined by~\cref{eq:StencilFunctionNURBS}.
  The corresponding physical geometry $\Omega$ is depicted in \Cref{fig:PoissonGeometries2D}.
  In this case, $\bdelta = \bm{0}\cdot h$ and $p=2$.
  The domain $\tilde{\Omega}_{\bm{0}}$ is shaded gray.
  For further details, see \Cref{sub:poisson_numerical_experiments}.}
\end{figure}

\subsection{Symmetric bilinear forms} %
\label{sub:symmetric_bilinear_forms}

When $a(w,v) = a(v,w)$, for all $w,v\in V$, a translational symmetry is induced on the set of corresponding stencil functions.
Indeed, it is simple to see that $\Phi(\tilde{\bmx}_j,\tilde{\bmx}_i) = \Phi(\tilde{\bmx}_i,\tilde{\bmx}_j)$ and, therefore,
\begin{equation}
  \Phi_\bdelta(\tilde{\bmx}_i)
  =
  \Phi(\tilde{\bmx}_i+\bdelta,\tilde{\bmx}_i)
  =
  \Phi(\tilde{\bmx}_j-\bdelta,\tilde{\bmx}_j)
  =
  \Phi_{-\bdelta}(\tilde{\bmx}_j)
  \,.
\end{equation}
A similar conclusion can be drawn in the NURBS scenario above.
\Cref{fig:StencilFunctions} presents a visual comparison of two stencil functions, $\Phi_\bdelta$ and $\Phi_{-\bdelta}$, generated by an isogeometric NURBS basis and the corresponding symmetric bilinear form~\cref{eq:BilinearFormPoissonTensor}.

\begin{figure}
\captionsetup[subfigure]{font=footnotesize,labelformat=empty}
  \centering
  \begin{subfigure}[c]{0.4\textwidth}
      \includegraphics[trim=2cm 0cm 2cm 0cm,clip=true,width=\textwidth]{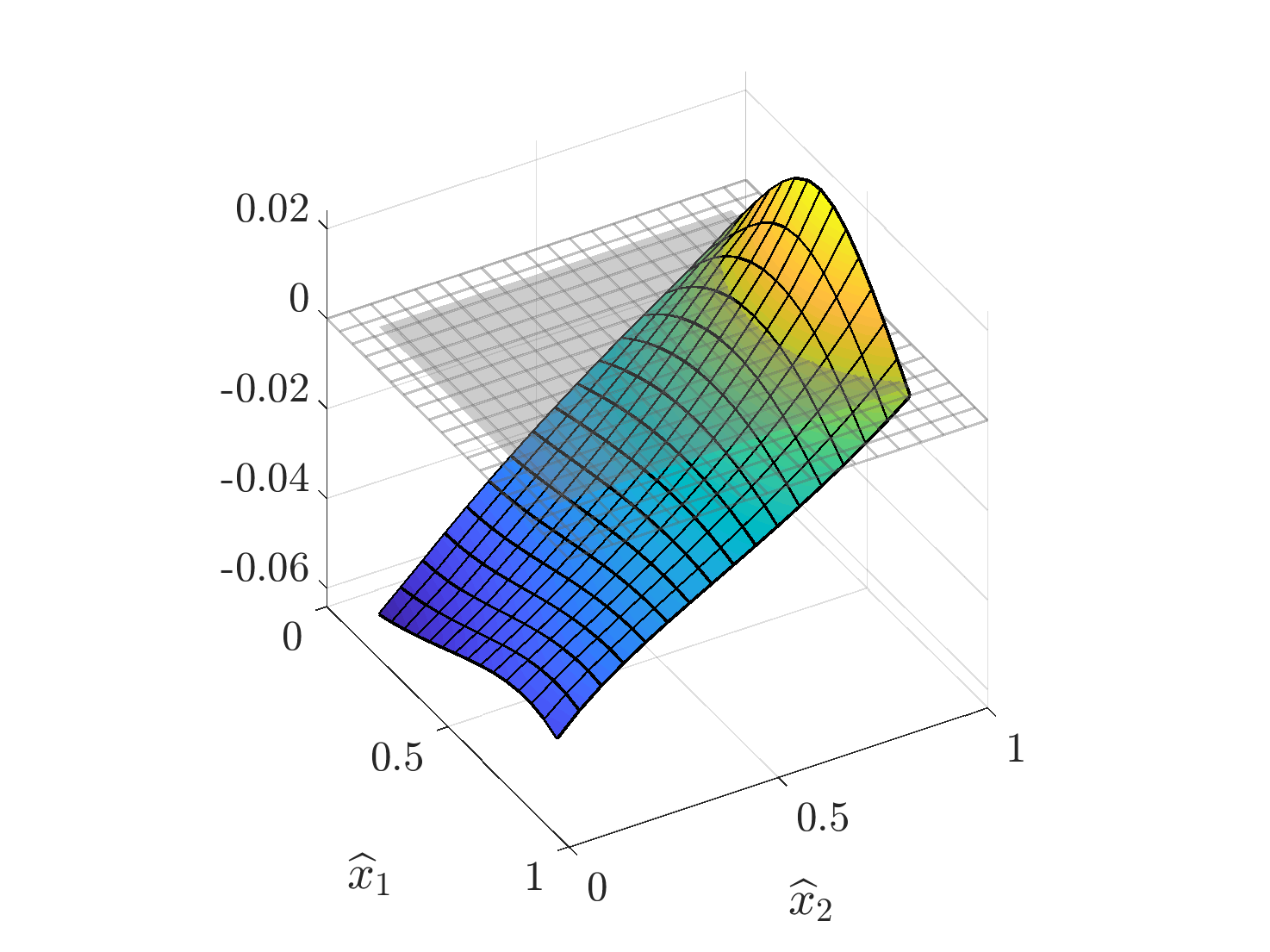}
      \label{fig:delta-20_standard}
  \end{subfigure}
  ~
  \begin{subfigure}[c]{0.4\textwidth}
      \includegraphics[trim=2cm 0cm 2cm 0cm,clip=true,width=\textwidth]{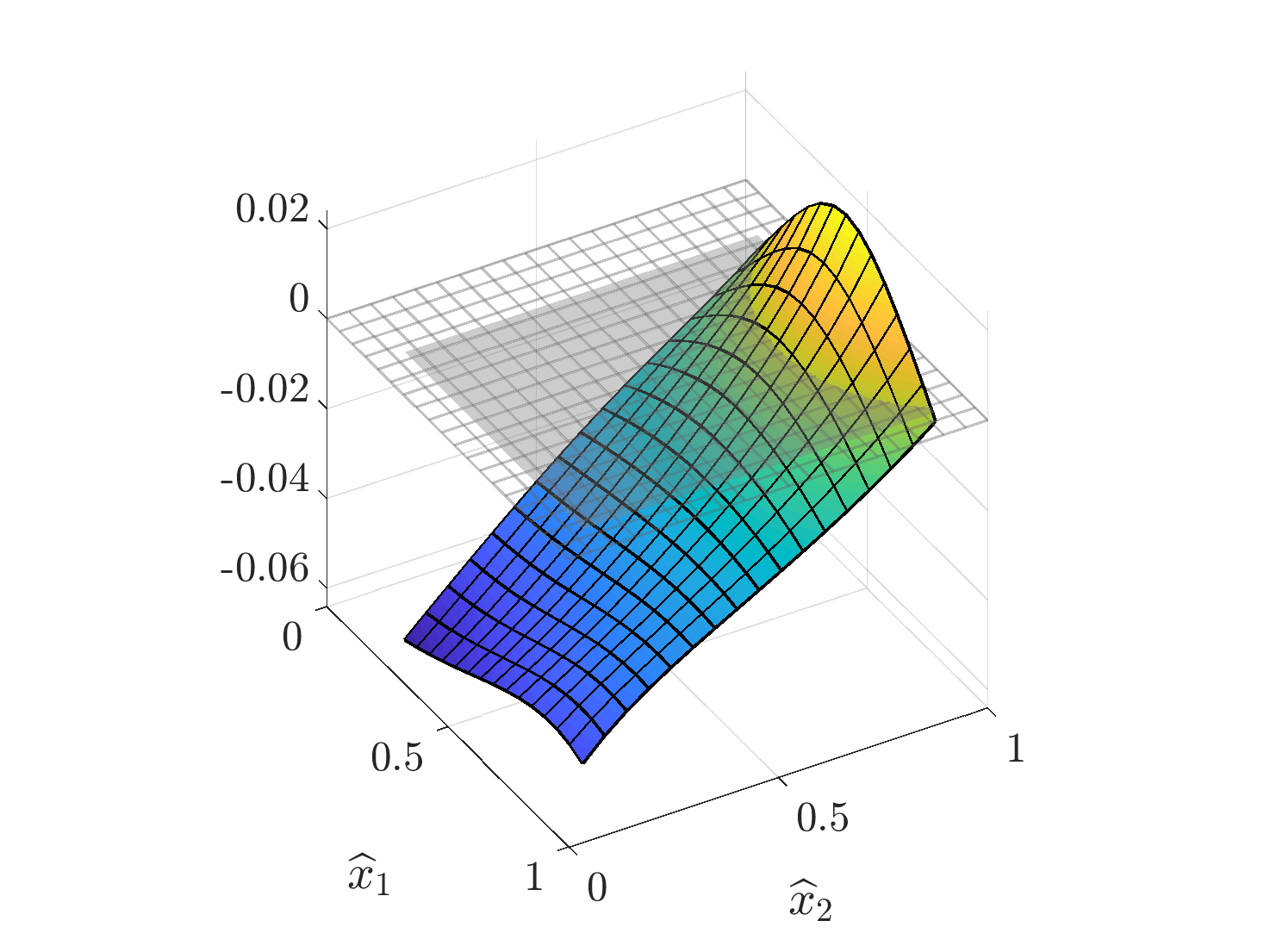}
      \label{fig:delta01_standard}
  \end{subfigure}
\caption[Graphs of stencil functions]{\label{fig:StencilFunctions}
Left and right, respectively: Graphs of stencil functions $\Phi_\bdelta(\tilde{\bmx})$ for $\bdelta = \protect{\begin{psmallmatrix}2h\\0\end{psmallmatrix}}$ and $\bdelta = \protect{\begin{psmallmatrix}-2h\\0\end{psmallmatrix}}$.
The associated domains $\tilde{\Omega}_\bdelta$ are shaded in gray (cf. \Cref{fig:BasisFunctions_delta-20}). Notice the translational symmetry $\Phi_\bdelta(\tilde{\bmx}) = \Phi_{-\bdelta}(\tilde{\bmx}+\bdelta)$.}
\end{figure}

\subsection{The multi-patch setting} %
\label{sub:the_multi_patch_setting}

In the multi-patch setting, the physical domain $\Omega$ is partitioned into a finite number of disjoint subdomains $\overline{\Omega} = \bigcup_{k=1}^{L} \overline{\Omega^{(k)}}$.
We may assume that each such domain or \emph{patch} $\Omega^{(k)}$, as they are usually called, can be identified with a common parametric domain $\hat{\Omega}$, through a unique NURBS mapping $\Omega^{(k)} = \bvarphi^{(k)}(\hat{\Omega})$.
In this case, one may define a separate set of stencil functions $\Phi_\bdelta^{(k)}(\tilde{\bmx}):\tilde{\Omega}_\bdelta \to \R$, for each index $k$.
The forthcoming analysis is immediately applicable to this setting, but treating it outright would require unnecessarily complicated notation.
We also did not consider this setting in our numerical experiments.

\subsection{Edge-based and face-based stencil functions} %
\label{sub:edge_and_face_based_stencil_functions}
  It is also possible to define stencil functions, in addition to those above, which exploit lower-dimensional translational symmetries.
  For instance, many functions in the basis $\{\hat{N}_i\}$ are only equivalent under translations parallel to a given edge in the parametric domain $\hat{\Omega}$.
  Likewise, a new family of stencil functions could be constructed for each edge or face in a patch $\Omega^{(k)}$.
  We do not pursue this observation here.

\section{Surrogate matrices: Interpolation of stencil functions} %
\label{sec:surrogate_matrices_interpolation_of_stencil_functions}

Recall~\cref{eq:StencilRoughDefinition,eq:SurrogateRoughDefinition}.
A surrogate matrix $\tilde{\sfA}\approx \sfA$ will be useful to us if: (1) each stencil function $\Phi_\bdelta$ has high regularity and, therefore, the high-order approximation $\Pi{\Phi}_\bdelta \approx \Phi_\bdelta$ will have high-order accuracy; and (2) each $\tilde{\Phi}_\bdelta = \Pi{\Phi}_\bdelta$ can be computed and evaluated fast.
In this section, we focus on the former of these two requirements.
Particulars on our implementation are withheld until \Cref{app:computing_surrogates_with_existing_iga_codes}.

As a simplifying accommodation, we perform all of our analysis on the largest subset of $\hat{\Omega}$ where every stencil function is defined.
That is, $\tilde{\Omega} = \bigcap_{\bdelta\in\scD} {\tilde{\Omega}_\bdelta}$.
A simple computation shows that
\begin{equation}
  \tilde{\Omega}
  =
  \bigg[\frac{3p+1}{2(m-p)},1-\frac{3p+1}{2(m-p)}\bigg]^n
  \subsetneq
  \hat{\Omega}
  \,.
\end{equation}
All of the results in this section can be reformulated with $\tilde{\Omega}$ replaced by $\tilde{\Omega}_\bdelta$.
However, this generalization also requires a different operator $\Pi$ for each $\bdelta \in \scD$; see, e.g., \cite{drzisga2018surrogate}.

\subsection{B-spline interpolation} %
\label{sub:b_spline_interpolation}

Let $\{\tilde{B}_j\}$ be a degree $q\geq 0$ multivariate B-spline basis on $\tilde{\Omega}$ with the quasi-uniform knot vector $\tilde{\bm{\Xi}} = \tilde{\Xi}_1\times \cdots\times \tilde{\Xi}_n$ and define $S_q(\tilde{\bm{\Xi}}) = \spann\{\tilde{B}_j\}$.
We will refer to each $\tilde{\bxi}_{\bmj} \in \tilde{\bm{\Xi}}$, $\bmj = (j_1,\cdots,j_n)$, as a \emph{sampling point} and define a new length scale parameter, hereby referred to as the \emph{sampling length}, $H = \max_{|\bmj| = 1,\bmi}\big\{\|\tilde{\bxi}_{\bmi+\bmj}-\tilde{\bxi}_{\bmi}\|_{\mathrm{max}} \,:\, \tilde{\bxi}_{\bmi+\bmj} \in \tilde{\bm{\Xi}}\big\}$.

In isogeometric analysis, the geometry determines the basis used in PDE discretization.
When constructing the surrogate stencil functions $\tilde{\Phi}_\bdelta$, we are primarily interested in using a basis of higher order $q$ than the underlying spatial discretization $p$.
Generally, such a choice $q>p$ is desirable because it will allow us to guarantee that the discretization error in a standard IGA method will dominate the error actually attributed to using a surrogate (cf. \Cref{subs:constant_sampling_length}).

In this paper, we only consider constructing B-spline interpolants $\tilde{\Phi}_\bdelta = \Pi_{H}\Phi_\bdelta$, where $\Pi_{H}$ is a stable local interpolation operator onto the space $S_q(\tilde{\bm{\Xi}})$.
Various global interpolants could also be considered, as well as sparse grid interpolants \cite{bungartz2004sparse} and least-squares projections (cf. \cite{drzisga2018surrogate}).
We see no benefit in using a NURBS basis to approximate $\Phi_{\bdelta}$, even when a NURBS mapping $\bvarphi:\hat{\Omega}\to\Omega$ defines the physical domain.
The following lemma follows directly from \cite[Theorem~4.2]{dahmen1980multidimensional}.
\begin{lemma}%
\label{lem:SplineBAE}
  For every bounded projection $\Pi_{H}\colon C^0({\tilde{\Omega}}) \to S_q(\tilde{\bm{\Xi}})$, with $\|\Pi_{H}\|< C_0$ for some $H$-independent constant $C_0$, it holds that
  \begin{equation}
    \|f-\Pi_{H} f\|_{L^\infty(\tilde{\Omega})}
    \leq
    C_1
    H^{q+1}
    [f]_{W^{q+1,\infty}(\tilde{\Omega})}
    \,,
    \quad
    \text{for all }
    f\in W^{q+1,\infty}(\tilde{\Omega})
    \,,
  \end{equation}
  where $C_1$ is a constant depending only on $q$, $\tilde{\Omega}$, and $\|\Pi_{H}\|$.
\end{lemma}

\revised{\Cref{fig:SurrogateStencilFunctions} presents graphs of stencil functions, $\Phi_\bdelta(\tilde{\bmx})$ and their surrogates $\tilde{\Phi}_\bdelta(\tilde{\bmx})$ obtained by B-spline interpolation for various $\bdelta$. The stencil functions are generated by an isogeometric NURBS basis and the symmetric bilinear form~\cref{eq:BilinearFormPoissonTensor}.}
\begin{remark}
\label{rem:SamplingPoints}
  The norm $\|\Pi_{H}\|$ can be greatly influenced by the distribution of the sampling points $\tilde{\bxi}_{\bmj}\in\tilde{\bm{\Xi}}$.
  In all of our experiments, we kept $\tilde{\bm{\Xi}}\subset \tilde{\Omega}\cap\tilde{\bbX}$.
  This convenient choice delivered good results.
  When we wish to underscore the convention $\tilde{\bm{\Xi}}\subset \tilde{\Omega}\cap\tilde{\bbX}$, we will denote the sampling points in $ \tilde{\bm{\Xi}}$ by $\tilde{\bmx}_i^\mathrm{s}$.
  Moreover, from now on, dependence on the subset $\tilde{\Omega}$ will not be stated since $\tilde{\Omega} \to \hat{\Omega}$ as $h\to 0$ and the parametric domain $\hat{\Omega} = (0,1)^n$ is always fixed.
\end{remark}

\subsection{Regularity of the stencil functions} %
\label{sub:regularityPoisson}

Define $\tilde{\Phi}_\bdelta = \Pi_{H}\Phi_\bdelta$, where $\Pi_{H}$ is any projection satisfying the assumptions of \Cref{lem:SplineBAE}.
In this subsection, we present an essential theorem on the error in the class of surrogate stencil functions defined in~\cref{eq:StencilFunctionNURBS}.
In order to expedite our presentation, we only prove \Cref{thm:RegularityOfStencilFunctions} here under an assumption which directly relates to the B-spline basis scenario~\cref{eq:StencilFunction}.
The general proof is given in \Cref{app:marsden_s_identity}.

\begin{theorem}
\label{thm:RegularityOfStencilFunctions}
  Let $\Phi_\bdelta:\tilde{\Omega}\to\R$ be defined by~\cref{eq:StencilFunctionNURBS} and assume that $\bvarphi:\hat{\Omega}\to\Omega$ induces a coefficient tensor ${K} \in \big[W^{q+1,\infty}(\hat{\Omega})\big]^{n\times n}$.
  If $W\in \mcQ_p(\hat{\Omega})$, then there exists a constant $C_2$, depending only on $p$, $q$, $\|\Pi_{H}\|$, and \changed{$\bvarphi$}, such that
  \begin{equation}
    \big\|\Phi_\bdelta-\tilde{\Phi}_\bdelta\big\|_{L^\infty(\tilde{\Omega})}
    \leq
    C_2
    \sspace
    h^{n-2} H^{q+1}
    \quad
    \text{for each }
    \bdelta\in\scD
    \,.
  \label{eq:NURBSStencilFunctionError}
  \end{equation}
  Moreover, if $W(\hat{\bmx})=1$, then $C_2 \leq C\, C_1\,[K]_{W^{q+1,\infty}(\hat{\Omega})}$, for some $C$ depending only on $p$.
\end{theorem}

\begin{proof}[Proof of \Cref{thm:RegularityOfStencilFunctions}, under the assumption $W(\hat{\bmx})=1$]
  For every multi-index \revised{$\balpha = (\alpha_1, \ldots, \alpha_n)$, with each $0 \leq \alpha_i \leq q+1$}, it holds that
  \begin{equation}
    D^{\balpha}\Phi_\bdelta(\tilde{\bmx})
    =
    \int_{\hat{\omega}_{\bdelta}} \hat{\nabla} \hat{B}(\hat{\bmy})^\top D^{\balpha} {K}(\tilde{\bmx} + \hat{\bmy})\, \hat{\nabla} \hat{B}_\bdelta(\hat{\bmy}) \dd \hat{\bmy}
    .
  \label{eq:DerivativeOfStencilFunction}
  \end{equation}
  Therefore,
  $
    \|D^{\balpha}\Phi_\bdelta\|_{L^\infty(\tilde{\Omega})}
    \leq
    \|D^{\balpha} {K}\|_{L^\infty(\hat{\Omega})} \|\hat{\nabla} \hat{B}\cdot \hat{\nabla} \hat{B}_\bdelta\|_{L^1(\hat{\omega}_\bdelta)}
    \leq
    C
    h^{n-2} \|D^{\balpha} {K}\|_{L^\infty(\hat{\Omega})}
  $
  \revised{and, moreover, $\Phi_\bdelta \in W^{q+1,\infty}(\tilde{\Omega})$.
  The result follows from \cref{lem:SplineBAE} using $f = \Phi_\bdelta$.}
\end{proof}

\begin{remark}
\label{rem:PolynomialReproduction}
  Observe that $[K]_{W^{q+1,\infty}(\hat{\Omega})} = 0$ iff $K\in [\mcQ_{q}(\hat{\Omega})]^{n\times n}$.
  Generally, due to the definition of $K$ appearing in~\cref{eq:BilinearFormPoissonTensor}, this assumption can only be expected to be satisfied when the geometry map $\bvarphi$ is affine.
  Nevertheless, if $W =  1$ as well, then \Cref{thm:RegularityOfStencilFunctions} would imply that $\tilde{\Phi}_\bdelta = \Phi_\bdelta \in \mcQ_{q}(\tilde{\Omega})$.
  This a useful reproduction property which can help to verify the implementation.
  It also manifests in stencil functions defined from more complicated bilinear forms than~\cref{eq:BilinearFormPoissonTensor}.
  The reproduction of stencil functions is described in a general scenario in \Cref{sub:polynomial_reproduction}.
\end{remark}

\begin{figure}
  \centering
    \begin{subfigure}[c]{0.28\textwidth}
        \includegraphics[trim=2cm 0cm 2cm 0cm,clip=true,width=\textwidth]{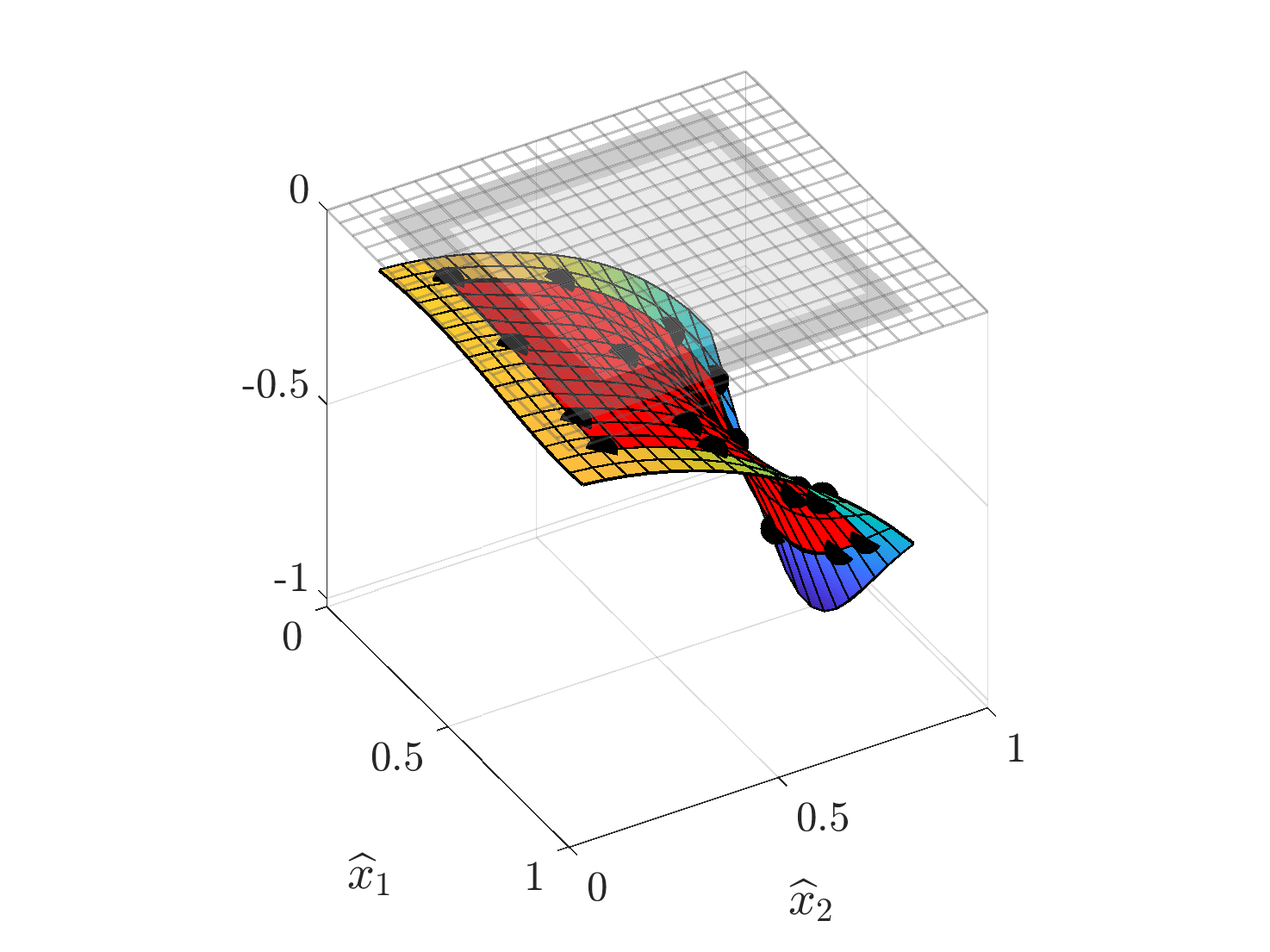}
        \label{fig:delta01}
    \end{subfigure}
    \quad
    \begin{subfigure}[c]{0.28\textwidth}
        \includegraphics[trim=2cm 0cm 2cm 0cm,clip=true,width=\textwidth]{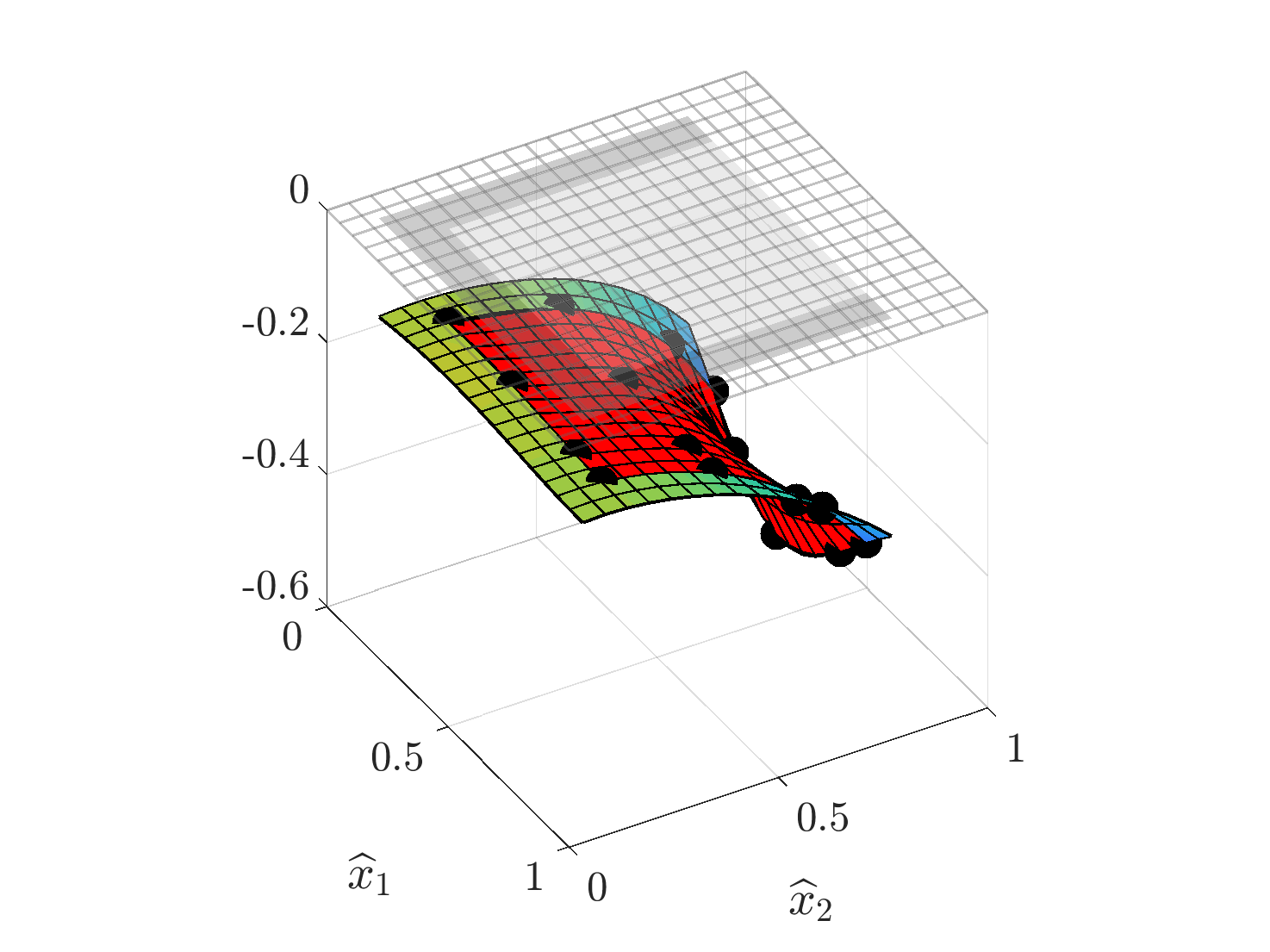}
        \label{fig:delta02}
    \end{subfigure}
    \quad
    \begin{subfigure}[c]{0.28\textwidth}
        \includegraphics[trim=2cm 0cm 2cm 0cm,clip=true,width=\textwidth]{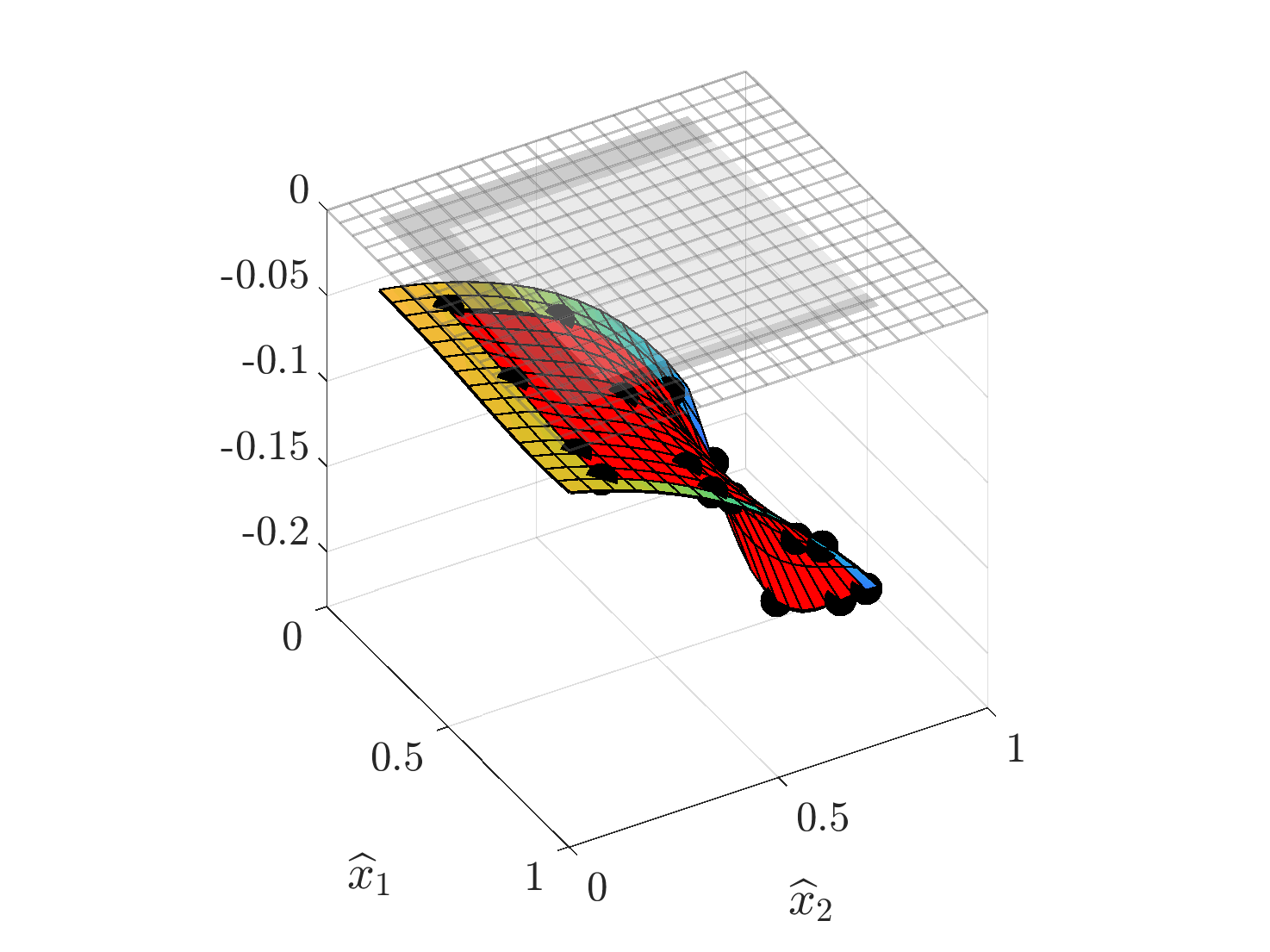}
        \label{fig:delta12}
    \end{subfigure}

    \begin{subfigure}[c]{0.28\textwidth}
        \includegraphics[trim=2cm 0cm 2cm 0cm,clip=true,width=\textwidth]{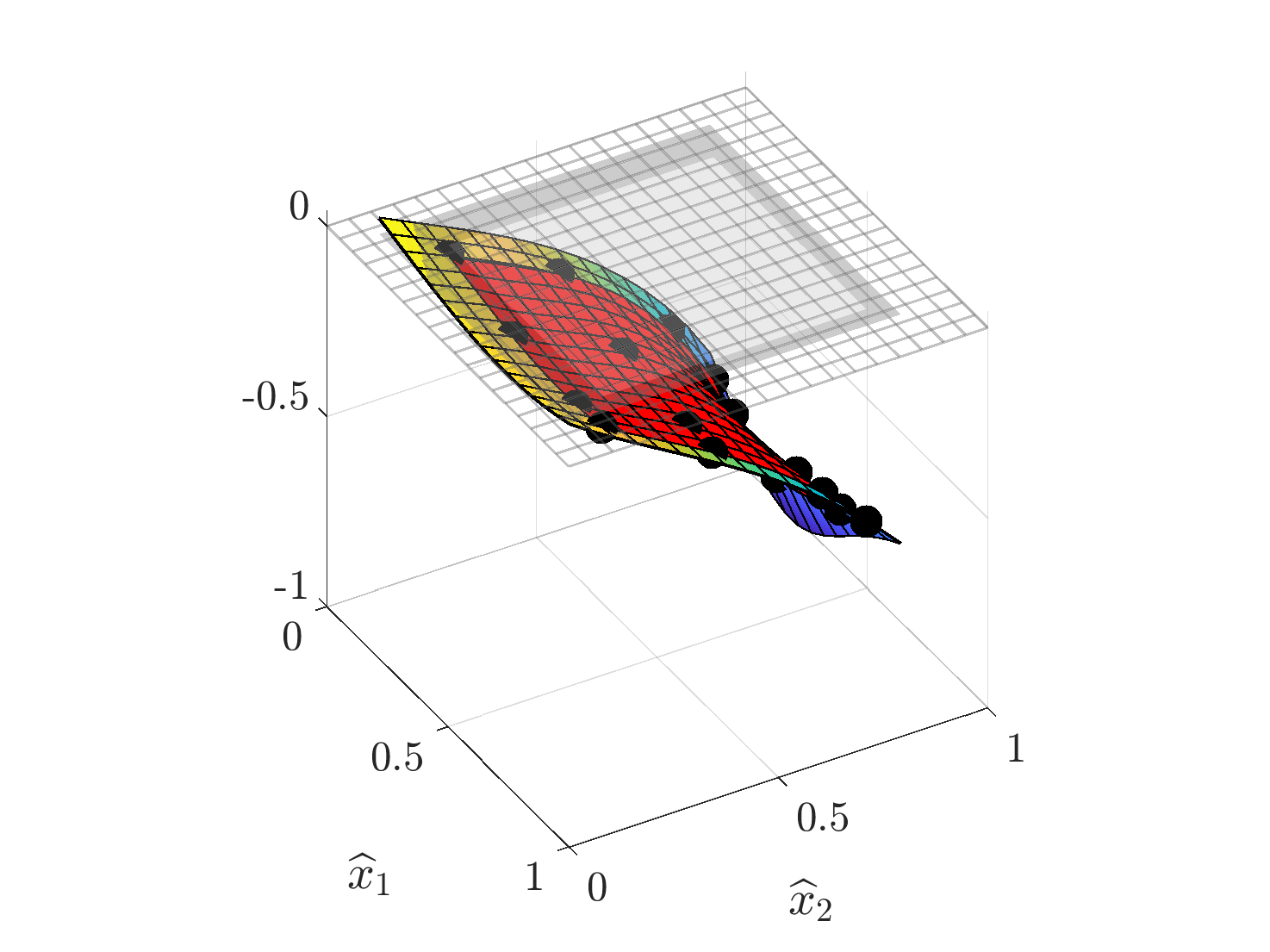}
        \label{fig:delta11}
    \end{subfigure}
    \quad
    \begin{subfigure}[c]{0.28\textwidth}
        \includegraphics[trim=2cm 0cm 2cm 0cm,clip=true,width=\textwidth]{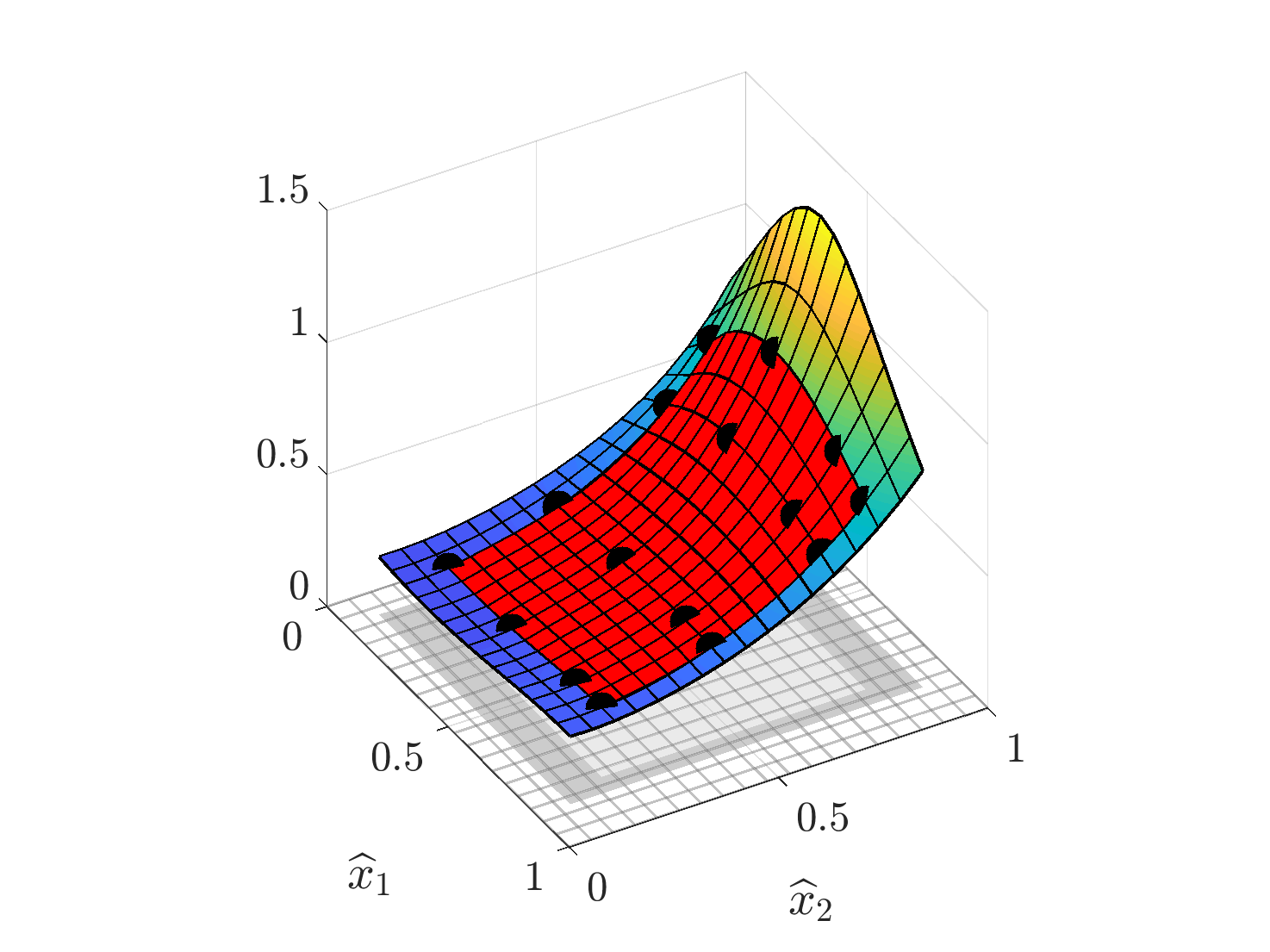}
        \label{fig:delta10}
    \end{subfigure}
    \quad
    \begin{subfigure}[c]{0.28\textwidth}
        \includegraphics[trim=2cm 0cm 2cm 0cm,clip=true,width=\textwidth]{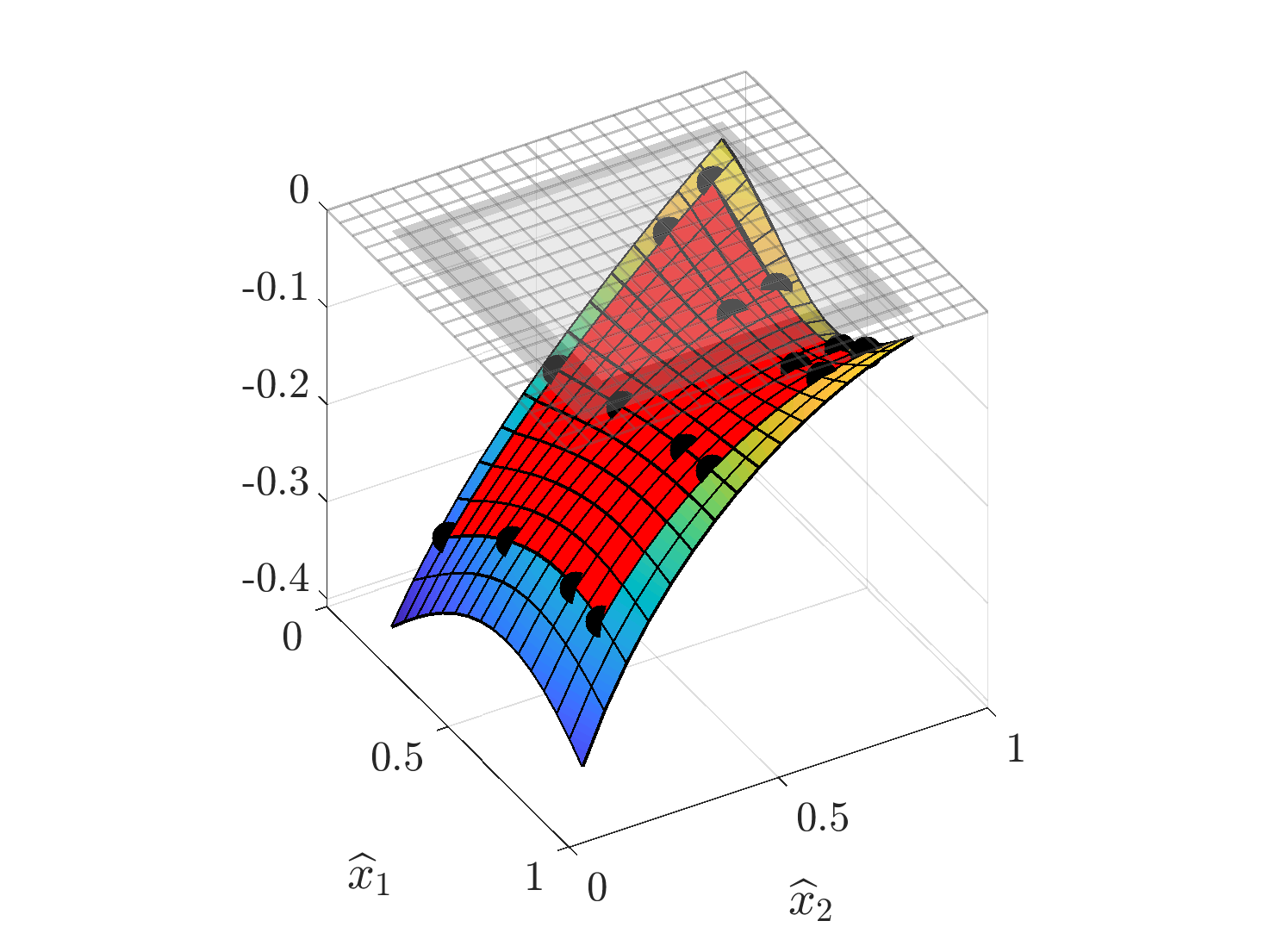}
        \label{fig:delta-11}
    \end{subfigure}

    \begin{subfigure}[c]{0.28\textwidth}
        \includegraphics[trim=2cm 0cm 2cm 0cm,clip=true,width=\textwidth]{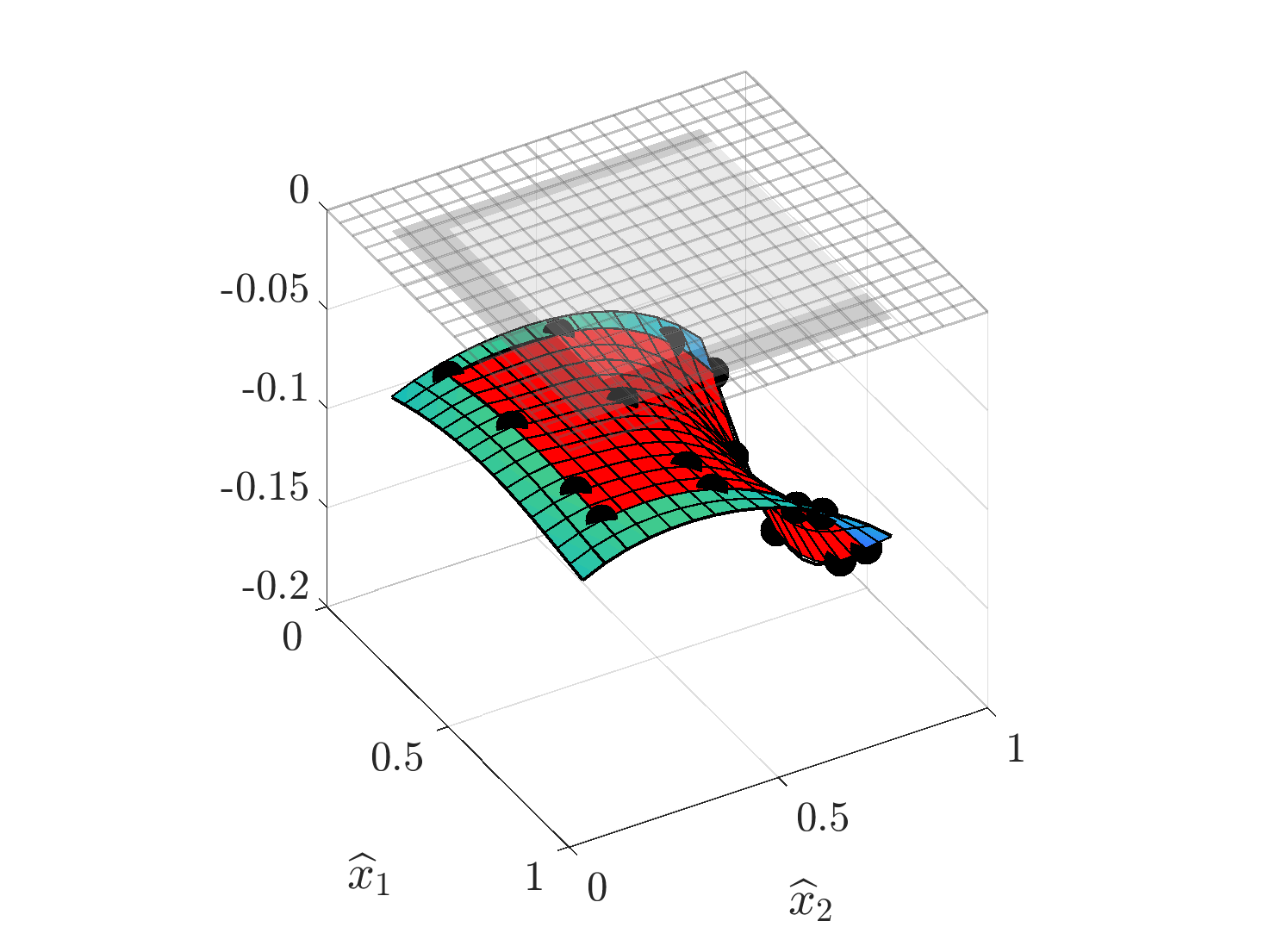}
        \label{fig:delta-12}
    \end{subfigure}
    \quad
    \begin{subfigure}[c]{0.28\textwidth}
        \includegraphics[trim=2cm 0cm 2cm 0cm,clip=true,width=\textwidth]{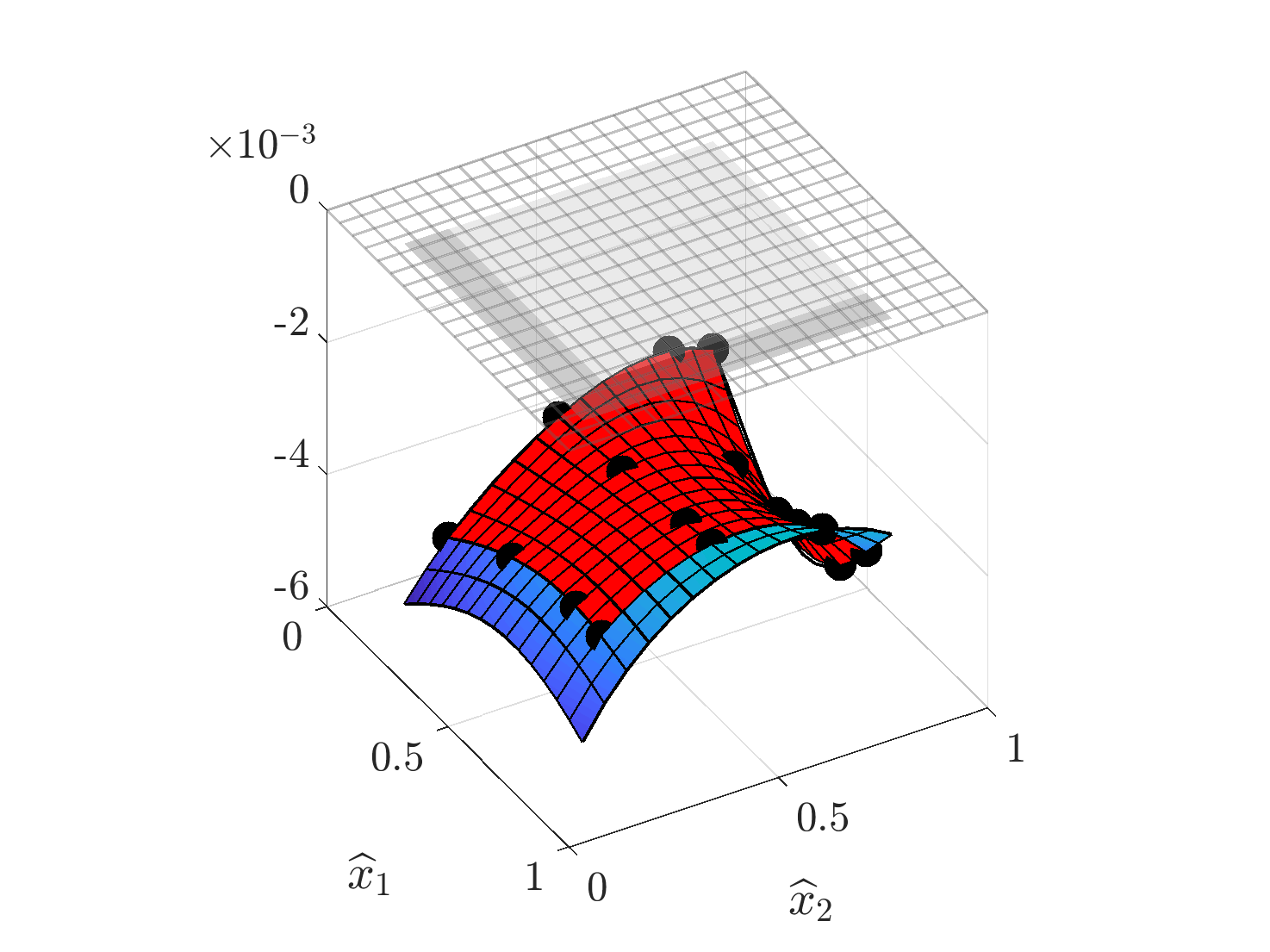}
        \label{fig:delta-22}
    \end{subfigure}
    \quad
    \begin{subfigure}[c]{0.28\textwidth}
        \includegraphics[trim=2cm 0cm 2cm 0cm,clip=true,width=\textwidth]{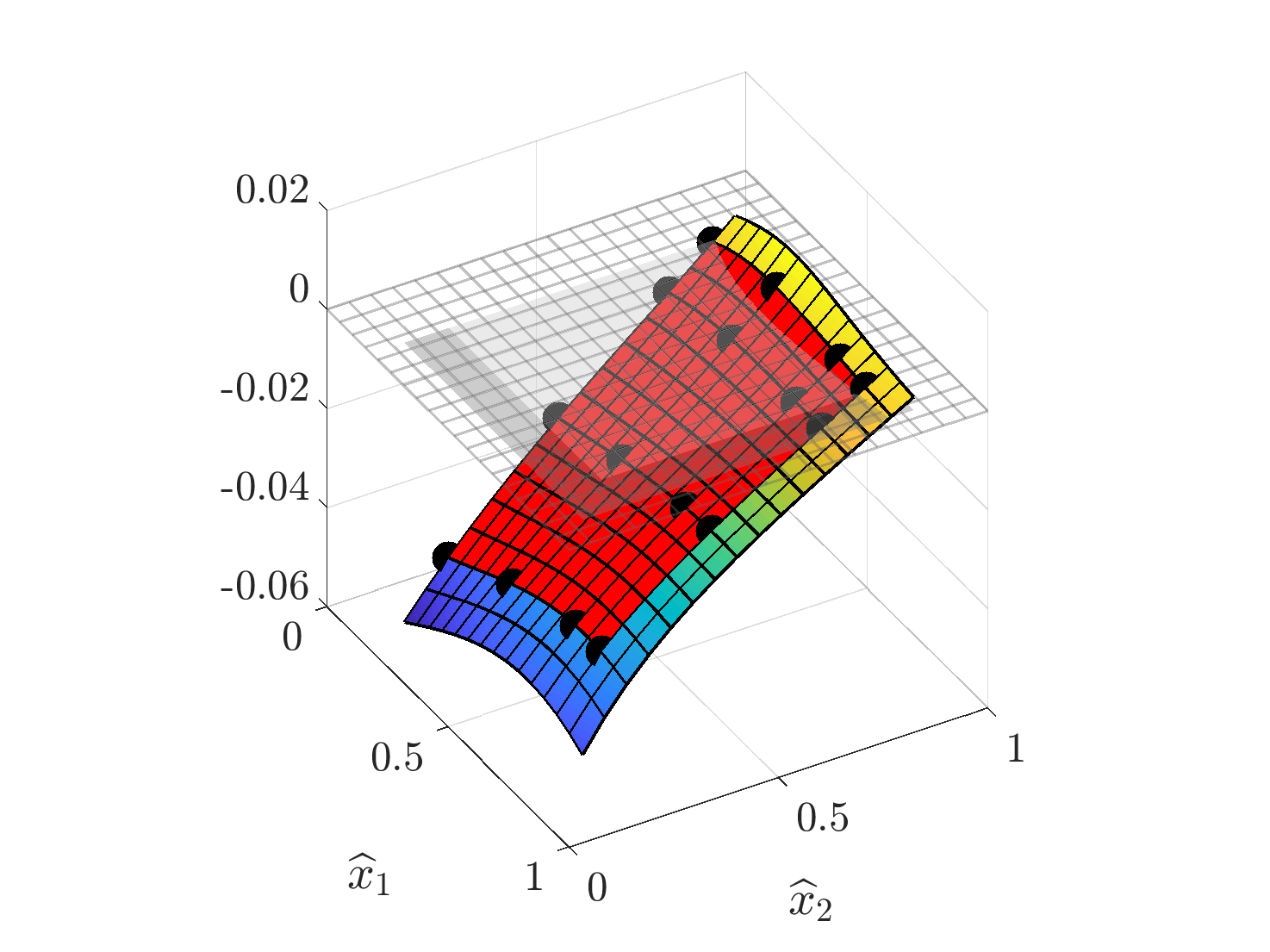}
        \label{fig:delta-21}
    \end{subfigure}

    \begin{subfigure}[c]{0.28\textwidth}
        \includegraphics[trim=2cm 0cm 2cm 0cm,clip=true,width=\textwidth]{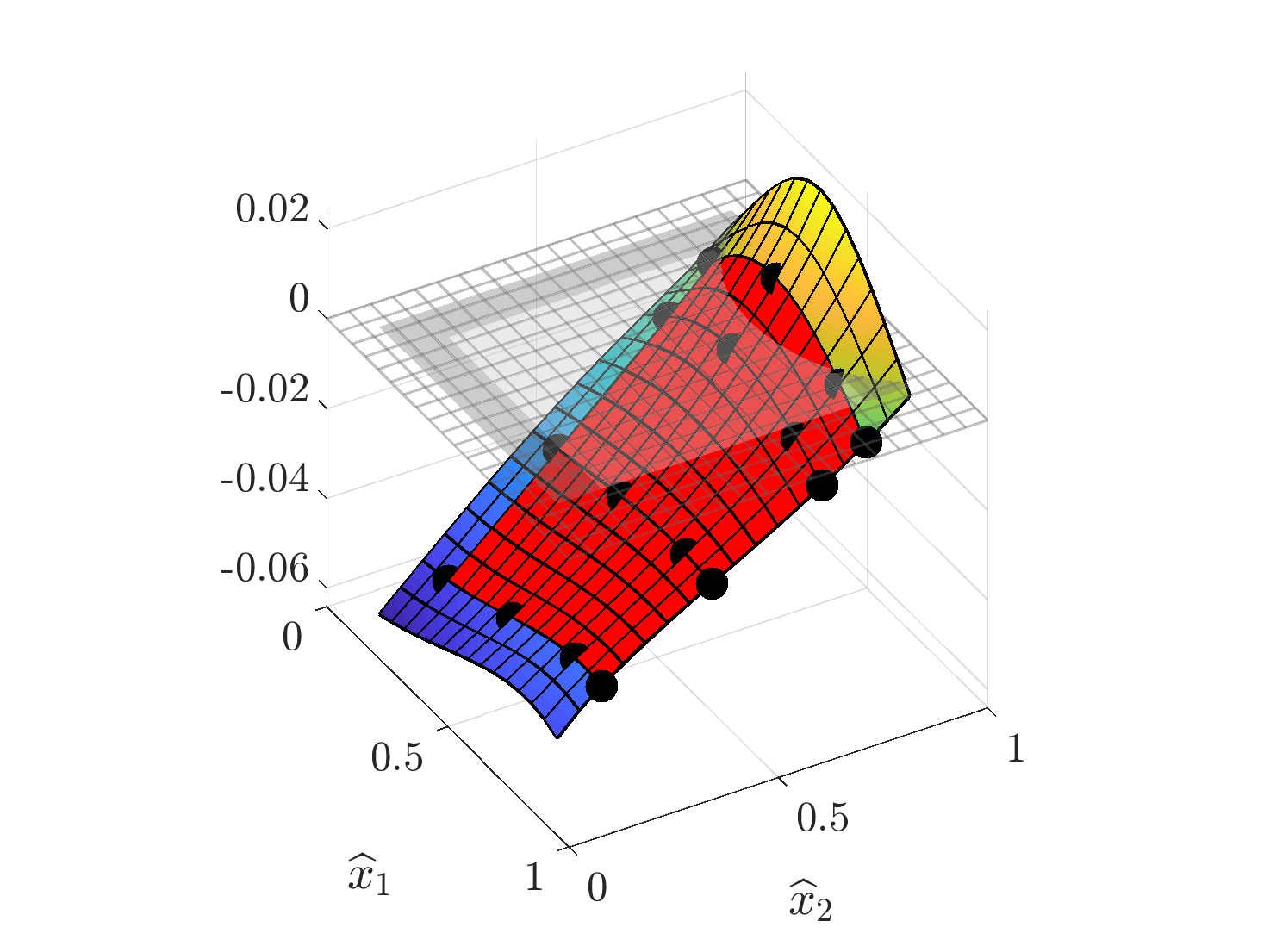}
        \label{fig:delta20}
    \end{subfigure}
    \quad
    \begin{subfigure}[c]{0.28\textwidth}
        \includegraphics[trim=2cm 0cm 2cm 0cm,clip=true,width=\textwidth]{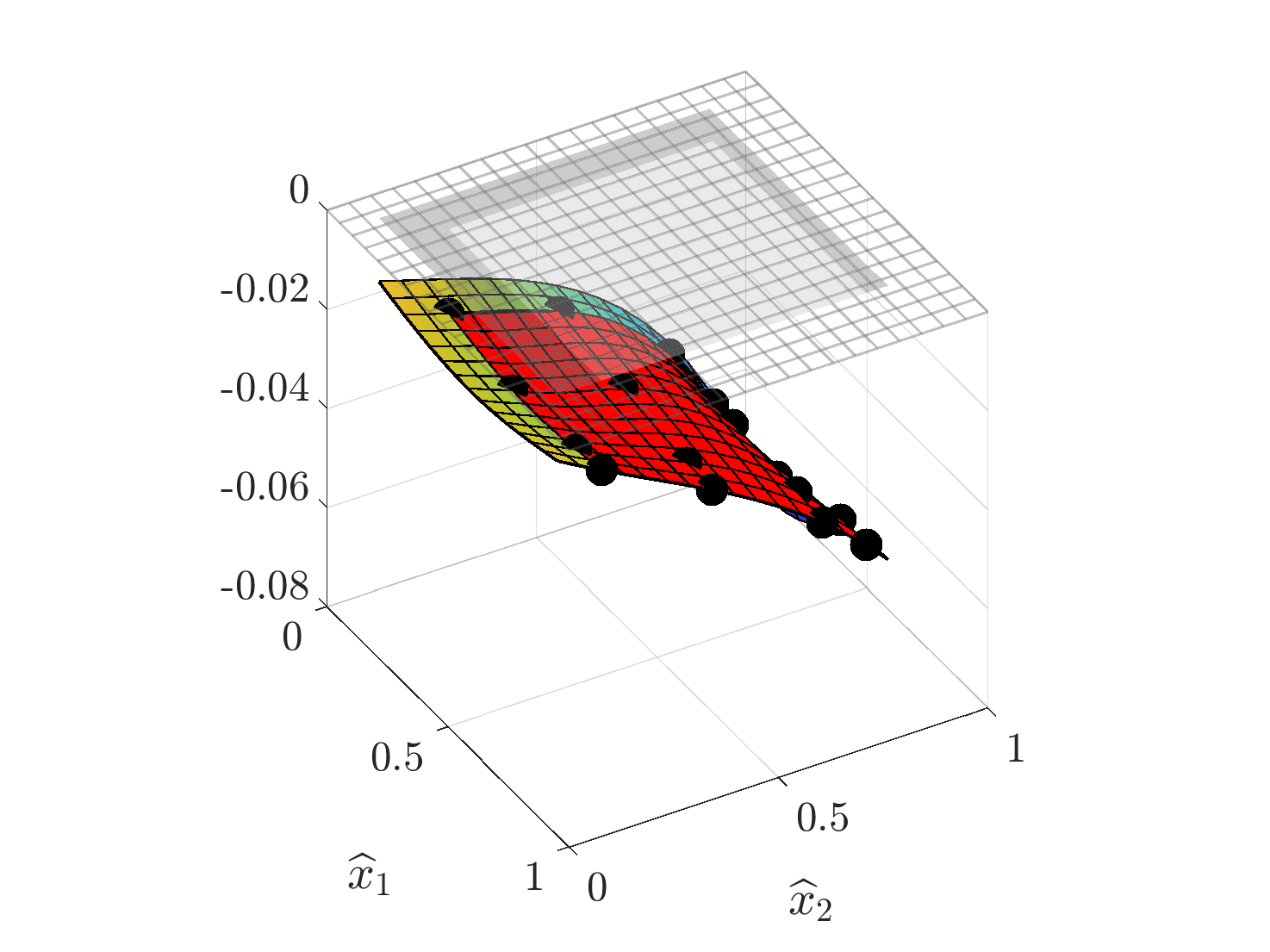}
        \label{fig:delta21}
    \end{subfigure}
    \quad
    \begin{subfigure}[c]{0.28\textwidth}
        \includegraphics[trim=2cm 0cm 2cm 0cm,clip=true,width=\textwidth]{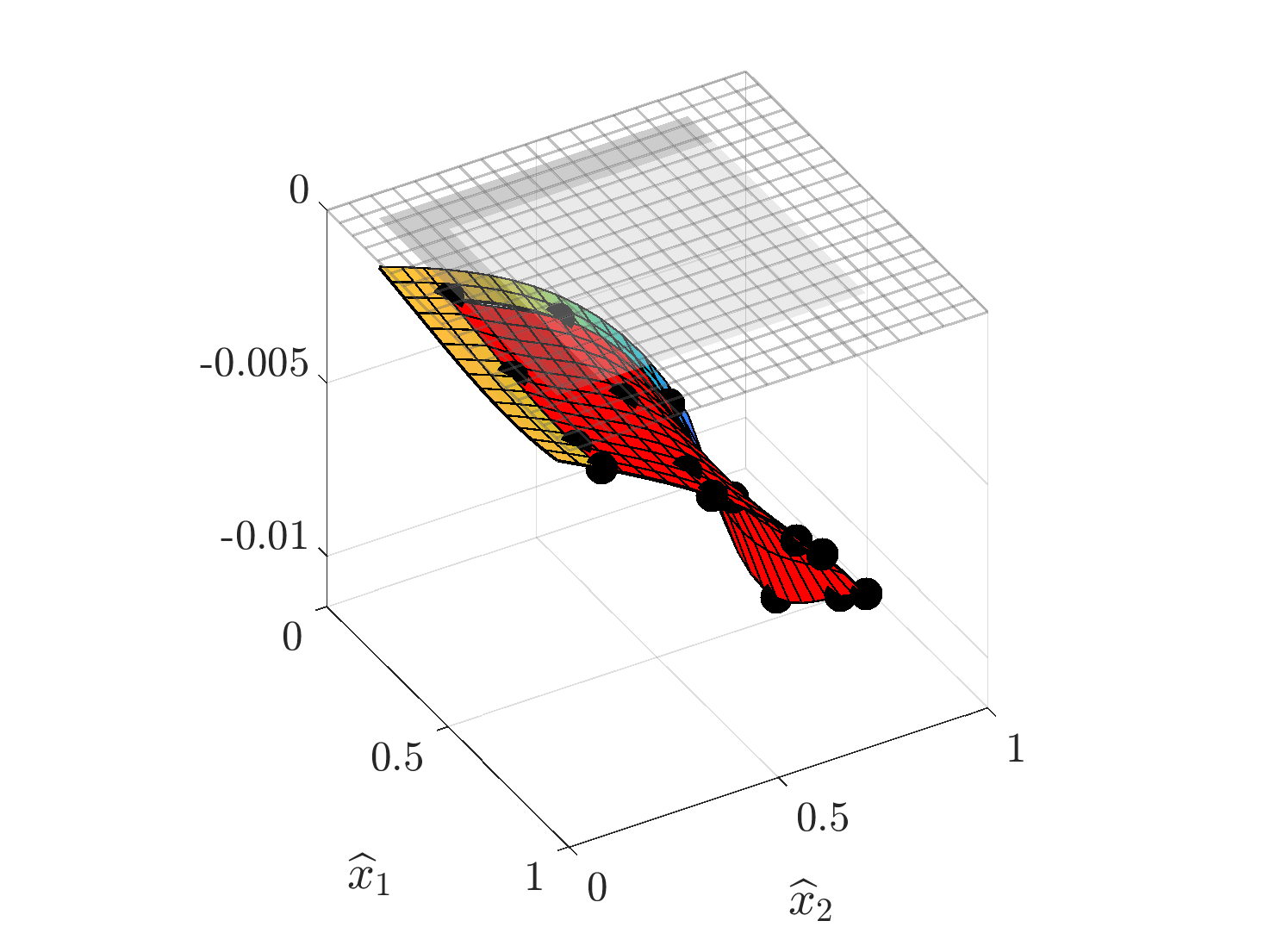}
        \label{fig:delta22}
    \end{subfigure}
  \caption{\label{fig:SurrogateStencilFunctions}Graphs of the stencil functions $\Phi_\bdelta(\tilde{\bmx})$ (color gradient) and their surrogates $\tilde{\Phi}_\bdelta(\tilde{\bmx})$ (red) for the various $\bdelta$ required in populating the upper diagonal of the stiffness matrix $\sfA_{ij}$ defined in~\cref{eq:DefinitionsOfSurrogateMatrices}. The sampling points $\tilde{\bmx}_{i}^\mathrm{s} \in \tilde{\bm{\Xi}} \subsetneq \tilde{\Omega}\cap\tilde{\bbX}$ used to construct the interpolant are depicted with black dots. The subsets $\tilde{\Omega}\subset \tilde{\Omega}_\bdelta$ are highlighted in light gray and the remaining subsets $\tilde{\Omega}_\bdelta\setminus\tilde{\Omega}$ are depicted in a darker gray.
  }
\end{figure}

\section{Surrogate matrices: Preserving structure} %
\label{sec:surrogate_matrices_preserving_structure}

In this section, we discuss a number of different surrogate matrix definitions.
Under both definitions, we also show that the surrogate matrices may actually fully reproduce the very matrices they approximate.
Note that these definitions may also be employed for general matrices and not only for a stiffness matrix.
If not stated otherwise, in the following subsections we assume that the matrix $\sfA$ emanates from a discretization of a general bilinear form.

\subsection{General surrogate matrices} %
\label{sub:general_surrogate_matrices}

In \Cref{sub:exploiting_basis_structure1}, the definition $\tilde{\sfA}_{ij} = \tilde{\Phi}_\bdelta(\tilde{\bmx}_i)$ was presented, with $\bdelta = \tilde{\bmx}_j-\tilde{\bmx}_i$.
However, this is valid exclusively for the special indices $1\leq i,j\leq N$ with a cardinal B-spline associated to them.
Of course, we may always compute the remaining components of the surrogate matrix $\tilde{\sfA}$ directly, using element-wise quadrature, as done in standard IGA assembly algorithms.
Therefore, we propose the following definition:
\begin{subequations}
\label{eq:DefinitionsOfSurrogateMatrices}
\begin{equation}
  \tilde{\sfA}_{ij}
  =
  \begin{cases}
  \tilde{\Phi}_\bdelta(\tilde{\bmx}_i) & \text{if } \tilde{\bmx}_i,\tilde{\bmx}_j\in\tilde{\Omega},\\
  \sfA_{ij}  &\text{otherwise.}\\
  \end{cases}
\label{eq:DefinitionOfSurrogateMatrix}
\end{equation}
This is a convenient definition which essentially can be used for constructing a surrogate matrix from any general bilinear form $a(\cdot,\cdot)$.
Typically, this definition is used for general non-symmetric matrices, e.g., for the divergence matrices $\sfB$ arising in the discretization of Stokes' flow (cf. \Cref{sec:stokes_equation}).
However,~\cref{eq:DefinitionOfSurrogateMatrix} can easily be improved when the bilinear form is symmetric or has a well-known kernel.

\subsection{Symmetry} %
\label{sub:symmetric_matrices}

If the bilinear form is symmetric, \cref{eq:DefinitionOfSurrogateMatrix} does not guarantee that the corresponding surrogate stiffness matrix $\tilde{\sfA}$ will be symmetric.
In order to enforce symmetry, it is convenient to just include the action of copying $\sfA_{ij}$ into $\sfA_{ji}$, for all $i>j$.
Therefore, we propose the following symmetric surrogate matrix definition:
\begin{equation}
  \tilde{\sfA}_{ij}
  =
  \begin{cases}
  \tilde{\Phi}_\bdelta(\tilde{\bmx}_i) & \text{if } \tilde{\bmx}_i,\tilde{\bmx}_j\in\tilde{\Omega} \text{ and } i\leq j,\\
  \tilde{\sfA}_{ji} & \text{if } \tilde{\bmx}_i,\tilde{\bmx}_j\in\tilde{\Omega} \text{ and } i>j,\\
  \sfA_{ij}  &\text{otherwise.}\\
  \end{cases}
\label{eq:DefinitionOfSurrogateMatrixSymmetric}
\end{equation}
With the definition, note that only $\frac{(2p+1)^n+1}{2}$ surrogate stencil functions need to be computed.
We employ this definition in the construction of surrogate mass matrices $\sfM$ which are symmetric but do not have a kernel (cf. \Cref{sec:transverse_vibrations_of_an_isotropic_membrane}).

\subsection{Preserving the kernel} %
\label{sub:preserving_the_kernel}

Recall that $\sum N_i(\bmx) = 1$.
In the situation $a(u,v) = \int_\Omega \nabla u \cdot \nabla v \dd x$, we have that $a(1,w) = a(w,1) = 0$, for all $w\in H^1(\Omega)$.
Therefore, $1 \in V_h = \spann\{N_i\}$ and, moreover, $\sfA \sfv_1 = 0$, where $\sfv_1 = (1,\ldots,1)^\top$.
Clearly, neither definition~\cref{eq:DefinitionOfSurrogateMatrix} nor definition~\cref{eq:DefinitionOfSurrogateMatrixSymmetric}, will guarantee that $\tilde{\sfA}\sfv_1 = 0$.
Therefore, we \revised{will} pose the following symmetric kernel-preserving definition:
\begin{equation}
  \tilde{\sfA}_{ij}
  =
  \begin{cases}
  \tilde{\Phi}_\bdelta(\tilde{\bmx}_i) & \text{if } \tilde{\bmx}_i,\tilde{\bmx}_j\in\tilde{\Omega} \text{ and } i<j,\\
  \tilde{\sfA}_{ji} & \text{if } \tilde{\bmx}_i,\tilde{\bmx}_j\in\tilde{\Omega} \text{ and } i>j,\\
  \sfA_{ij}  &\text{in all other cases where } i\neq j\\
  - \sum_{k\neq i}\tilde{\sfA}_{ik} & \text{if } i=j,\\
  \end{cases}
\label{eq:DefinitionOfSurrogateMatrixSymmetricKernel}
\end{equation}
\end{subequations}
With this definition, the reader may readily verify that $\tilde{\sfA}\sfv_1 = 0$ and $\tilde{\sfA} = \tilde{\sfA}^\top$.

\begin{remark}
Two important comments are in order.
First, in a matrix-free setting, where memory copying cannot be performed efficiently, one can actually design a set of surrogate stencil functions $\tilde{\Phi}_\bdelta$ which preserves the symmetry of the stiffness matrix (see, e.g., \cite[Remark~3.4]{drzisga2018surrogate}).
Second, in general, it is difficult to generalize the row sum trick in~\cref{eq:DefinitionOfSurrogateMatrixSymmetricKernel}, in a way which preserves symmetry, when the bilinear form $a(\cdot,\cdot)$ has a multi-dimensional kernel (cf. \Cref{sec:the_biharmonic_equation}).
\end{remark}

\subsection{Polynomial reproduction} %
\label{sub:polynomial_reproduction}

Until now, we have focused, almost entirely, on the analysis of surrogate stiffness matrices which derive from the bilinear form \changed{generated by} Poisson's equation.
Clearly, the methodology presented above can be applied to other settings as well.
The general scenario we are interested in is when $a(\cdot,\cdot)$ in~\cref{eq:ContinousVF} can be expressed in the parametric domain as
\begin{subequations}
\begin{equation}
    \hat{a}(\hat{w},\hat{v}) = \int_{\hat{\Omega}} G(\hat{\bmx},\revised{\hat{w}}(\hat{\bmx}),\revised{\hat{v}}(\hat{\bmx})) \dd \hat{\bmx}
    \qquad \text{for all } \hat{w}, \hat{v}\in \hat{V},
\label{eq:BilinearFormStructure}
\end{equation}
where, for all smooth $\hat{w},\hat{v}$,
\begin{equation}
    G(\hat{\bmx},\hat{w}(\hat{\bmy}),\hat{v}(\hat{\bmy})) = 0
    ,
    \qquad
    \text{whenever }
    \hat{\bmy}\notin\supp(\hat{w})\cap\supp(\hat{v})
    \,.
\label{eq:SupportAssumption}
\end{equation}
\end{subequations}

We now consider the general coefficient matrix $\sfA$ and a cardinal B-spline basis $\{\hat{B}_i\}$ (cf. \Cref{sub:exploiting_basis_structure2}).
Invoking~\cref{eq:SupportAssumption}, a simple change of variables leads us to
\begin{equation}
    \begin{aligned}
        \sfA_{ij}
        =
        \revised{\hat{a}}(\hat{B}_j,\hat{B}_i)
        &=
        \int_{\hat{\omega}_{\bdelta}} \!\!G(\tilde{\bmx}_i+\hat{\bmy},\hat{B}_{\bdelta}(\hat{\bmy}),\hat{B}(\hat{\bmy})) \dd \hat{\bmy}
        \,,
    \end{aligned}
\label{eq:GeneralChangeOfVars}
\end{equation}
where, as before, $\bdelta = \tilde{\bmx}_j - \tilde{\bmx}_i$ and $\hat{\omega}_\bdelta = \supp(\hat{B})\cap\supp(\hat{B}_{\bdelta})$.
Using the techniques put forth in \Cref{sub:basis_structure_nurbs}, this expression can easily be generalized for cardinal NURBS bases with polynomial weight functions $W(\hat{\bmx})$.
However, considering only the case of a cardinal B-spline basis, if $G(\cdot,\cdot,\cdot)$ is a $\mcQ_p(\hat{\Omega})$ polynomial in its first argument, we have the following reproduction property (cf. \Cref{rem:PolynomialReproduction}).
\begin{proposition}
\label{prop:PolynomialReproduction}
    Assume that~\cref{eq:BilinearFormStructure,eq:SupportAssumption} hold.
    For all $\hat{\bmx}\in\tilde{\Omega}$, define
    \begin{equation}
        \Phi_\bdelta(\tilde{\bmx})
        =
        \int_{\hat{\omega}_{\bdelta}} G(\tilde{\bmx}+\hat{\bmy},\hat{B}_{\bdelta}(\hat{\bmy}),\hat{B}(\hat{\bmy})) \dd \hat{\bmy}
        .
    \label{eq:GeneralStencilFunction}
    \end{equation}
    If $G(\cdot,\hat{\bmy},\hat{\bmy})\in\mcQ_p(\hat{\Omega})$, for every $\hat{\bmy}\in\hat{\Omega}$, then $\Phi_\bdelta \in \mcQ_q(\tilde{\Omega})$.
    Moreover, taking~\cref{eq:DefinitionOfSurrogateMatrix} as the definition of the surrogate $\tilde{\sfA}$, it holds that $\tilde{\sfA} = \sfA$.
\end{proposition}
\begin{proof}
    It suffices to show that $\tilde{\Phi}_\bdelta = \Phi_\bdelta$, for all $\bdelta \in \scD$.
    Let $\balpha = (\alpha_1,\ldots,\alpha_n)$ be a multi-index, $\tilde{\bmx}^\balpha = \tilde{x}_{1}^{\alpha_1}\cdots \tilde{x}_{n}^{\alpha_n}$.
    By assumption, we may express $G(\tilde{\bmx}, \hat{B}_\bdelta(\hat{\bmy}),\hat{B}(\hat{\bmy})) = \sum_{i=1}^n \sum_{\alpha_i \leq p} c_\balpha(\hat{\bmy}) \tilde{\bmx}^\balpha$, where each coefficient function $c_\balpha(\hat{\bmy})$ has support only in $\hat{\omega}_\bdelta$.
    Moreover, if $\tilde{\bmx}+\hat{\bmy}\in \hat{\Omega}$, then
    \begin{equation}
        G(\tilde{\bmx}+\hat{\bmy}, \hat{B}_{\bdelta}(\hat{\bmy}),\hat{B}(\hat{\bmy}))
        =
        \sum_{i=1}^n \sum_{\alpha_i \leq p} c_{\balpha}(\hat{\bmy}) (\tilde{\bmx}+\hat{\bmy})^\balpha
    \label{eq:PolyExpansion}
    \end{equation}
    is clearly an equal degree polynomial in the $\tilde{\bmx}$-variable.
    The proof is completed noting that the integral in~\cref{eq:GeneralStencilFunction} is performed in the $\hat{\bmy}$-variable over the set $\hat{\omega}_{\bdelta}$ and $\revised{\hat{\omega}_{\bdelta}} \subset \hat{\Omega}$, for every $\bdelta\in\scD$.~
\end{proof}

\section{Surrogate matrices: Faster assembly with existing software} %
\label{app:computing_surrogates_with_existing_iga_codes}
All the examples in this paper were implemented using the GeoPDEs package for Isogeometric Analysis in MATLAB and Octave \cite{de2011geopdes,vazquez2016new}.
This package provides a framework for implementing and testing new isogeometric methods for the solution of partial differential equations.
We reused most of the original low-level functions and only had to make changes to some high-level assembly functions.
A detailed explanation of the code modifications and extensions is provided in \cite{drzisga2019igasurrogateimpl}.
Additionally, a reference implementation with example code is available in the git repository \cite{githubdrzisga}.
Nonetheless, we give here a short explanation of the implementation in GeoPDEs.
In particular, for the Poisson problem, we modified \mbox{\texttt{op\_gradu\_gradv\_tp}} using the following strategy:

\revised{First, the sampling length $H$ and the sampling points need to be specified.
For this purpose, we introduce the sampling parameter $M ̄\in \N$ which relates the small scale $h$ to the coarse scale $H$ via $H = M \cdot h$.}
Starting from the first point $\tilde{\bmx}_i\in{\tilde{\Omega}}\cap\tilde{\bbX}$, let $\tilde{\bm{\Xi}}$ be the lattice containing every $M{}^\text{th}$ point $\tilde{\bmx}_i\in{\tilde{\Omega}}\cap\tilde{\bbX}$, in each Cartesian direction.
When $M$ does not evenly divide these points in any given Cartesian direction, include the $M{}^\text{th}$ endpoints $\tilde{\bmx}_i\in{\bdry\tilde{\Omega}}\cap\tilde{\bbX}$ as well; see, e.g., the black dots in \Cref{fig:SurrogateStencilFunctions}.
The stencil functions $\Phi_\bdelta(\tilde{\bmx})$ are evaluated at all points $\tilde{\bmx}_{i}^\mathrm{s}\in\tilde{\bm{\Xi}}$ and these values are used as the support points of the ensuing B-spline interpolant $\tilde{\Phi}_\bdelta(\tilde{\bmx})$.

An additional benefit of this choice is seen in that $\tilde{\Phi}_\bdelta(\tilde{\bmx}_{i}^\mathrm{s}) = \Phi_\bdelta(\tilde{\bmx}_{i}^\mathrm{s})$, at each point $\tilde{\bmx}_{i}^\mathrm{s}\in \tilde{\bm{\Xi}}$.
This leads to an increased point-wise accuracy and lower potential cost, since each entry $\tilde{\sfA}_{ij} = \tilde{\Phi}_\bdelta(\tilde{\bmx}_{i}^\mathrm{s})$ in the surrogate stiffness matrix is equal to the correct entry, $\sfA_{ij} = \Phi_\bdelta(\tilde{\bmx}_{i}^\mathrm{s})$.
Here, it is appropriate to point \revised{out} that when $M=1$ every point $\tilde{\bmx}_i$ is sampled, $\tilde{\bm{\Xi}} = \tilde{\Omega}\cap\tilde{\bbX}$.
In this case, $H=h$ and there is no difference from the surrogate $\tilde{\sfA}$ and the true $\sfA$.

In order to evaluate $\Phi_\bdelta(\tilde{\bmx}_{i}^\mathrm{s})$, we identify the matrix rows which correspond to the sampling points $\tilde{\bmx}_{i}^\mathrm{s}\in\tilde{\bm{\Xi}}$.
Additionally, we include the rows which correspond to basis functions near the domain boundary.
Each of these rows \revised{needs} to be assembled using quadrature formulas; see the red and green points in \cref{fig:Sparsity}.
The number of interior rows depends on $M$, whereas the number of rows corresponding to the boundary depends on the order of the basis functions $p$.
After that, we identify all of the active elements which need to be assembled to compute the estimated rows, cf.~\cref{fig:ActiveElements}.
Note that the number of required elements for each sample point depends on the order $p$.
In order to assemble these elements, we employ the \mbox{\texttt{op\_gradu\_gradv}} function, but skip the elements which are not active.

\begin{figure}%
\centering
\includegraphics[width=0.7\textwidth]{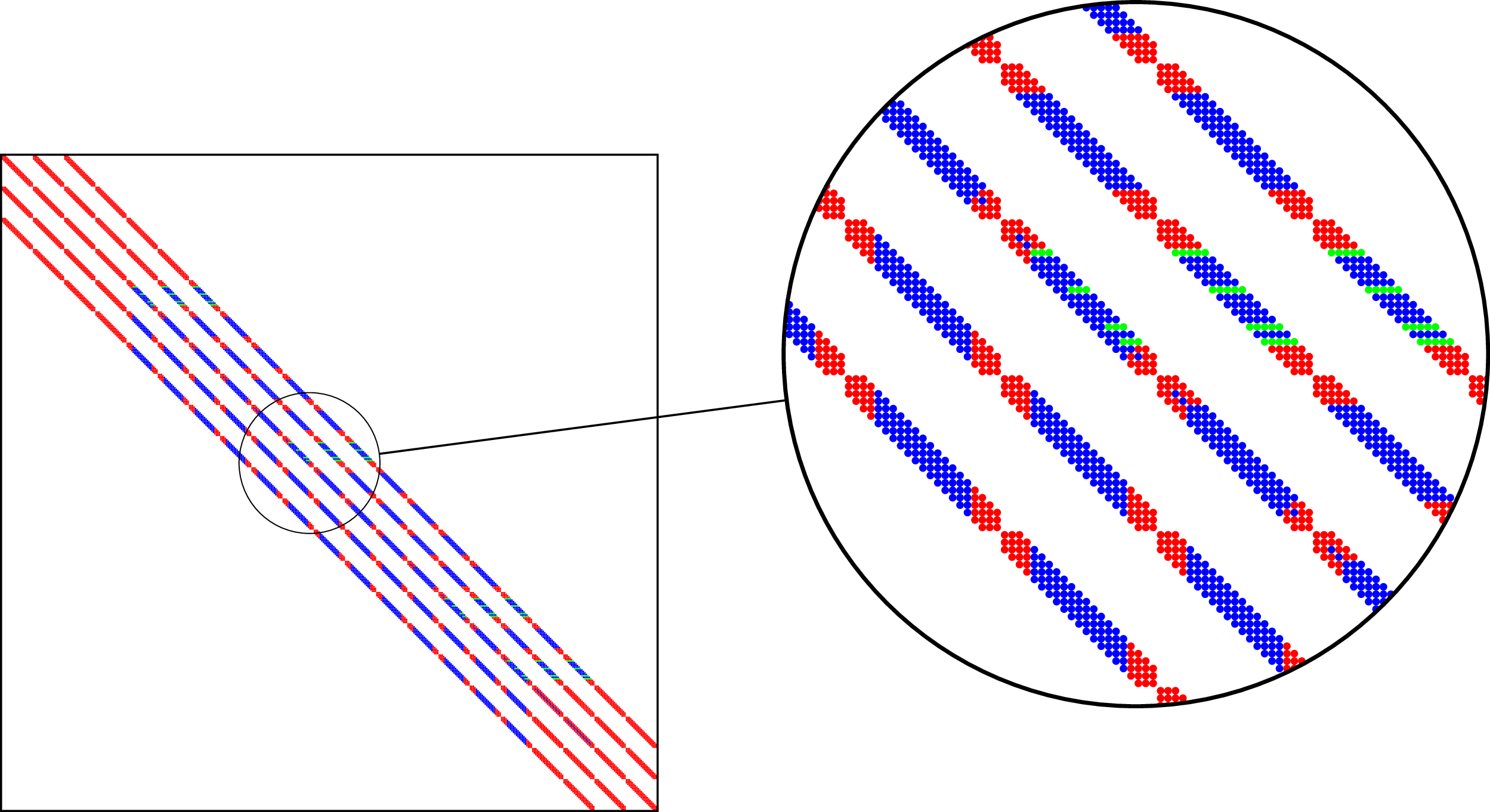}
\caption{\label{fig:Sparsity}Sparsity pattern of the surrogate stiffness matrix $\tilde{\sfA}$. The red and green points indicate the entries of the stiffness matrix which are evaluated in the standard way. The blue points indicate the entries which are obtained by evaluating the interpolated stencil functions. The red points correspond to the basis functions near the boundaries and the green entries are used as supporting points for the interpolation. \revised{Some of the diagonal entries are drawn in blue due to the modification of preserving the kernel.}}
\end{figure}

\begin{figure}
  \centering
    {\begin{subfigure}[c]{0.42\textwidth}
        \includegraphics[height=5.5cm]{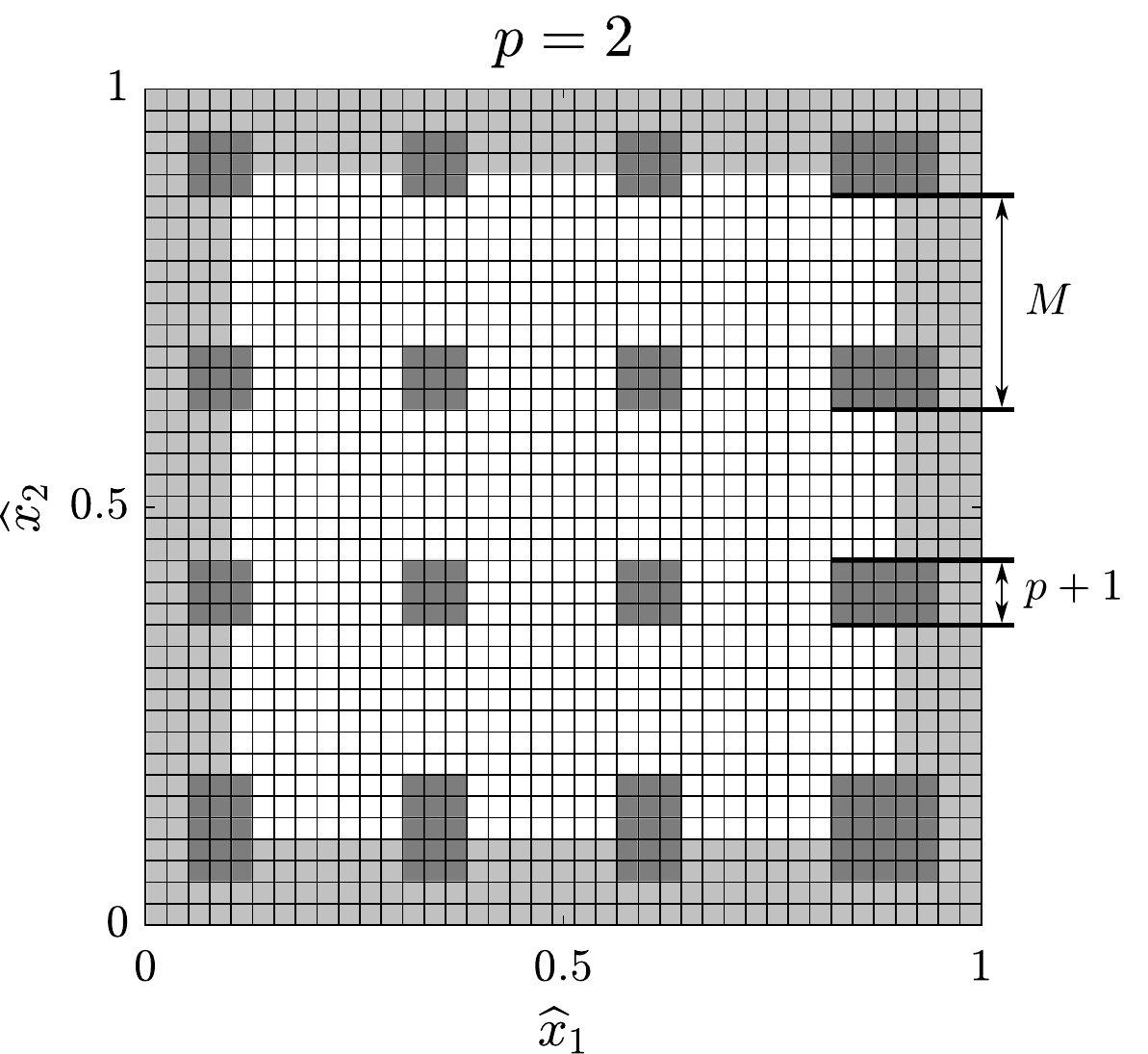}
    \end{subfigure}}
    \qquad\quad
    {\begin{subfigure}[c]{0.42\textwidth}
        \includegraphics[height=5.5cm]{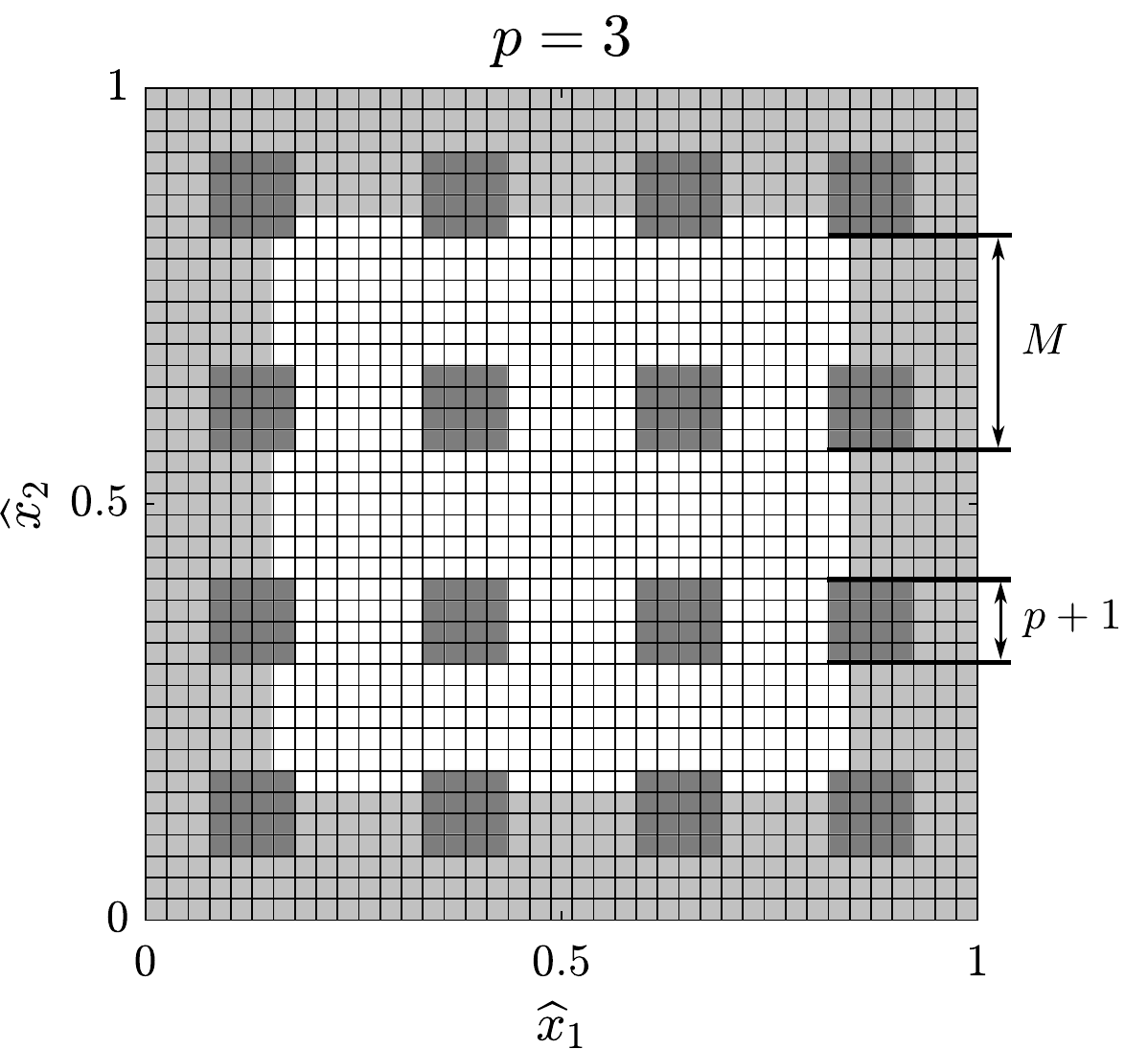}
    \end{subfigure}}
    \caption{\label{fig:ActiveElements}The active elements (shown in gray) involved in the surrogate assembly for $M=10$ with forty knots in each Cartesian direction. The light gray elements correspond to the active boundary elements and the dark gray elements correspond to the inner active elements required for the sampling of the stencil functions.}
\end{figure}

To construct the interpolated stencil functions $\tilde{\Phi}_\bdelta$, it is possible to use the builtin MATLAB functions \mbox{\texttt{interp2}} and \mbox{\texttt{interp3}}.
However, in 2D, we use the \mbox{\texttt{RectBivariateSpline}} function provided by the SciPy Python package \cite{scipy}, which supports spline interpolation up to order~$5$.
These interpolated stencil functions are then evaluated in order to retrieve the remaining values of $\tilde{\sfA}$; cf.~the blue off-diagonal entries in \cref{fig:Sparsity}.
The assembly functions for the mass matrix, the biharmonic equation, and the Stokes problem were modified in a similar way.
For symmetric matrices, only the upper-diagonal entries are interpolated and copied to the lower-diagonal entries (cf. \cref{eq:DefinitionOfSurrogateMatrixSymmetric}).
In the Poisson and biharmonic case, we additionally enforce the zero-row sum property by changing the diagonal entries for all rows which include at least one interpolated value (cf. \cref{eq:DefinitionOfSurrogateMatrixSymmetricKernel}).

In order to show that the surrogate approach may be easily applied to other IGA frameworks, we tried to keep the modifications as simple as possible.
However, in an IGA implementation tailored to the surrogate approach, even more properties may be exploited to achieve better performance.
For example, in the current implementation, the complete local stiffness matrices of the active elements are computed via quadrature, but in practice only a single row of the local matrix is required.
Exploiting this fact would save a significant amount of unnecessary computation, especially as $p$ grows, but such an implementation in GeoPDEs would also involve the modification of low-level functions.

\begin{remark}
\revised{Dirichlet boundary conditions are enforced for surrogate methods in the standard way; that is, by eliminating dofs from the original linear system.
As usual, let $\tilde{\sfA}$ be the full matrix without any consideration for boundary conditions.
Without loss of generality, we may assume that all the smallest global indices correspond to the set of interior dofs, denoted by $I$, and all of the largest correspond to the set of Dirichlet dofs, denoted by $D$.
The unconstrained linear system, $\tilde{\sfA} \tilde{\sfu}=\sff$, may then be written in block form as follows:
\begin{align}
\label{eq:BlockSystem}
\begin{bmatrix}
\tilde{\sfA}_{II} & \tilde{\sfA}_{ID}\\
\tilde{\sfA}_{DI} & \tilde{\sfA}_{DD}
\end{bmatrix} \begin{bmatrix}
\tilde{\sfu}_I\\
\tilde{\sfu}_D
\end{bmatrix} = \begin{bmatrix}
\sff_I\\
\sff_D
\end{bmatrix}
.
\end{align}
Recall that all cardinal basis functions vanish at domain boundary $\bdry\Omega$.
Therefore, following definition~\cref{eq:DefinitionOfSurrogateMatrixSymmetricKernel}, $\tilde{\sfA}_{ID} = {\sfA}_{ID}$ and $\tilde{\sfA}_{DD} = {\sfA}_{DD}$, where ${\sfA}_{ID}$ and ${\sfA}_{DD}$ are the corresponding submatrices of the standard stiffness matrix $\sfA$.
Furthermore, because the values of $\tilde{\sfu}_D$ are prescribed,~\cref{eq:BlockSystem} may be reduced to the linear system $\tilde{\sfA}_{II}\tilde{\sfu}_I = \sff_I - {\sfA}_{ID}\tilde{\sfu}_D$.
This reduced system may be solved to determine all the unprescribed solution coefficients, $\tilde{\sfu}_I$.
}
\end{remark}

\section{Poisson's equation} %
\label{sec:Poissons_equation}

In this section, we analyze a surrogate IGA discretization of Poisson's equation on the domain $\Omega = \bvarphi(\hat{\Omega})$.
Here, as well as in the forthcoming problems, we restrict our attention to Dirichlet boundary conditions.
This simplifies the analysis while also retaining all of its interesting features.
Given a function $f\in L^2(\Omega)$, the corresponding weak form is the following:
\begin{equation}
	\text{Find } u\in H^1_0(\Omega)
	\text{ satisfying }
	\quad
	a(u,v) = F(v)
	\quad
	\text{for all }
	v\in H^1_0(\Omega)
	\,,
\label{eq:PoissonWeakForm}
\end{equation}
where $a(u,v) = \int_\Omega \nabla u \cdot \nabla v \dd x$ and $F(v) = \int_\Omega f\hspace{0.25pt} v \dd x$.
At this point, it has been made well-understood that the bilinear form $a(\cdot,\cdot)$ can be rewritten on the parametric domain $\hat{\Omega}$, using the expression~\cref{eq:BilinearFormPoissonTensor}.
Recalling this detail, we continue on with the simplifying assumption $K \in \big[W^{q+1,\infty}(\hat{\Omega})\big]^{n\times n}$.

\subsection{Inconsistency} %
\label{sub:accuracy}

Recall~\cref{eq:VariationalFormulations,eq:DiscreteVariationalProblems} and take $V_h = \spann\{N_i\}$, where every $N_i = \hat{N}_i \circ \bvarphi^{-1}$.
Analysis of surrogate methods best proceeds using the surrogate bilinear form $\tilde{a}:V_h\times V_h\to \R$ inherent to the surrogate matrix $\tilde{\sfA}$.
Explicitly,
\begin{equation}
	\tilde{a}(w_h,v_h)
	=
	\sfv^\top \tilde{\sfA}\sfw
	,
\end{equation}
for all $w_h = \sum_i \sfw_i N_i, v_h = \sum_i \sfv_i N_i \in V_h$.
Here and throughout, we shall use definition~\cref{eq:DefinitionOfSurrogateMatrixSymmetricKernel} in constructing the surrogate stiffness matrix $\tilde{\sfA}$ and its associated surrogate bilinear form $\tilde{a}(\cdot,\cdot)$.

In the proceeding analysis, \Cref{thm:BilinearFormDifferencePoisson} is of fundamental importance.
Its proof is a simple consequence of \Cref{thm:RegularityOfStencilFunctions,lem:FiniteDifference}.
From now on, for any matrix $\sfN$, we use the notation $|\sfN|_{\max} = \max_{i\neq j} |\sfN_{ij}|$.

\begin{lemma}
\label{lem:FiniteDifference}
	For all $v_h,w_h \in V_h$, the following upper bound holds:
	\begin{align}
		|\a(v_h, w_h) - \tilde{a}(v_h, w_h)|
		\leq
		C_3
		h^{2-n}
		|\sfA-\tilde{\sfA}|_{\max} \|\nabla v_h\|_0 \|\nabla w_h\|_0
		\,,
	\label{eq:FiniteDifferenceBound}
	\end{align}
	where $C_3$ is a constant depending only on \changed{$\bvarphi$} and $p$.
\end{lemma}

\begin{proof}
\revised{
	In this proof, we will use the symbols ``$\lesssim$'' and ``$\eqsim$'' to denote upper bounds and equivalence, respectively, up to constants depending at most on $p$ and {$\bvarphi$}.
	For each $i=1,\ldots,N$, let $\mcI(i)$ be the set of indices $j$ such that $\supp({N}_i)\cap\supp({N}_j)\neq \emptyset$ and notice that $\mcI(i) \leq |\scD| = (2p+1)^n$.
	We begin with the observation that $\sum_{i} \big(\sfA_{ij} - \tilde{\sfA}_{ij}\big) = \sum_{j} \big(\sfA_{ij} - \tilde{\sfA}_{ij}\big) = 0$.
	With these two identities in hand, we find that
	\begin{equation}
	\label{eq:DifferenceIdentity}
		\begin{aligned}
			\a(v_h, w_h) - \tilde{a}(v_h, w_h)
			&=
			-\frac{1}{2}\,
			{\sum_{i,j}}
			\big(\sfA_{ij} - \tilde{\sfA}_{ij}\big) (\sfv_i - \sfv_j)\sspace ({\sfw_i} - \sfw_j)
			\\
			&\leq
			|\sfA-\tilde{\sfA}|_{\max}
			\sum_i \sum_{j\in \mcI(i)} |\sfv_i - \sfv_j|\sspace |\sfw_i - \sfw_j|
			\\
			&\leq
			|\sfA-\tilde{\sfA}|_{\max}
			\sum_i \Bigg(
			\sum_{j\in \mcI(i)}
			|\sfv_i - \sfv_j|^2
			\Bigg)^{\onehalf}
			\Bigg(
			\sum_{j\in \mcI(i)} 
			|\sfw_i - \sfw_j|^2
			\Bigg)^{\onehalf}
			.
		\end{aligned}
	\end{equation}

	For the time being, fix the index $i=1,\ldots,N$.
	For each coefficient vector $\sfv\in\R^N$, define $|\sfv|_{\mcI(i)} = \big(\sum_{j\in \mcI(i)}|\sfv_i - \sfv_j|^2\big)^{1/2}$.
	One may easily check that $|\,\cdot\,|_{\mcI(i)}$ is a seminorm on $\R^N$.
	In order to identify its kernel, simply observe that $|\sfv|_{\mcI(i)} = 0$ iff $\sfv_j = \sfv_i$ for each coefficient $j\in \mcI(i)$.
	Since $i\in\mcI(i)$, an equivalent way of stating this condition is that there exists some constant $C\in\R$ such that $\sfv_j = C$ for each $j\in\mcI(i)$.
	Now, recall that ${v}_h({\bmx}) = \sum_i \sfv_i {N}_i({\bmx})$ for all ${\bmx} \in {\Omega}$ and consider the following alternative seminorm:
	\begin{equation}
		|\sfv|_{{N}_i}
		=
		\|\nabla {v}_h\|_{0,\supp({N}_i)}
		\,.
	\end{equation}
	Notice that $|\sfv|_{{N}_i} = 0$ iff ${v}_h({\bmx})$ is equal to a constant on $\supp({N}_i)$, say $C\in\R$.
	From the partition of unity property inherent to all NURBS bases, it must hold that $C = {v}_h(\bmx) = \sum_{j\in \mcI(i)} C{N}_j(\bmx)$ for all $\bmx\in\supp({N}_i)$.
	In other words, since $\{{N}_i|_{\supp(N_i)}\}_{j\in\mcI(i)}$ is a linearly independent set, $|\sfv|_{{N}_i} = 0$ iff $\sfv_j = C$ for each $j\in\mcI(i)$.

	In the previous paragraph, we showed that the kernels of $|\sfv|_{{N}_i}$ and $|\sfv|_{\mcI(i)}$ are identical; namely, $|\sfv|_{{N}_i} = 0$ iff $|\sfv|_{\mcI(i)} = 0$ iff $\sfv \in Q_i$, where
	\begin{equation}
		Q_i = \{\sfv\in\R^N \colon \sfv_j = \sfv_i \text{ for each } j\in\mcI(i)\}
		\,.
	\end{equation}
	Clearly, $|\,\cdot\,|_{{N}_i}$ and $|\,\cdot\,|_{\mcI(i)}$ induce norms on the quotient space $\R^N / Q_i$.
	The next important observation is that $|Q_i| = N + 1 - |\mcI(i)|$, which may be witnessed by inspection.
	Because $\mathrm{dim}(\R^N / Q_i) = |\mcI(i)| - 1 \leq (2p+1)^n -1 $ is finite and depends only on $p$, it follows from the well-known equivalence of norms on finite dimensional vector spaces (e.g., the vector space $\R^N / Q_i$) that the seminorms $|\,\cdot\,|_{{N}_i}$ and $|\,\cdot\,|_{\mcI(i)}$ are equivalent.
	Of course, the corresponding equivalence constants will depend on $h$, $\bvarphi$, and $p$.
	Nevertheless, a standard scaling argument is all that is required to see that $|\sfv|_{\mcI(i)} \eqsim h^{2-n} |\sfv|_{{N}_i}$.
	Therefore, employing~\cref{eq:DifferenceIdentity}, we may simply write
	\begin{equation}
		\begin{aligned}
			|\a(v_h, w_h) - \tilde{a}(v_h, w_h)|
			&\lesssim
			h^{2-n} |\sfA-\tilde{\sfA}|_{\max}
			\sum_i
			\big\|\nabla v_h\big\|_{0,\supp({N}_i)} \big\|\nabla w_h\big\|_{0,\supp({N}_i)}
			\,.
		\end{aligned}
	\end{equation}
	The proof is completed by applying the discrete Cauchy--Schwarz inequality to the right-hand side of the inequality above and employing the fact that $\big(\sum_i \|f\|_{0,\supp({N}_i)}^2\big)^{\onehalf} \eqsim \|f\|_0$, for all $f\in L^2(\Omega)$.
}
\end{proof}

\begin{theorem}
\label{thm:BilinearFormDifferencePoisson}
	Let $C_4 = C_2\cdot C_3$.
	For all $v_h,w_h \in V_h$, the following upper bound holds:
	\begin{align}
		|\a(v_h, w_h) - \tilde{a}(v_h, w_h)|
		\leq
		C_4
		H^{q+1}
		\|\nabla v_h\|_0 \|\nabla w_h\|_0
		\,.
	\end{align}
\end{theorem}
\begin{proof}
  Obviously, we are done if $a(\cdot,\cdot) = \tilde{a}(\cdot,\cdot)$.
  Therefore, assume $a(\cdot,\cdot) \neq \tilde{a}(\cdot,\cdot)$ and let $i<j$ be the indices of the maximal value $|\sfA_{ij}-\tilde{\sfA}_{ij}| = \big|\sfA - \tilde{\sfA}\big|_{\max} > 0$.
  Since $\tilde{\sfA}$ is defined by~\cref{eq:DefinitionOfSurrogateMatrixSymmetricKernel}, \Cref{thm:RegularityOfStencilFunctions} leads us to the inequality
  \begin{equation}
    \big|\sfA - \tilde{\sfA}\big|_{\max}
    =
    |\Phi_\bdelta(\tilde{\bmx}_i) - \tilde{\Phi}_\bdelta(\tilde{\bmx}_i)|
    \leq
    C_2h^{n-2} H^{q+1}
    \,.
  \label{eq:InftyNormStep1}
  \end{equation}
  The proof is completed using \Cref{lem:FiniteDifference}.
\end{proof}

\begin{remark}
\label{rem:StructureScalingLoss}
	In \Cref{lem:FiniteDifference}, it is important that $\tilde{\sfA}$ be defined using~\cref{eq:DefinitionOfSurrogateMatrixSymmetricKernel}.
	Indeed, in~\cref{eq:DifferenceIdentity}, it is the symmetry and the zero row sum property preserved in this definition which allows $|\a(v_h, w_h) - \tilde{a}(v_h, w_h)|$ to be bounded by products of differences in the coefficients $\sfv_i$ and $\sfw_j$.
	If we had used definition~\cref{eq:DefinitionOfSurrogateMatrix} or \cref{eq:DefinitionOfSurrogateMatrixSymmetric}, one would have to directly work with the upper bound
	\begin{equation}
		|\a(v_h, w_h) - \tilde{a}(v_h, w_h)|
		\leq
		\|\sfA - \tilde{\sfA}\|_{\max}
		\sum_{i}
		\sum_{j\in\mcI(i)}
		|\sfv_i|\sspace|\sfw_j|
		\,.
	\end{equation}
	This, in turn, can only be finessed to arrive at an inequality like
	\begin{equation}
		|\a(v_h, w_h) - \tilde{a}(v_h, w_h)|
		\leq
		C_3^\prime
		h^{-n}
		\|\sfA - \tilde{\sfA}\|_{\max}
		\|v_h\|_0 \|w_h\|_0
		\,,
	\label{eq:FiniteDifferenceBoundAlternative}
	\end{equation}
	for some other constant $C_3^\prime$, depending only on \changed{$\bvarphi$} and $p$.
	Notice the loss of an $h^2$ scaling factor when comparing~\cref{eq:FiniteDifferenceBound,eq:FiniteDifferenceBoundAlternative}.
	\changed{
	This difference could greatly affect solution accuracy.
	}
\end{remark}

\subsection{\textit{A priori} error estimation} %
\label{sub:poisson_s_equation}

Define $V_{h,0} = V_h \cap H^1_0(\Omega)$.
The following lemma is identical in spirit to \cite[Theorem~7.1]{drzisga2018surrogate}.
By the Lax--Milgram theorem, this lemma allows us to conclude that there exists a unique surrogate solution corresponding to~\cref{eq:PoissonWeakForm}, namely $\tilde{u}_h\in V_{h,0}$, for every sufficiently small $H>0$.
\begin{lemma}
\label{cor:Coercivity}
	Let $\alpha = (1+C_P)^{-1}$, where $C_P$ is the Poincar\'e constant for the domain $\Omega$.
	If $H^{q+1}<{\alpha}\cdot C_4^{-1}$, then the surrogate bilinear form $\tilde{a}:V_h\times V_h \to \R$ is coercive on $V_{h,0}$.
	Letting $\tilde{\alpha}>0$ be the associated coercivity constant, it also holds that $\tilde{\alpha}\to\alpha$, as $H\to 0$.
\end{lemma}
\begin{proof}
	Let $S = \{v\in H^1_0(\Omega) \,:\, \|v\|_1 = 1\}$ be the surface of the unit ball in $H^1_0(\Omega)$.
	Notice that $\alpha \leq a(v_h,v_h) \leq \tilde{a}(v_h,v_h) + |a(v_h,v_h) - \tilde{a}(v_h,v_h)|$ for all $v_h \in V_h\cap S$ and, therefore,
	\begin{equation}
		\alpha
		-
		|a(v_h,v_h) - \tilde{a}(v_h,v_h)|
		\leq
		\tilde{a}(v_h,v_h)
		\quad
		\text{for all } v_h \in V_h\cap S
		\,.
	\label{eq:Coercivity1}
	\end{equation}
	Invoking \cref{lem:FiniteDifference}, the second term on the left may be bounded from above as follows:
	\begin{equation}
		|a(v_h,v_h) - \tilde{a}(v_h,v_h)|
		\leq
		C_4 H^{q+1}
		\,.
	\label{eq:Coercivity2}
	\end{equation}
	Clearly, if $C_4 H^{q+1}< {\alpha}$, then $0 < \alpha - |a(v_h,v_h) - \tilde{a}(v_h,v_h)| \leq \tilde{a}(v_h,v_h)$, as necessary.
\end{proof}

\begin{theorem}
\label{thm:APrioriBounds}
	Let $\theta >1$.
	If $u\in H^{p+1}(\Omega)$, then there exists a constant $C_5$, depending only on $p$ and \changed{$\bvarphi$}, such that
	\begin{subequations}
	\begin{equation}
		\|u - \approxsol{u}\|_{1}
		\leq
		C_5h^{p}\|u\|_{p+1} + \theta\cdot\alpha^{-1}C_4H^{q+1} \|\nabla u\|_0
		\,,
	\label{APrioriBoundH1}
	\end{equation}
	for every sufficiently small $H>0$.
	Additionally, if $\Omega\subset \R^n$ is convex, then there exists a constant $C_6$, depending only on $p$ and \changed{$\bvarphi$}, such that
	\begin{equation}
		\|u - \approxsol{u}\|_0
		\leq
		C_6h^{p+1}\|u\|_{p+1} + \theta\cdot C_P C_4H^{q+1} \|\nabla u\|_0
		\,,
	\label{eq:APrioriBoundL2}
	\end{equation}
	\end{subequations}
	for every sufficiently small $H>0$.
\end{theorem}

\begin{proof}
	We first prove~\cref{APrioriBoundH1}.
	Let $u_h\in V_{h,0}$ be the solution of the standard IGA discretization~\cref{eq:DiscreteVF} associated to~\cref{eq:PoissonWeakForm}.
	Clearly, $\|u - \tilde{u}_h\|_1 \leq \|u - u_h\|_1 + \|u_h - \tilde{u}_h\|_1$.
	Moreover, by \cite[Theorem~6.1]{da2014mathematical}, $\|u - u_h\|_1 \leq C_5 h^{p}\|u\|_{p+1}$.
	Recalling \cref{cor:Coercivity}, we find that
	\begin{equation}
		\tilde{\alpha}
		\|u_h - \tilde{u}_h\|_1^2
		\leq
		\tilde{a}(u_h - \tilde{u}_h,u_h - \tilde{u}_h)
		=
		\tilde{a}(u_h - \tilde{u}_h,u_h)
		-
		a(u_h - \tilde{u}_h,u_h)
		\,.
	\end{equation}
	After invoking~\cref{thm:BilinearFormDifferencePoisson}, it now readily follows that $\|u_h - \tilde{u}_h\|_1 \leq \tilde{\alpha}^{-1}C_4 H^{q+1} \|\nabla u\|_0$.
	Since $\tilde{\alpha}\to\alpha$, in the limit $H\to 0$, it also holds that $\|u_h - \tilde{u}_h\|_1 \leq \theta\cdot\alpha^{-1}C_4 H^{q+1} \|\nabla u\|_0$, for all sufficiently small $H>0$.
	
	Our proof of~\cref{eq:APrioriBoundL2}, also involves the triangle inequality: $\|u-\tilde{u}_h\|_0 \leq \|u-u_h\|_0 + \|u_h-\tilde{u}_h\|_0$.
	If $\Omega$ is convex, then $\|u-u_h\|_0 \leq C_6 h^{p+1}\|u\|_{p+1}$.
	Next, find $w_h\in V_{h,0}$ satisfying $a(w_h,v_h) = (u_h-\tilde{u}_h,v_h)_\Omega$, for all $v_h\in V_{h,0}$.
	It holds that $\|\nabla w_h\|_0 \leq C_P \|u_h-\tilde{u}_h\|_0$.
	Now, observe that
		$\|u_h-\tilde{u}_h\|_0^2
		=
		a(w_h,u_h-\tilde{u}_h)
		=
		\tilde{a}(w_h,\tilde{u}_h) - a(w_h,\tilde{u}_h)
		$.
	Finally, invoke~\cref{thm:BilinearFormDifferencePoisson} to arrive at
	\begin{equation}
		\|u_h-\tilde{u}_h\|_0^2
		\leq
		C_4
		H^{q+1} \|\nabla \tilde{u}_h\|_{0}\|\nabla w_h\|_0
		\leq
		C_P C_4
		H^{q+1} \|\nabla \tilde{u}_h\|_{0} \|u_h-\tilde{u}_h\|_{0}
		\,,
	\end{equation}
	which works to deliver the stated estimate, since $\|\nabla \tilde{u}_h\|_{0}\to\|\nabla u_h\|_0 \leq \|\nabla u\|_0$, as $H\to 0$.
\end{proof}

\subsection{Numerical experiments} %
\label{sub:poisson_numerical_experiments}

\begin{figure}
  \centering
    \begin{subfigure}[c]{0.35\textwidth}
        \includegraphics[trim=1.5cm 0cm 1.5cm 0cm,clip=true,height=4.5cm]{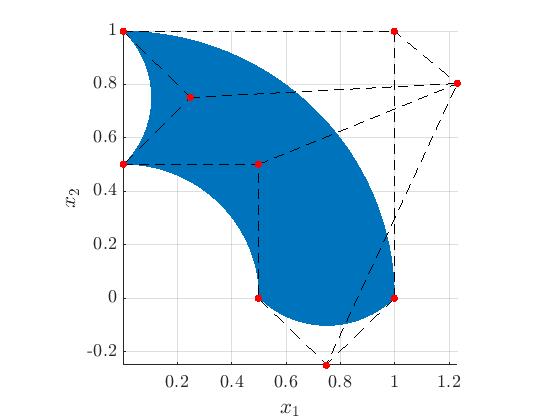}
        \caption{
        \label{fig:PoissonGeometries2D}
        2D domain
        }
    \end{subfigure}
    \qquad
    \begin{subfigure}[c]{0.35\textwidth}
        \includegraphics[trim=1cm 0cm 1.5cm 0cm,clip=true,height=4.5cm]{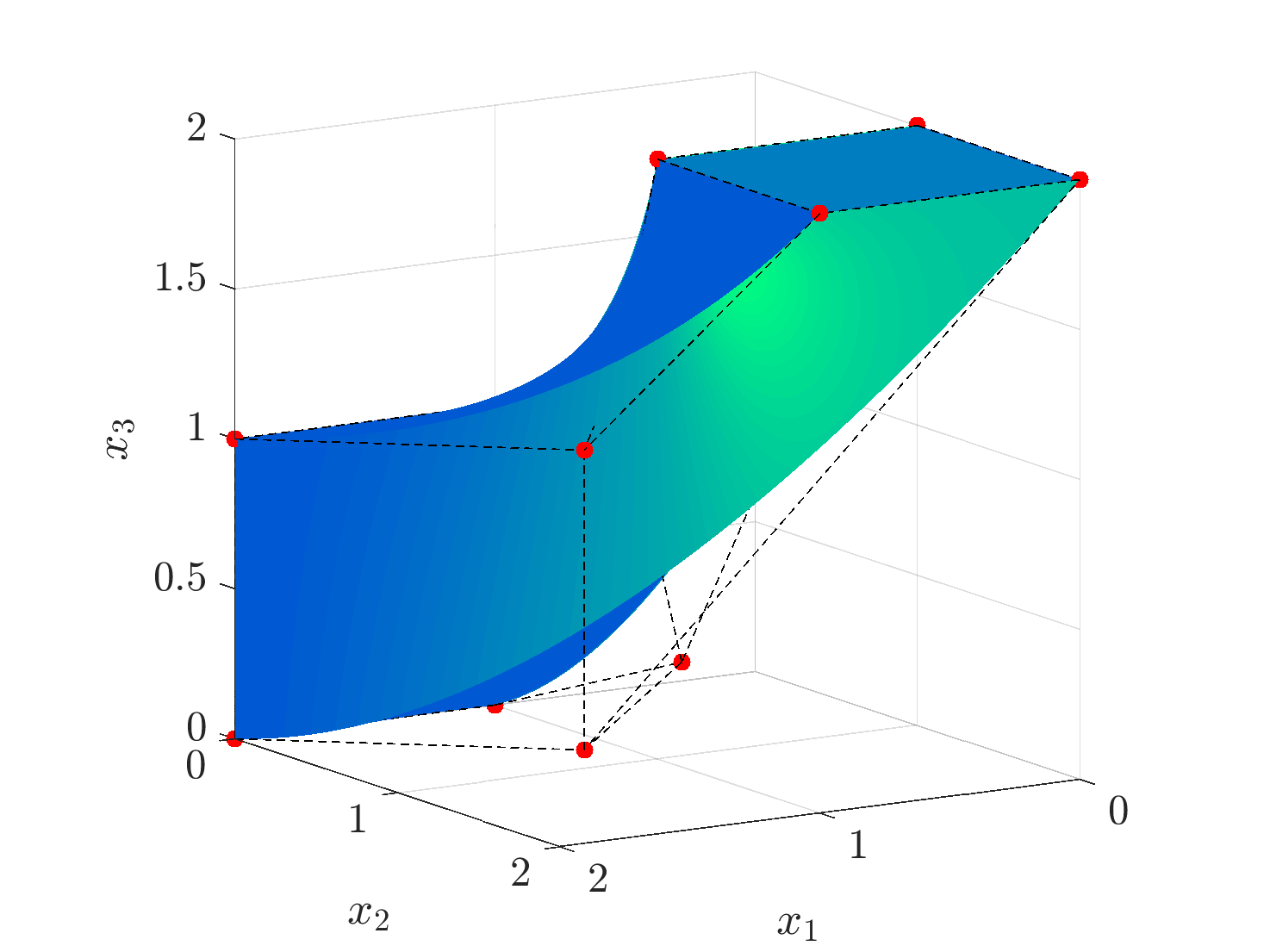}
        \caption{
        \label{fig:PoissonGeometries3D}
        3D domain
        }
    \end{subfigure}
    \caption{\label{fig:PoissonGeometries}Two domains $\Omega$ for problem~\cref{eq:PoissonWeakForm}.}
\end{figure}

The two bounds~\cref{APrioriBoundH1,eq:APrioriBoundL2} need to be experimentally verified.
Because both estimates depend on the two scales $h$ and $H$, in order to be overtly thorough, we should provide verification in both scales, independently.
That is, first holding the $h$-scale fixed and varying $H$ and, alternatively, holding the $H$-scale fixed and varying $h$.
We forgo this verification step and instead refer the reader to similar studies with low-order finite elements in \cite{bauer2017two,bauer2018large,bauer2018new,drzisga2018surrogate}.
In this section, therefore, we verify the bounds above under the assumption $H = H(h)$.
We expect that this would be the typical use case.

\subsubsection{Overview} %
\label{ssub:overview}

In our experiments with Poisson's equation~\cref{eq:PoissonWeakForm}, we considered a $2$D and a $3$D domain $\Omega$; see \Cref{fig:PoissonGeometries}.
The $2$D domain is given by a second-order NURBS parameterization and the $3$D domain is parameterized by second-order B-splines.
This particular $2$D domain (\Cref{fig:PoissonGeometries2D}) was chosen instead of the $2$D quarter annulus featured later (\Cref{fig:EigenvalueGeometry}) in order to break some symmetries which would otherwise appear in its stencil functions (\Cref{fig:SurrogateStencilFunctions}).
(Additional stencil function symmetries can appear if $\bvarphi\colon\hat{\Omega}\to\Omega$ is a conformal mapping.)
We also note that, after fixing the NURBS parameterizations for the edges on the boundary of the $2$D domain, the parameterization $\bvarphi\colon\hat{\Omega}\to\Omega$ was generated as a Coons patch \cite{piegl2012nurbs}.
Because each of the four edges had polynomial weight functions in their parameterizations, the resulting Coons patch NURBS parameterization, $\bvarphi\colon\hat{\Omega}\to\Omega$, contained a global polynomial weight function $W(\hat{\bmx}) = \sum_j w_j\hat{B}_j(\hat{\bmx})$.
Consequently, the assumptions of \Cref{thm:RegularityOfStencilFunctions} were satisfied in all of our experiments.

\begin{figure}%
\begin{minipage}{\textwidth}
\centering
\begin{minipage}{0.32\textwidth}
\begin{scaletikzpicturetowidth}{\textwidth}
\begin{tikzpicture}[scale=\tikzscale,font=\large]
\begin{loglogaxis}[
xlabel={Degrees of freedom},
ylabel={Relative $H^1$ error},
xmajorgrids,
ymajorgrids,
legend style={fill=white, fill opacity=0.6, draw opacity=1, text opacity=1},
legend pos=north east
]
\addplot[black, mark=*, very thick, mark options={scale=1.5}, fill opacity=0.6] table [x index = {0}, y index={1}, col sep=comma] {./Results/PoissonStatic2D/h1errors_p2.csv};
\addlegendentry{$M=1$};
\addplot[color2, mark=diamond*, very thick] table [x index = {0}, y index={2}, col sep=comma] {./Results/PoissonStatic2D/h1errors_p2.csv};
\addlegendentry{$q=1$};
\addplot[color3, mark=diamond*, very thick] table [x index = {0}, y index={3}, col sep=comma] {./Results/PoissonStatic2D/h1errors_p2.csv};
\addlegendentry{$q=2$};
\addplot[color4, mark=diamond*, very thick] table [x index = {0}, y index={4}, col sep=comma] {./Results/PoissonStatic2D/h1errors_p2.csv};
\addlegendentry{$q=3$};
\addplot[color5, mark=diamond*, very thick] table [x index = {0}, y index={5}, col sep=comma] {./Results/PoissonStatic2D/h1errors_p2.csv};
\addlegendentry{$q=4$};
\addplot[color6, mark=diamond*, very thick] table [x index = {0}, y index={6}, col sep=comma] {./Results/PoissonStatic2D/h1errors_p2.csv};
\addlegendentry{$q=5$};
\logLogSlopeTriangle{0.7}{0.2}{0.15}{2}{2}{};
\legend{}; %
\end{loglogaxis}
\end{tikzpicture} \end{scaletikzpicturetowidth}
\end{minipage}
\qquad
\begin{minipage}{0.32\textwidth}
\begin{scaletikzpicturetowidth}{\textwidth}
\begin{tikzpicture}[scale=\tikzscale,font=\large]
\begin{loglogaxis}[
xlabel={Degrees of freedom},
ylabel={Relative $L^2$ error},
xmajorgrids,
ymajorgrids,
legend style={fill=white, fill opacity=0.6, draw opacity=1, text opacity=1},
legend pos=north east
]
\addplot[black, mark=*, very thick, mark options={scale=1.5}, fill opacity=0.6] table [x index = {0}, y index={1}, col sep=comma] {./Results/PoissonStatic2D/l2errors_p2.csv};
\addlegendentry{$M=1$};
\addplot[color2, mark=diamond*, very thick] table [x index = {0}, y index={2}, col sep=comma] {./Results/PoissonStatic2D/l2errors_p2.csv};
\addlegendentry{$q=1$};
\addplot[color3, mark=diamond*, very thick] table [x index = {0}, y index={3}, col sep=comma] {./Results/PoissonStatic2D/l2errors_p2.csv};
\addlegendentry{$q=2$};
\addplot[color4, mark=diamond*, very thick] table [x index = {0}, y index={4}, col sep=comma] {./Results/PoissonStatic2D/l2errors_p2.csv};
\addlegendentry{$q=3$};
\addplot[color5, mark=diamond*, very thick] table [x index = {0}, y index={5}, col sep=comma] {./Results/PoissonStatic2D/l2errors_p2.csv};
\addlegendentry{$q=4$};
\addplot[color6, mark=diamond*, very thick] table [x index = {0}, y index={6}, col sep=comma] {./Results/PoissonStatic2D/l2errors_p2.csv};
\addlegendentry{$q=5$};
\logLogSlopeTriangle{0.7}{0.2}{0.15}{2}{3}{};
\legend{}; %
\end{loglogaxis}
\end{tikzpicture} \end{scaletikzpicturetowidth}
\end{minipage}
\end{minipage}
\\
\begin{minipage}{\textwidth}
\centering
\begin{minipage}{0.32\textwidth}
\begin{scaletikzpicturetowidth}{\textwidth}
\begin{tikzpicture}[scale=\tikzscale,font=\large]
\begin{loglogaxis}[
xlabel={Degrees of freedom},
ylabel={Relative $H^1$ error},
xmajorgrids,
ymajorgrids,
legend style={fill=white, fill opacity=0.6, draw opacity=1, text opacity=1},
legend pos=north east
]
\addplot[black, mark=*, very thick, mark options={scale=1.5}, fill opacity=0.6] table [x index = {0}, y index={1}, col sep=comma] {./Results/PoissonStatic2DHighFrequency/h1errors_p2.csv};
\addlegendentry{$M=1$};
\addplot[color2, mark=diamond*, very thick] table [x index = {0}, y index={2}, col sep=comma] {./Results/PoissonStatic2DHighFrequency/h1errors_p2.csv};
\addlegendentry{$q=1$};
\addplot[color3, mark=diamond*, very thick] table [x index = {0}, y index={3}, col sep=comma] {./Results/PoissonStatic2DHighFrequency/h1errors_p2.csv};
\addlegendentry{$q=2$};
\addplot[color4, mark=diamond*, very thick] table [x index = {0}, y index={4}, col sep=comma] {./Results/PoissonStatic2DHighFrequency/h1errors_p2.csv};
\addlegendentry{$q=3$};
\addplot[color5, mark=diamond*, very thick] table [x index = {0}, y index={5}, col sep=comma] {./Results/PoissonStatic2DHighFrequency/h1errors_p2.csv};
\addlegendentry{$q=4$};
\addplot[color6, mark=diamond*, very thick] table [x index = {0}, y index={6}, col sep=comma] {./Results/PoissonStatic2DHighFrequency/h1errors_p2.csv};
\addlegendentry{$q=5$};
\logLogSlopeTriangle{0.7}{0.2}{0.17}{2}{2}{};
\legend{}; %
\end{loglogaxis}
\end{tikzpicture} \end{scaletikzpicturetowidth}
\end{minipage}
\qquad
\begin{minipage}{0.32\textwidth}
\begin{scaletikzpicturetowidth}{\textwidth}
\begin{tikzpicture}[scale=\tikzscale,font=\large]
\begin{loglogaxis}[
xlabel={Degrees of freedom},
ylabel={Relative $L^2$ error},
xmajorgrids,
ymajorgrids,
legend style={fill=white, fill opacity=0.6, draw opacity=1, text opacity=1},
legend pos=north east
]
\addplot[black, mark=*, very thick, mark options={scale=1.5}, fill opacity=0.6] table [x index = {0}, y index={1}, col sep=comma] {./Results/PoissonStatic2DHighFrequency/l2errors_p2.csv};
\addlegendentry{$M=1$};
\addplot[color2, mark=diamond*, very thick] table [x index = {0}, y index={2}, col sep=comma] {./Results/PoissonStatic2DHighFrequency/l2errors_p2.csv};
\addlegendentry{$q=1$};
\addplot[color3, mark=diamond*, very thick] table [x index = {0}, y index={3}, col sep=comma] {./Results/PoissonStatic2DHighFrequency/l2errors_p2.csv};
\addlegendentry{$q=2$};
\addplot[color4, mark=diamond*, very thick] table [x index = {0}, y index={4}, col sep=comma] {./Results/PoissonStatic2DHighFrequency/l2errors_p2.csv};
\addlegendentry{$q=3$};
\addplot[color5, mark=diamond*, very thick] table [x index = {0}, y index={5}, col sep=comma] {./Results/PoissonStatic2DHighFrequency/l2errors_p2.csv};
\addlegendentry{$q=4$};
\addplot[color6, mark=diamond*, very thick] table [x index = {0}, y index={6}, col sep=comma] {./Results/PoissonStatic2DHighFrequency/l2errors_p2.csv};
\addlegendentry{$q=5$};
\logLogSlopeTriangle{0.7}{0.2}{0.17}{2}{3}{};
\end{loglogaxis}
\end{tikzpicture} \end{scaletikzpicturetowidth}
\end{minipage}
\end{minipage}
\caption{\label{fig:varying_hs_p2} $2$D domain. Relative errors for $p=2$, $M=5$, and two different manufactured solutions. Top row: $u(\bmx) = \sin\left(\pi \,x_1\right)\,\sin\left(\pi \,x_2\right)$. Bottom row: $u(\bmx) = \sin\left(20\,\pi \,x_1\right)\,\sin\left(20\,\pi \,x_2\right)$.}
\end{figure}

In constructing $\tilde{\sfA}$ (cf. \Cref{app:computing_surrogates_with_existing_iga_codes}) and, thus, the surrogate stencil functions $\tilde{\Phi}_\bdelta = \Pi_{H}\Phi_\bdelta$, we tried local B-spline interpolants $\Pi_{H}\colon C^0({\tilde{\Omega}})\to S_q(\tilde{\bm{\Xi}})$ of various order $1\leq q\leq 5$.
As stated previously (see \Cref{rem:SamplingPoints}), we used the convention $\tilde{\bm{\Xi}}\subset\tilde{\Omega}\cap\tilde{\bbX}$.
In our $2$D experiments, we explored both $p=2$ or $p=3$ NURBS bases and the full range of possible $q$.
In our $3$D experiments, we only explored the case $p=2$ using $q=1,3$.
This limitation of our $3$D experiments was due to the restraints of the MATLAB B-spline interpolation routine we were using.

In the first set of experiments (\Cref{fig:varying_hs_p2,fig:varying_hs_p3,fig:varying_hs_p2_3D}), the sampling length $H$ was set to $H=Mh$, with the constant $M=5$.
In the second set of experiments (\Cref{fig:dynamic_sampling_p2_timing,fig:dynamic_sampling_p2}), we used a mesh-dependent sampling length, defined below.

\subsubsection{Constant sampling lengths} %
\label{subs:constant_sampling_length}

\begin{figure}%
\begin{minipage}{\textwidth}
\centering
\begin{minipage}{0.32\textwidth}
\begin{scaletikzpicturetowidth}{\textwidth}
\begin{tikzpicture}[scale=\tikzscale,font=\large]
\begin{loglogaxis}[
xlabel={Degrees of freedom},
ylabel={Relative $H^1$ error},
xmajorgrids,
ymajorgrids,
legend style={fill=white, fill opacity=0.6, draw opacity=1, text opacity=1},
legend pos=north east
]
\addplot[black, mark=*, very thick, mark options={scale=1.5}, fill opacity=0.6] table [x index = {0}, y index={1}, col sep=comma] {./Results/PoissonStatic2D/h1errors_p3.csv};
\addlegendentry{$M=1$};
\addplot[color3, mark=diamond*, very thick] table [x index = {0}, y index={3}, col sep=comma] {./Results/PoissonStatic2D/h1errors_p3.csv};
\addlegendentry{$q=2$};
\addplot[color4, mark=diamond*, very thick] table [x index = {0}, y index={4}, col sep=comma] {./Results/PoissonStatic2D/h1errors_p3.csv};
\addlegendentry{$q=3$};
\addplot[color5, mark=diamond*, very thick] table [x index = {0}, y index={5}, col sep=comma] {./Results/PoissonStatic2D/h1errors_p3.csv};
\addlegendentry{$q=4$};
\addplot[color6, mark=diamond*, very thick] table [x index = {0}, y index={6}, col sep=comma] {./Results/PoissonStatic2D/h1errors_p3.csv};
\addlegendentry{$q=5$};
\logLogSlopeTriangle{0.7}{0.2}{0.15}{2}{3}{};
\legend{}; %
\end{loglogaxis}
\end{tikzpicture} \end{scaletikzpicturetowidth}
\end{minipage}
\qquad
\begin{minipage}{0.32\textwidth}
\begin{scaletikzpicturetowidth}{\textwidth}
\begin{tikzpicture}[scale=\tikzscale,font=\large]
\begin{loglogaxis}[
xlabel={Degrees of freedom},
ylabel={Relative $L^2$ error},
xmajorgrids,
ymajorgrids,
legend style={fill=white, fill opacity=0.6, draw opacity=1, text opacity=1},
legend pos=north east
]
\addplot[black, mark=*, very thick, mark options={scale=1.5}, fill opacity=0.6] table [x index = {0}, y index={1}, col sep=comma] {./Results/PoissonStatic2D/l2errors_p3.csv};
\addlegendentry{$M=1$};
\addplot[color3, mark=diamond*, very thick] table [x index = {0}, y index={3}, col sep=comma] {./Results/PoissonStatic2D/l2errors_p3.csv};
\addlegendentry{$q=2$};
\addplot[color4, mark=diamond*, very thick] table [x index = {0}, y index={4}, col sep=comma] {./Results/PoissonStatic2D/l2errors_p3.csv};
\addlegendentry{$q=3$};
\addplot[color5, mark=diamond*, very thick] table [x index = {0}, y index={5}, col sep=comma] {./Results/PoissonStatic2D/l2errors_p3.csv};
\addlegendentry{$q=4$};
\addplot[color6, mark=diamond*, very thick] table [x index = {0}, y index={6}, col sep=comma] {./Results/PoissonStatic2D/l2errors_p3.csv};
\addlegendentry{$q=5$};
\logLogSlopeTriangle{0.7}{0.2}{0.15}{2}{4}{};
\legend{}; %
\end{loglogaxis}
\end{tikzpicture} \end{scaletikzpicturetowidth}
\end{minipage}
\end{minipage}
\\
\begin{minipage}{\textwidth}
\centering
\begin{minipage}{0.32\textwidth}
\begin{scaletikzpicturetowidth}{\textwidth}
\begin{tikzpicture}[scale=\tikzscale,font=\large]
\begin{loglogaxis}[
xlabel={Degrees of freedom},
ylabel={Relative $H^1$ error},
xmajorgrids,
ymajorgrids,
legend style={fill=white, fill opacity=0.6, draw opacity=1, text opacity=1},
legend pos=north east
]
\addplot[black, mark=*, very thick, mark options={scale=1.5}, fill opacity=0.6] table [x index = {0}, y index={1}, col sep=comma] {./Results/PoissonStatic2DHighFrequency/h1errors_p3.csv};
\addlegendentry{$M=1$};
\addplot[color3, mark=diamond*, very thick] table [x index = {0}, y index={3}, col sep=comma] {./Results/PoissonStatic2DHighFrequency/h1errors_p3.csv};
\addlegendentry{$q=2$};
\addplot[color4, mark=diamond*, very thick] table [x index = {0}, y index={4}, col sep=comma] {./Results/PoissonStatic2DHighFrequency/h1errors_p3.csv};
\addlegendentry{$q=3$};
\addplot[color5, mark=diamond*, very thick] table [x index = {0}, y index={5}, col sep=comma] {./Results/PoissonStatic2DHighFrequency/h1errors_p3.csv};
\addlegendentry{$q=4$};
\addplot[color6, mark=diamond*, very thick] table [x index = {0}, y index={6}, col sep=comma] {./Results/PoissonStatic2DHighFrequency/h1errors_p3.csv};
\addlegendentry{$q=5$};
\logLogSlopeTriangle{0.7}{0.2}{0.17}{2}{3}{};
\legend{}; %
\end{loglogaxis}
\end{tikzpicture} \end{scaletikzpicturetowidth}
\end{minipage}
\qquad
\begin{minipage}{0.32\textwidth}
\begin{scaletikzpicturetowidth}{\textwidth}
\begin{tikzpicture}[scale=\tikzscale,font=\large]
\begin{loglogaxis}[
xlabel={Degrees of freedom},
ylabel={Relative $L^2$ error},
xmajorgrids,
ymajorgrids,
legend style={fill=white, fill opacity=0.6, draw opacity=1, text opacity=1},
legend pos=north east
]
\addplot[black, mark=*, very thick, mark options={scale=1.5}, fill opacity=0.6] table [x index = {0}, y index={1}, col sep=comma] {./Results/PoissonStatic2DHighFrequency/l2errors_p3.csv};
\addlegendentry{$M=1$};
\addplot[color3, mark=diamond*, very thick] table [x index = {0}, y index={3}, col sep=comma] {./Results/PoissonStatic2DHighFrequency/l2errors_p3.csv};
\addlegendentry{$q=2$};
\addplot[color4, mark=diamond*, very thick] table [x index = {0}, y index={4}, col sep=comma] {./Results/PoissonStatic2DHighFrequency/l2errors_p3.csv};
\addlegendentry{$q=3$};
\addplot[color5, mark=diamond*, very thick] table [x index = {0}, y index={5}, col sep=comma] {./Results/PoissonStatic2DHighFrequency/l2errors_p3.csv};
\addlegendentry{$q=4$};
\addplot[color6, mark=diamond*, very thick] table [x index = {0}, y index={6}, col sep=comma] {./Results/PoissonStatic2DHighFrequency/l2errors_p3.csv};
\addlegendentry{$q=5$};
\logLogSlopeTriangle{0.7}{0.2}{0.17}{2}{4}{};
\end{loglogaxis}
\end{tikzpicture} \end{scaletikzpicturetowidth}
\end{minipage}
\end{minipage}
\caption{\label{fig:varying_hs_p3} $2$D domain. Relative errors for $p=3$, $M=5$, and two different manufactured solutions. Top row: $u(\bmx) = \sin\left(\pi \,x_1\right)\,\sin\left(\pi \,x_2\right)$. Bottom row: $u(\bmx) = \sin\left(20\,\pi \,x_1\right)\,\sin\left(20\,\pi \,x_2\right)$.}
\end{figure}

\begin{figure}%
\begin{minipage}{\textwidth}
\centering
\begin{minipage}{0.32\textwidth}
\begin{scaletikzpicturetowidth}{\textwidth}
\begin{tikzpicture}[scale=\tikzscale,font=\large]
\begin{loglogaxis}[
xlabel={Degrees of freedom},
ylabel={Relative $H^1$ error},
xmajorgrids,
ymajorgrids,
legend style={fill=white, fill opacity=0.6, draw opacity=1, text opacity=1},
legend pos=north east
]
\addplot[black, mark=*, very thick, mark options={scale=1.5}, fill opacity=0.6] table [x index = {0}, y index={1}, col sep=comma] {./Results/PoissonStatic3D/h1errors_p2.csv};
\addlegendentry{$M=1$};
\addplot[color2, mark=diamond*, very thick] table [x index = {0}, y index={2}, col sep=comma] {./Results/PoissonStatic3D/h1errors_p2.csv};
\addlegendentry{$q=1$};
\addplot[color4, mark=diamond*, very thick] table [x index = {0}, y index={3}, col sep=comma] {./Results/PoissonStatic3D/h1errors_p2.csv};
\addlegendentry{$q=3$};
\logLogSlopeTriangle{0.7}{0.2}{0.17}{3}{2}{};
\legend{}; %
\end{loglogaxis}
\end{tikzpicture}
 \end{scaletikzpicturetowidth}
\end{minipage}
\qquad
\begin{minipage}{0.32\textwidth}
\begin{scaletikzpicturetowidth}{\textwidth}
\begin{tikzpicture}[scale=\tikzscale,font=\large]
\begin{loglogaxis}[
xlabel={Degrees of freedom},
ylabel={Relative $L^2$ error},
xmajorgrids,
ymajorgrids,
legend style={fill=white, fill opacity=0.6, draw opacity=1, text opacity=1},
legend pos=north east
]
\addplot[black, mark=*, very thick, mark options={scale=1.5}, fill opacity=0.6] table [x index = {0}, y index={1}, col sep=comma] {./Results/PoissonStatic3D/l2errors_p2.csv};
\addlegendentry{$M=1$};
\addplot[color2, mark=diamond*, very thick] table [x index = {0}, y index={2}, col sep=comma] {./Results/PoissonStatic3D/l2errors_p2.csv};
\addlegendentry{$q=1$};
\addplot[color4, mark=diamond*, very thick] table [x index = {0}, y index={3}, col sep=comma] {./Results/PoissonStatic3D/l2errors_p2.csv};
\addlegendentry{$q=3$};
\logLogSlopeTriangle{0.7}{0.2}{0.17}{3}{3}{};
\legend{}; %
\end{loglogaxis}
\end{tikzpicture}
 \end{scaletikzpicturetowidth}
\end{minipage}
\end{minipage}
\\
\begin{minipage}{\textwidth}
\centering
\begin{minipage}{0.32\textwidth}
\begin{scaletikzpicturetowidth}{\textwidth}
\begin{tikzpicture}[scale=\tikzscale,font=\large]
\begin{loglogaxis}[
xlabel={Degrees of freedom},
ylabel={Relative $H^1$ error},
xmajorgrids,
ymajorgrids,
legend style={fill=white, fill opacity=0.6, draw opacity=1, text opacity=1},
legend pos=north east
]
\addplot[black, mark=*, very thick, mark options={scale=1.5}, fill opacity=0.6] table [x index = {0}, y index={1}, col sep=comma] {./Results/PoissonStatic3DHighFrequency/h1errors_p2.csv};
\addlegendentry{$M=1$};
\addplot[color2, mark=diamond*, very thick] table [x index = {0}, y index={2}, col sep=comma] {./Results/PoissonStatic3DHighFrequency/h1errors_p2.csv};
\addlegendentry{$q=1$};
\addplot[color4, mark=diamond*, very thick] table [x index = {0}, y index={3}, col sep=comma] {./Results/PoissonStatic3DHighFrequency/h1errors_p2.csv};
\addlegendentry{$q=3$};
\logLogSlopeTriangle{0.8}{0.2}{0.15}{3}{2}{};
\legend{}; %
\end{loglogaxis}
\end{tikzpicture}
 \end{scaletikzpicturetowidth}
\end{minipage}
\qquad
\begin{minipage}{0.32\textwidth}
\begin{scaletikzpicturetowidth}{\textwidth}
\begin{tikzpicture}[scale=\tikzscale,font=\large]
\begin{loglogaxis}[
xlabel={Degrees of freedom},
ylabel={Relative $L^2$ error},
xmajorgrids,
ymajorgrids,
legend style={fill=white, fill opacity=0.6, draw opacity=1, text opacity=1},
legend pos=north east
]
\addplot[black, mark=*, very thick, mark options={scale=1.5}, fill opacity=0.6] table [x index = {0}, y index={1}, col sep=comma] {./Results/PoissonStatic3DHighFrequency/l2errors_p2.csv};
\addlegendentry{$M=1$};
\addplot[color2, mark=diamond*, very thick] table [x index = {0}, y index={2}, col sep=comma] {./Results/PoissonStatic3DHighFrequency/l2errors_p2.csv};
\addlegendentry{$q=1$};
\addplot[color4, mark=diamond*, very thick] table [x index = {0}, y index={3}, col sep=comma] {./Results/PoissonStatic3DHighFrequency/l2errors_p2.csv};
\addlegendentry{$q=3$};
\logLogSlopeTriangle{0.8}{0.2}{0.15}{3}{3}{};
\end{loglogaxis}
\end{tikzpicture}
 \end{scaletikzpicturetowidth}
\end{minipage}
\end{minipage}
\caption{\label{fig:varying_hs_p2_3D} $3$D domain. Relative errors for $p=2$, $M=5$, and two different manufactured solutions. Top row: $u(\bmx) = \sin\left(\pi \,x_1\right)\,\sin\left(\pi \,x_2\right)\,\sin\left(\pi \,x_3\right)$. Bottom row: ${u(\bmx) = \sin\left(20\,\pi \,x_1\right)\,\sin\left(20\,\pi \,x_2\right)\,\sin\left(20\,\pi \,x_3\right)}$.}
\end{figure}

For the constant sampling strategy, with $M=5$, convergence plots of the errors in surrogate solutions are presented 
in \Cref{fig:varying_hs_p2,fig:varying_hs_p3,fig:varying_hs_p2_3D}.
After inspecting these figures, the reader should observe that there is a noticeable difference in the accuracy of the surrogate solutions $\tilde{u}_h$, dependent on the load $f$.
For instance, with the ``low-frequency'' manufactured solution, $u(\bmx) = \sin\left(\pi \,x_1\right)\,\sin\left(\pi \,x_2\right)$, the errors in the surrogate solutions appear to converge to the error generated by the (standard) non-surrogate IGA solution $u_h$ asymptotically, at various rates.
However, for the ``high-frequency'' manufactured solution, $u(\bmx) = \sin\left(20\,\pi \,x_1\right)\,\sin\left(20\,\pi \,x_2\right)$, each of the surrogate errors and the corresponding standard IGA error are nearly indistinguishable (except, sometimes, in the case $q=1$).

This difference can be explained in a simple way.
Observe that the $h$-dependent terms in~\cref{APrioriBoundH1,eq:APrioriBoundL2} are multiplied by the high-order norm $\|u\|_{p+1}$ and, meanwhile, the $H$-dependent terms are only multiplied by the lower-order seminorm $\|\nabla u\|_0$.
Let us call the first term in each bound the \emph{discretization error} and the second term the \emph{consistency error}.
Because of the presence of the high-order norm $\|u\|_{p+1}$, the discretization error is much more sensitive to irregularities or oscillations in the solution.
Another important detail not to overlook is that $C_4$ ultimately involves high-order norms of the tensor coefficient $K(\hat{\bmx})$.
Meanwhile, $C_5$ and $C_6$ depend (through $\Omega$), at most, on the $H^2$-norm of $K(\hat{\bmx})$.

These observations can lead in many interesting directions, but the conclusion which the reader should ultimately arrive at is that the assembly of the stiffness matrix should only be carried out to the accuracy required by the problem at hand.
Therefore, the correct surrogate assembly strategy must take into account properties of the problem geometry and the given loads, as well as each of the parameters $h$, $p$, and $q$.
A simple way of doing this is to use mesh-dependent sampling lengths.

\begin{remark}
\label{rem:parity}
In \cite{drzisga2018surrogate}, it was documented that the $L^2$ error converged like $h^{p+1} + H^{q+2}$ when $\Pi_H$ resembles an $L^2$ projection.
However, when using an interpolation instead, like it is done in this work, the error converges like $h^{p+1} + H^{q+2}$ whenever the interpolation order $q$ is even and like $h^{p+1} + H^{q+1}$ otherwise.
This parity effect can also be observed in results of this work; see, e.g., the top right plot in~\cref{fig:varying_hs_p3}.
We do not prove this improved convergence rate but refer to proofs of quadrature formulas where the same parity effect can be observed\revised{; see, e.g., \cite[Chapter 2.5]{davis2007methods}.}
\end{remark}

\subsubsection{Mesh-dependent sampling lengths} %
\label{ssub:dynamic_sampling_length}

In a successful surrogate method, the discretization error should always dominate the consistency error.
Of course, however, this domination should not be exercised to such an extent that overall efficiency is compromised.
As a rule of thumb, keeping the consistency error at or below $5\%$ of the discretization error is often acceptable.
Even within this somewhat conservative threshold, balancing the two errors appropriately can lead to very significant performance advantages.

Let $q>p$ and consider the mesh-dependent sampling length $M(h) = \max\Big\{1,\floor[\big]{c \cdot h^{\frac{p-q+\beta}{q+1}}}\Big\}$, where both $c,\beta\geq0$ are tuneable parameters.
Returning the definition $H = M\cdot h$, we now find that, for any $c>0$,
\begin{equation}
	H^{q+1}
	\leq
	(M\cdot h)^{q+1}
	=
	c^{q+1} h^{p+1+\beta}
	,
	\quad
	\text{as }
	h\to 0\,.
\end{equation}
Here, for any $\beta>0$, both the $H^1$ and $L^2$ discretization error will eventually dominate their associated consistency error.
Therefore, this parameter exists to maintain the dominance of the discretization error, throughout mesh refinements.
We found $\beta = \nicefrac{1}{2}$ to be a suitable choice for our purposes.
The parameter $c$, on the other hand, must be calibrated to the problem at hand.

All run-time measurements in this sections were obtained on a machine equipped with two Intel\textsuperscript{\textregistered} Xeon\textsuperscript{\textregistered} Gold 6136 processors with a nominal base frequency of 3.0 GHz.
Each processor has 12 physical cores which results in a total of 24 physical cores, but only single core was used to run the following experiments.
The total available memory of \SI{251}{\giga\byte} is split into two NUMA domains, one for each socket.
For compatibility reasons \revised{with using SciPy}, we employed the \revised{BLAS library in version 3.7.1} provided by the operating system \revised{Ubuntu 18.04.2}, but using other optimized libraries might improve the performance of the MATLAB solver.

In \Cref{fig:dynamic_sampling_p2_timing} we show the assembly times our implementation accrued for various values of $c$, when $p=2$ and $q=5$.
Inspection of \Cref{fig:dynamic_sampling_p2} clearly shows that $c = 0.75$ is suitable for $u(\bmx) = \sin\left(\pi \,x_1\right)\,\sin\left(\pi \,x_2\right)$ and $c = 3$ is suitable for $u(\bmx) = \sin\left(20\,\pi \,x_1\right)\,\sin\left(20\,\pi \,x_2\right)$.
\revised{Let $t_\mathrm{std}$ be the time required to assemble the matrix with the standard approach and $t_\mathrm{surr}$ the time using the surrogate approach. The speed-up from using the surrogate approach is then defined as $\frac{t_\mathrm{std}}{t_\mathrm{surr}} - 1$.}
When compared to the non-surrogate assembly strategy at just over one million degrees of freedom, note that the first choice gives an assembly speed-up of around $1500\%$ and the second choice gives a speed-up of more than $5000\%$.
These enormous speed-ups are in fact not that surprising after inspecting the percentage of matrix entries which must be computed using traditional quadrature formulas; see, again, \Cref{fig:dynamic_sampling_p2_timing}.

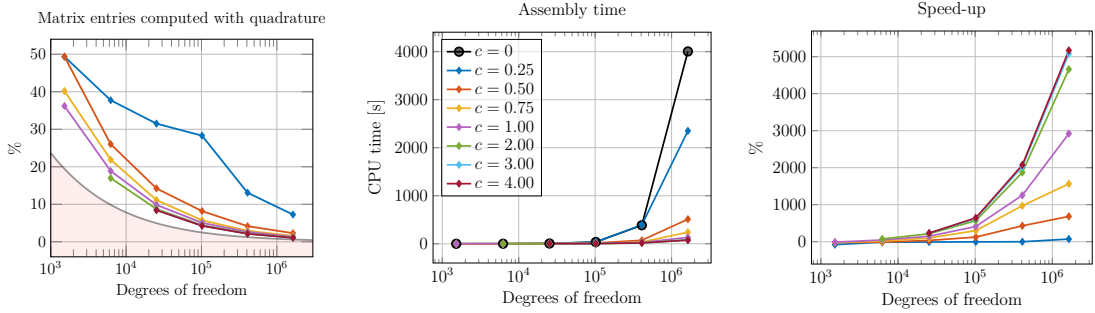
\begin{figure}%
\centering
\begin{minipage}{0.30\textwidth}
\begin{scaletikzpicturetowidth}{\textwidth}
\begin{tikzpicture}[font=\large]
\tikzstyle{every node}=[scale=\tikzscale]
\begin{semilogxaxis}[
scale=\tikzscale,
xlabel={Degrees of freedom},
ylabel={\%},
xmajorgrids,
ymajorgrids,
title={Matrix entries computed with quadrature},
enlargelimits=false,
xmin=1000,
xmax=3000000,
ymin=-4,
ymax=54,
/pgfplots/ytick={0,10,...,50},
every axis plot/.append style={line width=0.69pt}
]
\addplot[color1, mark=diamond*, mark options={scale=\tikzscale}] table [x index = {0}, y index={10}, col sep=comma] {./Results/PoissonDynamic2DHighFrequency/timing_p2_q5_cc0.250.csv};
\addplot[color2, mark=diamond*, mark options={scale=\tikzscale}] table [x index = {0}, y index={10}, col sep=comma] {./Results/PoissonDynamic2DHighFrequency/timing_p2_q5_cc0.500.csv};
\addplot[color3, mark=diamond*, mark options={scale=\tikzscale}] table [x index = {0}, y index={10}, col sep=comma] {./Results/PoissonDynamic2DHighFrequency/timing_p2_q5_cc0.750.csv};
\addplot[color4, mark=diamond*, mark options={scale=\tikzscale}] table [x index = {0}, y index={10}, col sep=comma] {./Results/PoissonDynamic2DHighFrequency/timing_p2_q5_cc1.000.csv};
\addplot[color5, mark=diamond*, mark options={scale=\tikzscale}] table [x index = {0}, y index={10}, col sep=comma] {./Results/PoissonDynamic2DHighFrequency/timing_p2_q5_cc2.000.csv};
\addplot[color6, mark=diamond*, mark options={scale=\tikzscale}] table [x index = {0}, y index={10}, col sep=comma] {./Results/PoissonDynamic2DHighFrequency/timing_p2_q5_cc3.000.csv};
\addplot[color7, mark=diamond*, mark options={scale=\tikzscale}] table [x index = {0}, y index={10}, col sep=comma] {./Results/PoissonDynamic2DHighFrequency/timing_p2_q5_cc4.000.csv};
\addplot[name path=f,domain=1000:3000000, samples=200,color=black,opacity=0.4]{100*(1 - ((x^(1/2) - 2*2)/(x)^(1/2))^2)};
\path[name path=axis] (axis cs:1000,-4) -- (axis cs:3000000,-4);
\addplot[thick,color=red,opacity=0.2,fill=red,fill opacity=0.07] fill between[of=f and axis, soft clip={domain=1000:3000000}];
\end{semilogxaxis}
\end{tikzpicture} \end{scaletikzpicturetowidth}
\end{minipage}
\hfill
\begin{minipage}{0.32\textwidth}
\begin{scaletikzpicturetowidth}{\textwidth}
\begin{tikzpicture}[scale=\tikzscale,font=\large]
\begin{semilogxaxis}[
xlabel={Degrees of freedom},
ylabel={CPU time [s]},
xmajorgrids,
ymajorgrids,
title={Assembly time},
legend style={fill=white, fill opacity=0.6, draw opacity=1, text opacity=1},
legend pos=north west,
/pgf/number format/1000 sep={}
]
\addplot[black, mark=*, very thick, mark options={scale=1.5}, fill opacity=0.6] table [x index = {0}, y index={1}, col sep=comma] {./Results/PoissonDynamic2DHighFrequency/rel/std_time.csv};
\addlegendentry{$c=0\phantom{.00}$};
\addplot[color1, mark=diamond*, very thick] table [x index = {0}, y index={6}, col sep=comma] {./Results/PoissonDynamic2DHighFrequency/timing_p2_q5_cc0.250.csv};
\addlegendentry{$c = 0.25$};
\addplot[color2, mark=diamond*, very thick] table [x index = {0}, y index={6}, col sep=comma] {./Results/PoissonDynamic2DHighFrequency/timing_p2_q5_cc0.500.csv};
\addlegendentry{$c = 0.50$};
\addplot[color3, mark=diamond*, very thick] table [x index = {0}, y index={6}, col sep=comma] {./Results/PoissonDynamic2DHighFrequency/timing_p2_q5_cc0.750.csv};
\addlegendentry{$c = 0.75$};
\addplot[color4, mark=diamond*, very thick] table [x index = {0}, y index={6}, col sep=comma] {./Results/PoissonDynamic2DHighFrequency/timing_p2_q5_cc1.000.csv};
\addlegendentry{$c = 1.00$};
\addplot[color5, mark=diamond*, very thick] table [x index = {0}, y index={6}, col sep=comma] {./Results/PoissonDynamic2DHighFrequency/timing_p2_q5_cc2.000.csv};
\addlegendentry{$c = 2.00$};
\addplot[color6, mark=diamond*, very thick] table [x index = {0}, y index={6}, col sep=comma] {./Results/PoissonDynamic2DHighFrequency/timing_p2_q5_cc3.000.csv};
\addlegendentry{$c = 3.00$};
\addplot[color7, mark=diamond*, very thick] table [x index = {0}, y index={6}, col sep=comma] {./Results/PoissonDynamic2DHighFrequency/timing_p2_q5_cc4.000.csv};
\addlegendentry{$c = 4.00$};
\end{semilogxaxis}
\end{tikzpicture}
 \end{scaletikzpicturetowidth}
\end{minipage}
\hfill
\begin{minipage}{0.32\textwidth}
\begin{scaletikzpicturetowidth}{\textwidth}
\begin{tikzpicture}[scale=\tikzscale,font=\large]
\begin{semilogxaxis}[
xlabel={Degrees of freedom},
ylabel={\%},
xmajorgrids,
ymajorgrids,
title={Speed-up},
/pgfplots/ytick={0,1000,...,6000},
/pgf/number format/1000 sep={}
]
\addplot[color1, mark=diamond*, very thick] table [x index = {0}, y index={1}, col sep=comma] {./Results/PoissonDynamic2DHighFrequency/rel/timing_p2_q5_cc0.250_rel.csv};
\addplot[color2, mark=diamond*, very thick] table [x index = {0}, y index={1}, col sep=comma] {./Results/PoissonDynamic2DHighFrequency/rel/timing_p2_q5_cc0.500_rel.csv};
\addplot[color3, mark=diamond*, very thick] table [x index = {0}, y index={1}, col sep=comma] {./Results/PoissonDynamic2DHighFrequency/rel/timing_p2_q5_cc0.750_rel.csv};
\addplot[color4, mark=diamond*, very thick] table [x index = {0}, y index={1}, col sep=comma] {./Results/PoissonDynamic2DHighFrequency/rel/timing_p2_q5_cc1.000_rel.csv};
\addplot[color5, mark=diamond*, very thick] table [x index = {0}, y index={1}, col sep=comma] {./Results/PoissonDynamic2DHighFrequency/rel/timing_p2_q5_cc2.000_rel.csv};
\addplot[color6, mark=diamond*, very thick] table [x index = {0}, y index={1}, col sep=comma] {./Results/PoissonDynamic2DHighFrequency/rel/timing_p2_q5_cc3.000_rel.csv};
\addplot[color7, mark=diamond*, very thick] table [x index = {0}, y index={1}, col sep=comma] {./Results/PoissonDynamic2DHighFrequency/rel/timing_p2_q5_cc4.000_rel.csv};
\end{semilogxaxis}
\end{tikzpicture}
 \end{scaletikzpicturetowidth}
\end{minipage}
\caption{\label{fig:dynamic_sampling_p2_timing} $2$D domain with $p=2$ and $q=5$. Assembly times with the mesh-dependent sampling strategy $M(h) = \max\Big\{1,\floor[\big]{c \cdot h^{\frac{p-q+\frac{1}{2}}{q+1}}}\Big\}$.
On the left, note the \revised{percentage} of matrix entries computed with quadrature.
Here, the \revised{percentage} of entries involving non-cardinal basis functions (roughly $1-(\frac{m-2p}{m})^n$) is shaded out to indicate the theoretical lower bound.}
\end{figure}

\begin{figure}
	\centering
	\begin{minipage}{0.32\textwidth}
	\begin{scaletikzpicturetowidth}{\textwidth}
\begin{tikzpicture}[scale=\tikzscale,font=\large]
\begin{loglogaxis}[
xlabel={Degrees of freedom},
ylabel={Relative $L^2$ error},
xmajorgrids,
ymajorgrids,
legend style={at={(0.96,0.96)},anchor=north east},
]
\addplot[black, mark=*, very thick, mark options={scale=1.5}, fill opacity=0.6] table [x index = {0}, y index={1}, col sep=comma] {./Results/PoissonDynamic2D/timing_p2_q5_cc0.750.csv};
\addlegendentry{$c = 0\phantom{.00}$};
\addplot[color1, mark=diamond*, very thick] table [x index = {0}, y index={2}, col sep=comma] {./Results/PoissonDynamic2D/timing_p2_q5_cc0.250.csv};
\addlegendentry{$c = 0.25$};
\addplot[color2, mark=diamond*, very thick] table [x index = {0}, y index={2}, col sep=comma] {./Results/PoissonDynamic2D/timing_p2_q5_cc0.500.csv};
\addlegendentry{$c = 0.50$};
\addplot[color3, mark=diamond*, very thick] table [x index = {0}, y index={2}, col sep=comma] {./Results/PoissonDynamic2D/timing_p2_q5_cc0.750.csv};
\addlegendentry{$c = 0.75$};
\addplot[color4, mark=diamond*, very thick] table [x index = {0}, y index={2}, col sep=comma] {./Results/PoissonDynamic2D/timing_p2_q5_cc1.000.csv};
\addlegendentry{$c = 1.00$};
\logLogSlopeTriangle{0.7}{0.2}{0.17}{2}{3}{};
\end{loglogaxis}
\end{tikzpicture} 	\end{scaletikzpicturetowidth}
	\end{minipage}
	\qquad
	\begin{minipage}{0.32\textwidth}
	\begin{scaletikzpicturetowidth}{\textwidth}
\begin{tikzpicture}[scale=\tikzscale,font=\large]
\begin{loglogaxis}[
xlabel={Degrees of freedom},
ylabel={Relative $L^2$ error},
xmajorgrids,
ymajorgrids,
legend style={fill=white, fill opacity=0.6, draw opacity=1, text opacity=1},
legend pos=north east
]
\addplot[black, mark=*, very thick, mark options={scale=1.5}, fill opacity=0.6] table [x index = {0}, y index={1}, col sep=comma] {./Results/PoissonDynamic2DHighFrequency/timing_p2_q5_cc0.750.csv};
\addlegendentry{$c=0\phantom{.00}$};
\addplot[color4, mark=diamond*, very thick] table [x index = {0}, y index={2}, col sep=comma] {./Results/PoissonDynamic2DHighFrequency/timing_p2_q5_cc1.000.csv};
\addlegendentry{$c = 1.00$};
\addplot[color5, mark=diamond*, very thick] table [x index = {0}, y index={2}, col sep=comma] {./Results/PoissonDynamic2DHighFrequency/timing_p2_q5_cc2.000.csv};
\addlegendentry{$c = 2.00$};
\addplot[color6, mark=diamond*, very thick] table [x index = {0}, y index={2}, col sep=comma] {./Results/PoissonDynamic2DHighFrequency/timing_p2_q5_cc3.000.csv};
\addlegendentry{$c = 3.00$};
\addplot[color7, mark=diamond*, very thick] table [x index = {0}, y index={2}, col sep=comma] {./Results/PoissonDynamic2DHighFrequency/timing_p2_q5_cc4.000.csv};
\addlegendentry{$c = 4.00$};
\logLogSlopeTriangle{0.7}{0.2}{0.17}{2}{3}{};
\end{loglogaxis}
\end{tikzpicture} 	\end{scaletikzpicturetowidth}
	\end{minipage}
	\caption{\label{fig:dynamic_sampling_p2}$2$D domain with $p=2$ and $q=5$. Relative $L^2$ errors for $M(h) = \max\Big\{1,\floor[\big]{c \cdot h^{\frac{p-q+\frac{1}{2}}{q+1}}}\Big\}$, with two different manufactured solutions.
	Left: $u(\bmx) = \sin\left(\pi \,x_1\right)\,\sin\left(\pi \,x_2\right)$.
	Right: $u(\bmx) = \sin\left(20\,\pi \,x_1\right)\,\sin\left(20\,\pi \,x_2\right)$.}
\end{figure}

\section{Transverse vibrations of an isotropic membrane} %
\label{sec:transverse_vibrations_of_an_isotropic_membrane}

One of the great advantages of the IGA paradigm is its superior accuracy in structural vibration problems.
Any worthwhile surrogate method should maintain this advantage.
Therefore, in this section, we extend the analysis of the previous section to the analysis of transverse vibrations of a two-dimensional elastic membrane.
The corresponding weak form is the following:
\begin{equation}
	\left\{
	\begin{alignedat}{3}
		&\text{Find all pairs }  u\in H^1_0(\Omega)\setminus\{0\}
		\text{ and } \lambda\in\R
		\text{ satisfying}
		\quad
		\\
		&a(u,v) = \lambda\,m(u,v)
		\quad
		\text{for all }
		v\in H^1_0(\Omega)
		\,,
	\end{alignedat}
	\right.
\label{eq:EigenvalueWeakForm}
\end{equation}
where $a(w,v) = \int_\Omega \nabla w \cdot \nabla v \dd x$ and $m(w,v) = \int_\Omega w \sspace v \dd x$.
It is well-known that there are countably many solutions to this problem, $\{(u^{(k)},\lambda^{(k)})\}_{j=1}^\infty$, where each $\lambda^{(k)}>0$.
From here on, we make the ordering assumption $\lambda^{(j)} \leq \lambda^{(k)}$, for every $j<k$.

\subsection{Surrogate mass matrices} %
\label{sub:surrogate_mass_matrices}

So far, we have only rigorously analyzed surrogates of the elliptic bilinear form $a(\cdot,\cdot)$, written above.
Using the techniques developed thus far, results similar to \Cref{thm:BilinearFormDifferencePoisson} can be proven for other bilinear forms, such as $m(\cdot,\cdot)$.
Let the corresponding mass matrix and surrogate mass matrix be denoted $\sfM$ and $\tilde{\sfM}$, respectively.
Employing definition~\cref{eq:DefinitionOfSurrogateMatrixSymmetric}, we take for granted that the accompanying surrogate bilinear form $\tilde{m}(\cdot,\cdot)$ satisfies
\begin{equation}
	|m(v_h, w_h) - \tilde{m}(v_h, w_h)|
	\leq
	C_7
	H^{q+1}
	\| v_h\|_0 \| w_h\|_0
	\,,
\label{eq:SurrogateMassBound}
\end{equation}
for some $C_7$ depending only on \changed{$\bvarphi$}, $p$, $q$, and $\|\Pi_{H}\|$.

In the special case that the geometry mapping $\bvarphi(\hat{\bmx})$ is a polynomial, we have a surrogate reproduction property of the mass form $m(u,v)$.
This property is formalized in the following corollary to \Cref{prop:PolynomialReproduction}.
\begin{corollary}
\label{lem:MassMatrixReproduction}
Assume that the domain mapping $\bvarphi\colon \hat{\Omega} \rightarrow \Omega$ is defined through a polynomial of order $p$, i.e., $\bvarphi \in \big[\mcQ_p(\hat{\Omega})\big]^n$.
Let $\sfM$ be the coefficient matrix arising from the discretization of $m(\cdot,\cdot)$ and $\tilde{\sfM}$ the corresponding surrogate matrix.
If $q \geq n \cdot p - 1$, it holds that $\sfM = \tilde{\sfM}$.
\end{corollary}
\begin{proof}
Let $\bmJ(\hat{\bmx})$ be the Jacobian (tensor) of $\bvarphi(\hat{\bmx})$.
The transformation of the integral from the physical to the reference domain is given by
\begin{equation}
m(u,v) = \int_\Omega u \sspace v \dd x = \int_{\hat{\Omega}} \hat{u} \sspace \hat{v}\, \det(\bmJ) \dd \hat{\bmx}.
\end{equation}
It remains to show that $G(\hat{\bmx}, \hat{u}(\hat{\bmy}), \hat{v}(\hat{\bmy})) = \hat{u}(\hat{\bmy}) \sspace \hat{v}(\hat{\bmy})\, \det(\bmJ(\hat{\bmx}))$ is a polynomial of degree $n \cdot p - 1$ in the $\hat{\bmx}$-variable.
Since $\bvarphi \in \big[\mcQ_p(\hat{\Omega})\big]^n$, it holds that $\det(\bmJ) \in \mcQ_{n\cdot p - 1}(\hat{\Omega})$.
Applying \cref{prop:PolynomialReproduction} yields the desired reproduction property.
\end{proof}

\subsection{A priori analysis of the eigenvalue error} %
\label{sub:a_priori_analysis_of_the_eigenvalue_error}

Adopt the obvious notation, $(u_h^{(k)},\lambda_h^{(k)})$ and $(\tilde{u}_h^{(k)},\tilde{\lambda}_h^{(k)})$, for the standard IGA and surrogate IGA solutions corresponding to~\cref{eq:EigenvalueWeakForm}, respectively.
Due to space limitations, we only analyze the convergence of the eigenvalues $\tilde{\lambda}_h^{(k)} \to 
\lambda^{(k)}$.

\begin{theorem}
\label{thm:APrioriEigenvalues}
	Let $\theta >1$.
	If $u^{(k)}\in H^{p+1}(\Omega)$, then there exists a constant $C_8$, depending only on $p$ and \changed{$\bvarphi$}, such that
	\begin{equation}
		\frac{|\lambda^{(k)} - \tilde{\lambda}_h^{(k)}|}{\lambda^{(k)}}
		\leq
		C_8 h^{2p} + \theta \cdot (C_4 + C_7) H^{q+1}
		\,,
	\label{eq:EigenvalueBound}
	\end{equation}
	for every sufficiently small $H>0$.
\end{theorem}

\begin{proof}
	Clearly, $|\lambda^{(k)} - \tilde{\lambda}_h^{(k)}| \leq |\lambda^{(k)} - \lambda_h^{(k)}| + |\lambda_h^{(k)} - \tilde{\lambda}_h^{(k)}|$.
	It is known that if $u^{(k)}\in H^{p+1}(\Omega)$, then $|\lambda^{(k)} - \lambda_h^{(k)}| \leq C_8 \lambda^{(k)} h^{2p}$.
	We now focus on the term $|\lambda_h^{(k)} - \tilde{\lambda}_h^{(k)}|$.

	Let $V_h^{(k)}$ be the set of all $k$-dimensional subsets of $V_h$ and fix $\theta>1$.
	Two important members of this set are $E_h^{(k)} = \spann\big\{ u_h^{(l)}\big\}_{l=1}^k$ and $\tilde{E}_h^{(k)} = \spann\big\{ \tilde{u}_h^{(l)}\big\}_{l=1}^k$.
	Observe that
	\begin{equation}
	\begin{aligned}
		\tilde{\lambda}_h^{(k)}
		&=
		\min_{E_h\in V_h^{(k)}} \max_{v\in E_h^{\phantom{(k)}}}
		\frac{\tilde{a}(v,v)}{\tilde{m}(v,v)}
		\leq
		\max_{v\in E_h^{(k)}}
		\frac{\tilde{a}(v,v)}{\tilde{m}(v,v)}
		=
		\max_{v\in E_h^{(k)}}
		\frac{a(v,v)}{m(v,v)}
		\frac{\tilde{a}(v,v)}{a(v,v)}
		\frac{m(v,v)}{\tilde{m}(v,v)}
		\\
		&\leq
		\lambda_h^{(k)}
		\max_{v\in E_h^{(k)}}
		\frac{\tilde{a}(v,v)}{a(v,v)}
		\max_{w\in E_h^{(k)}}
		\frac{m(w,w)}{\tilde{m}(w,w)}
		\,.
	\end{aligned}
	\label{eq:EigenvalueConvergenceA}
	\end{equation}
	Note that $\frac{\tilde{a}(v,v)}{a(v,v)} = 1 + \frac{\tilde{a}(v,v)-a(v,v)}{a(v,v)} \leq 1 + C_4 H^{q+1}$, by \Cref{thm:BilinearFormDifferencePoisson}.
	Similarly, by~\cref{eq:SurrogateMassBound}, $\frac{m(w,w)}{\tilde{m}(w,w)} = 1 + \frac{\tilde{m}(w,w)-m(w,w)}{\tilde{m}(w,w)} \leq 1 + \theta\cdot C_7 H^{q+1}$, in the limit $H\to 0$.
	These observations, together with~\cref{eq:EigenvalueConvergenceA}, imply
	\begin{subequations}
	\begin{equation}
		\tilde{\lambda}_h^{(k)}
		\leq
		\lambda_h^{(k)}
		+
		\lambda_h^{(k)} (C_4 + \theta\cdot C_7) H^{q+1}
		\,,
	\label{eq:EigenvalueConvergenceB}
	\end{equation}
	for every sufficiently small $H>0$.
	Following a similar argument,
	\begin{equation}
	\begin{aligned}
		\lambda_h^{(k)}
		&\leq
		\tilde{\lambda}_h^{(k)}
		\frac{\tilde{a}(v,v)}{\tilde{a}(v,v)}
		\max_{w\in E_h^{(k)}}
		\frac{\tilde{m}(w,w)}{m(w,w)}
		&\leq \tilde{\lambda}_h^{(k)} + \tilde{\lambda}_h^{(k)} (\theta\cdot C_4 + C_7) H^{q+1}
		\,,
	\end{aligned}
	\label{eq:EigenvalueConvergenceC}
	\end{equation}
	\end{subequations}
	for every sufficiently small $H>0$.
	Note that~\cref{eq:EigenvalueConvergenceB} implies that $\tilde{\lambda}_h^{(k)} \leq \theta \lambda_h^{(k)}$, as $H\to 0$.
	After introducing this inequality into~\cref{eq:EigenvalueConvergenceC}, we arrive at the upper bound $|\lambda_h^{(k)} - \tilde{\lambda}_h^{(k)}| \leq \lambda_h^{(k)}\, \theta\cdot (C_4 + C_7) H^{q+1}$, in the limit $H\to 0$.
	Because $\lambda_h^{(k)} \to \lambda^{(k)}$, we also have $\lambda_h^{(k)} \leq \theta\lambda^{(k)}$, as $H\to 0$.
	Only elementary algebra remains in order to arrive at~\cref{eq:EigenvalueBound}.
\end{proof}

\begin{figure}
\centering
	\includegraphics[trim=1cm 0cm 1.5cm 0cm,clip=true,height=4.5cm]{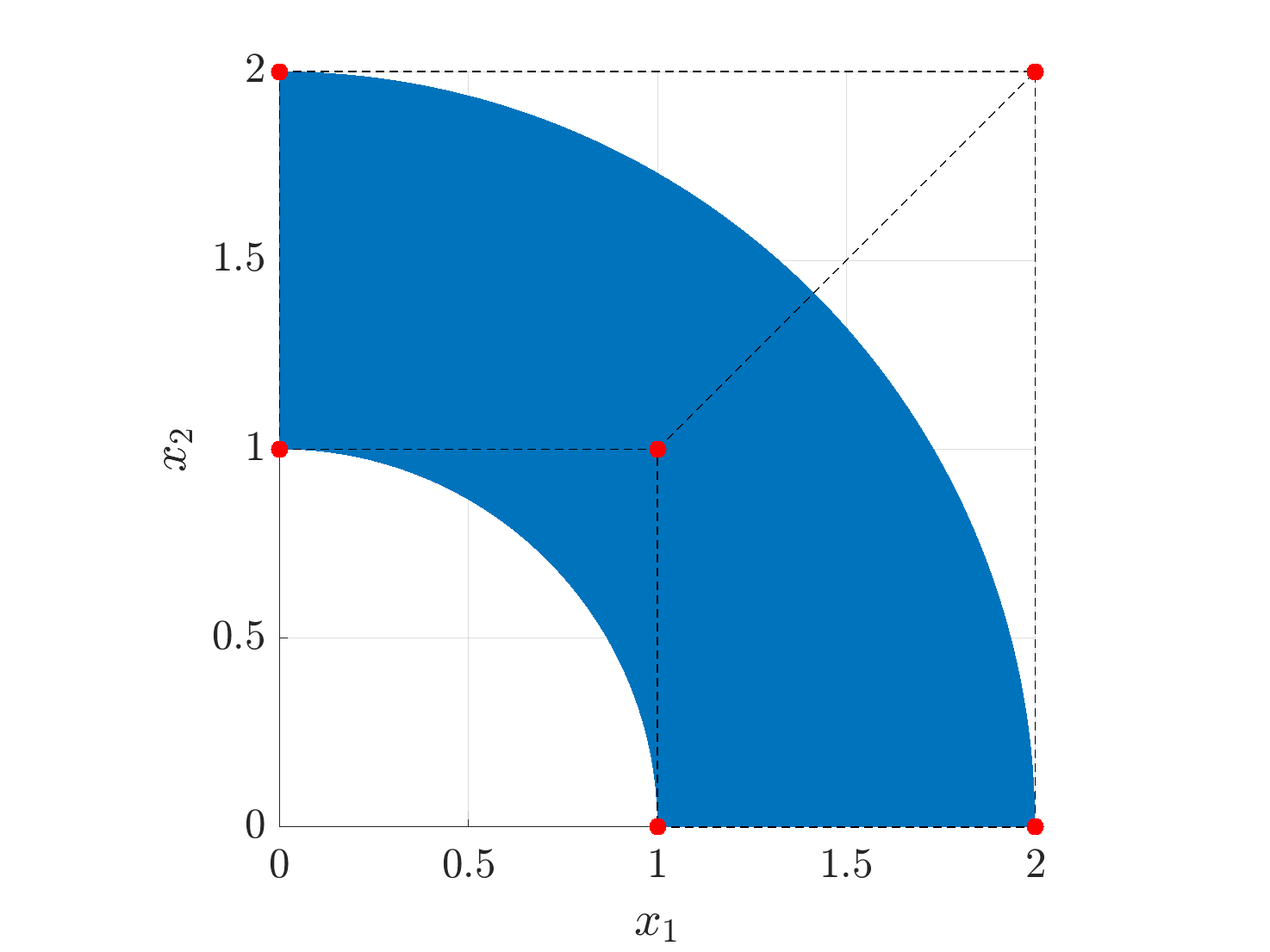}
	\caption{\label{fig:EigenvalueGeometry}The physical domain $\Omega$ for problems~\cref{eq:EigenvalueWeakForm,eq:PlateWeakForm}.}
\end{figure}

\begin{remark}
	One upshot of~\cref{thm:APrioriEigenvalues} is that, if $H = \mcO(h)$, then one may wish to choose $q+1 > 2p$, in order to recover an optimal spectral convergence rate.
	Of course, for irregular geometries, it is difficult to maintain the assumption $\lambda^{(j)} \in H^{p+1}(\Omega)$, and a lower $q$ may still provide a useful approximation.
\end{remark}

\subsection{Numerical experiments} %
\label{sub:Membrane_numerical_experiments}

\begin{figure}%
\centering
\begin{minipage}{0.32\textwidth}
\begin{scaletikzpicturetowidth}{\textwidth}
\begin{tikzpicture}[scale=\tikzscale,font=\large]
\begin{loglogaxis}[
xlabel={Degrees of freedom},
ylabel={$|\lambda^{(1)} - \tilde{\lambda}_h^{(1)}|$},
xmajorgrids,
ymajorgrids,
legend style={fill=white, fill opacity=0.6, draw opacity=1, text opacity=1},
legend pos=north east
]
\addplot[black, mark=*, very thick, mark options={scale=1.5}, fill opacity=0.6] table [x index = {0}, y index={1}, col sep=comma] {./Results/PoissonStatic2D/eig_errors_p2_I1_QuarterAnnulus.csv};
\addlegendentry{$M=1$};
\addplot[color2, mark=diamond*, very thick] table [x index = {0}, y index={2}, col sep=comma] {./Results/PoissonStatic2D/eig_errors_p2_I1_QuarterAnnulus.csv};
\addlegendentry{$q=1$};
\addplot[color3, mark=diamond*, very thick] table [x index = {0}, y index={3}, col sep=comma] {./Results/PoissonStatic2D/eig_errors_p2_I1_QuarterAnnulus.csv};
\addlegendentry{$q=2$};
\addplot[color4, mark=diamond*, very thick] table [x index = {0}, y index={4}, col sep=comma] {./Results/PoissonStatic2D/eig_errors_p2_I1_QuarterAnnulus.csv};
\addlegendentry{$q=3$};
\addplot[color5, mark=diamond*, very thick] table [x index = {0}, y index={5}, col sep=comma] {./Results/PoissonStatic2D/eig_errors_p2_I1_QuarterAnnulus.csv};
\addlegendentry{$q=4$};
\addplot[color6, mark=diamond*, very thick] table [x index = {0}, y index={6}, col sep=comma] {./Results/PoissonStatic2D/eig_errors_p2_I1_QuarterAnnulus.csv};
\addlegendentry{$q=5$};
\logLogSlopeTriangle{0.7}{0.2}{0.13}{2}{4}{};
\legend{}; %
\end{loglogaxis}
\end{tikzpicture} \end{scaletikzpicturetowidth}
\end{minipage}
\hfill
\begin{minipage}{0.32\textwidth}
\begin{scaletikzpicturetowidth}{\textwidth}
\begin{tikzpicture}[scale=\tikzscale,font=\large]
\begin{loglogaxis}[
xlabel={Degrees of freedom},
ylabel={$|\lambda^{(2)} - \tilde{\lambda}_h^{(2)}|$},
xmajorgrids,
ymajorgrids,
legend style={fill=white, fill opacity=0.6, draw opacity=1, text opacity=1},
legend pos=north east
]
\addplot[black, mark=*, very thick, mark options={scale=1.5}, fill opacity=0.6] table [x index = {0}, y index={1}, col sep=comma] {./Results/PoissonStatic2D/eig_errors_p2_I2_QuarterAnnulus.csv};
\addlegendentry{$M=1$};
\addplot[color2, mark=diamond*, very thick] table [x index = {0}, y index={2}, col sep=comma] {./Results/PoissonStatic2D/eig_errors_p2_I2_QuarterAnnulus.csv};
\addlegendentry{$q=1$};
\addplot[color3, mark=diamond*, very thick] table [x index = {0}, y index={3}, col sep=comma] {./Results/PoissonStatic2D/eig_errors_p2_I2_QuarterAnnulus.csv};
\addlegendentry{$q=2$};
\addplot[color4, mark=diamond*, very thick] table [x index = {0}, y index={4}, col sep=comma] {./Results/PoissonStatic2D/eig_errors_p2_I2_QuarterAnnulus.csv};
\addlegendentry{$q=3$};
\addplot[color5, mark=diamond*, very thick] table [x index = {0}, y index={5}, col sep=comma] {./Results/PoissonStatic2D/eig_errors_p2_I2_QuarterAnnulus.csv};
\addlegendentry{$q=4$};
\addplot[color6, mark=diamond*, very thick] table [x index = {0}, y index={6}, col sep=comma] {./Results/PoissonStatic2D/eig_errors_p2_I2_QuarterAnnulus.csv};
\addlegendentry{$q=5$};
\logLogSlopeTriangle{0.7}{0.2}{0.13}{2}{4}{};
\legend{}; %
\end{loglogaxis}
\end{tikzpicture} \end{scaletikzpicturetowidth}
\end{minipage}
\hfill
\begin{minipage}{0.32\textwidth}
\begin{scaletikzpicturetowidth}{\textwidth}
\begin{tikzpicture}[scale=\tikzscale,font=\large]
\begin{loglogaxis}[
xlabel={Degrees of freedom},
ylabel={$|\lambda^{(3)} - \tilde{\lambda}_h^{(3)}|$},
xmajorgrids,
ymajorgrids,
legend style={fill=white, fill opacity=0.6, draw opacity=1, text opacity=1},
legend pos=north east
]
\addplot[black, mark=*, very thick, mark options={scale=1.5}, fill opacity=0.6] table [x index = {0}, y index={1}, col sep=comma] {./Results/PoissonStatic2D/eig_errors_p2_I3_QuarterAnnulus.csv};
\addlegendentry{$M=1$};
\addplot[color2, mark=diamond*, very thick] table [x index = {0}, y index={2}, col sep=comma] {./Results/PoissonStatic2D/eig_errors_p2_I3_QuarterAnnulus.csv};
\addlegendentry{$q=1$};
\addplot[color3, mark=diamond*, very thick] table [x index = {0}, y index={3}, col sep=comma] {./Results/PoissonStatic2D/eig_errors_p2_I3_QuarterAnnulus.csv};
\addlegendentry{$q=2$};
\addplot[color4, mark=diamond*, very thick] table [x index = {0}, y index={4}, col sep=comma] {./Results/PoissonStatic2D/eig_errors_p2_I3_QuarterAnnulus.csv};
\addlegendentry{$q=3$};
\addplot[color5, mark=diamond*, very thick] table [x index = {0}, y index={5}, col sep=comma] {./Results/PoissonStatic2D/eig_errors_p2_I3_QuarterAnnulus.csv};
\addlegendentry{$q=4$};
\addplot[color6, mark=diamond*, very thick] table [x index = {0}, y index={6}, col sep=comma] {./Results/PoissonStatic2D/eig_errors_p2_I3_QuarterAnnulus.csv};
\addlegendentry{$q=5$};
\logLogSlopeTriangle{0.7}{0.2}{0.13}{2}{4}{};
\legend{}; %
\end{loglogaxis}
\end{tikzpicture} \end{scaletikzpicturetowidth}
\end{minipage}
\\
\begin{minipage}{0.32\textwidth}
\begin{scaletikzpicturetowidth}{\textwidth}
\begin{tikzpicture}[scale=\tikzscale,font=\large]
\begin{loglogaxis}[
xlabel={Degrees of freedom},
ylabel={$|\lambda^{(4)} - \tilde{\lambda}_h^{(4)}|$},
xmajorgrids,
ymajorgrids,
legend style={fill=white, fill opacity=0.6, draw opacity=1, text opacity=1},
legend pos=north east
]
\addplot[black, mark=*, very thick, mark options={scale=1.5}, fill opacity=0.6] table [x index = {0}, y index={1}, col sep=comma] {./Results/PoissonStatic2D/eig_errors_p2_I4_QuarterAnnulus.csv};
\addlegendentry{$M=1$};
\addplot[color2, mark=diamond*, very thick] table [x index = {0}, y index={2}, col sep=comma] {./Results/PoissonStatic2D/eig_errors_p2_I4_QuarterAnnulus.csv};
\addlegendentry{$q=1$};
\addplot[color3, mark=diamond*, very thick] table [x index = {0}, y index={3}, col sep=comma] {./Results/PoissonStatic2D/eig_errors_p2_I4_QuarterAnnulus.csv};
\addlegendentry{$q=2$};
\addplot[color4, mark=diamond*, very thick] table [x index = {0}, y index={4}, col sep=comma] {./Results/PoissonStatic2D/eig_errors_p2_I4_QuarterAnnulus.csv};
\addlegendentry{$q=3$};
\addplot[color5, mark=diamond*, very thick] table [x index = {0}, y index={5}, col sep=comma] {./Results/PoissonStatic2D/eig_errors_p2_I4_QuarterAnnulus.csv};
\addlegendentry{$q=4$};
\addplot[color6, mark=diamond*, very thick] table [x index = {0}, y index={6}, col sep=comma] {./Results/PoissonStatic2D/eig_errors_p2_I4_QuarterAnnulus.csv};
\addlegendentry{$q=5$};
\logLogSlopeTriangle{0.7}{0.2}{0.13}{2}{4}{};
\legend{}; %
\end{loglogaxis}
\end{tikzpicture} \end{scaletikzpicturetowidth}
\end{minipage}
\hfill
\begin{minipage}{0.32\textwidth}
\begin{scaletikzpicturetowidth}{\textwidth}
\begin{tikzpicture}[scale=\tikzscale,font=\large]
\begin{loglogaxis}[
xlabel={Degrees of freedom},
ylabel={$|\lambda^{(5)} - \tilde{\lambda}_h^{(5)}|$},
xmajorgrids,
ymajorgrids,
legend style={fill=white, fill opacity=0.6, draw opacity=1, text opacity=1},
legend pos=north east
]
\addplot[black, mark=*, very thick, mark options={scale=1.5}, fill opacity=0.6] table [x index = {0}, y index={1}, col sep=comma] {./Results/PoissonStatic2D/eig_errors_p2_I5_QuarterAnnulus.csv};
\addlegendentry{$M=1$};
\addplot[color2, mark=diamond*, very thick] table [x index = {0}, y index={2}, col sep=comma] {./Results/PoissonStatic2D/eig_errors_p2_I5_QuarterAnnulus.csv};
\addlegendentry{$q=1$};
\addplot[color3, mark=diamond*, very thick] table [x index = {0}, y index={3}, col sep=comma] {./Results/PoissonStatic2D/eig_errors_p2_I5_QuarterAnnulus.csv};
\addlegendentry{$q=2$};
\addplot[color4, mark=diamond*, very thick] table [x index = {0}, y index={4}, col sep=comma] {./Results/PoissonStatic2D/eig_errors_p2_I5_QuarterAnnulus.csv};
\addlegendentry{$q=3$};
\addplot[color5, mark=diamond*, very thick] table [x index = {0}, y index={5}, col sep=comma] {./Results/PoissonStatic2D/eig_errors_p2_I5_QuarterAnnulus.csv};
\addlegendentry{$q=4$};
\addplot[color6, mark=diamond*, very thick] table [x index = {0}, y index={6}, col sep=comma] {./Results/PoissonStatic2D/eig_errors_p2_I5_QuarterAnnulus.csv};
\addlegendentry{$q=5$};
\logLogSlopeTriangle{0.7}{0.2}{0.13}{2}{4}{};
\legend{}; %
\end{loglogaxis}
\end{tikzpicture} \end{scaletikzpicturetowidth}
\end{minipage}
\hfill
\begin{minipage}{0.32\textwidth}
\begin{scaletikzpicturetowidth}{\textwidth}
\begin{tikzpicture}[scale=\tikzscale,font=\large]
\begin{loglogaxis}[
xlabel={Degrees of freedom},
ylabel={$|\lambda^{(6)} - \tilde{\lambda}_h^{(6)}|$},
xmajorgrids,
ymajorgrids,
legend style={fill=white, fill opacity=0.6, draw opacity=1, text opacity=1},
legend pos=north east
]
\addplot[black, mark=*, very thick, mark options={scale=1.5}, fill opacity=0.6] table [x index = {0}, y index={1}, col sep=comma] {./Results/PoissonStatic2D/eig_errors_p2_I6_QuarterAnnulus.csv};
\addlegendentry{$M=1$};
\addplot[color2, mark=diamond*, very thick] table [x index = {0}, y index={2}, col sep=comma] {./Results/PoissonStatic2D/eig_errors_p2_I6_QuarterAnnulus.csv};
\addlegendentry{$q=1$};
\addplot[color3, mark=diamond*, very thick] table [x index = {0}, y index={3}, col sep=comma] {./Results/PoissonStatic2D/eig_errors_p2_I6_QuarterAnnulus.csv};
\addlegendentry{$q=2$};
\addplot[color4, mark=diamond*, very thick] table [x index = {0}, y index={4}, col sep=comma] {./Results/PoissonStatic2D/eig_errors_p2_I6_QuarterAnnulus.csv};
\addlegendentry{$q=3$};
\addplot[color5, mark=diamond*, very thick] table [x index = {0}, y index={5}, col sep=comma] {./Results/PoissonStatic2D/eig_errors_p2_I6_QuarterAnnulus.csv};
\addlegendentry{$q=4$};
\addplot[color6, mark=diamond*, very thick] table [x index = {0}, y index={6}, col sep=comma] {./Results/PoissonStatic2D/eig_errors_p2_I6_QuarterAnnulus.csv};
\addlegendentry{$q=5$};
\logLogSlopeTriangle{0.7}{0.2}{0.13}{2}{4}{};
\legend{}; %
\end{loglogaxis}
\end{tikzpicture} \end{scaletikzpicturetowidth}
\end{minipage}
\\
\begin{minipage}{0.32\textwidth}
\begin{scaletikzpicturetowidth}{\textwidth}
\begin{tikzpicture}[scale=\tikzscale,font=\large]
\begin{loglogaxis}[
xlabel={Degrees of freedom},
ylabel={$|\lambda^{(7)} - \tilde{\lambda}_h^{(7)}|$},
xmajorgrids,
ymajorgrids,
legend style={fill=white, fill opacity=0.6, draw opacity=1, text opacity=1},
legend pos=north east
]
\addplot[black, mark=*, very thick, mark options={scale=1.5}, fill opacity=0.6] table [x index = {0}, y index={1}, col sep=comma] {./Results/PoissonStatic2D/eig_errors_p2_I7_QuarterAnnulus.csv};
\addlegendentry{$M=1$};
\addplot[color2, mark=diamond*, very thick] table [x index = {0}, y index={2}, col sep=comma] {./Results/PoissonStatic2D/eig_errors_p2_I7_QuarterAnnulus.csv};
\addlegendentry{$q=1$};
\addplot[color3, mark=diamond*, very thick] table [x index = {0}, y index={3}, col sep=comma] {./Results/PoissonStatic2D/eig_errors_p2_I7_QuarterAnnulus.csv};
\addlegendentry{$q=2$};
\addplot[color4, mark=diamond*, very thick] table [x index = {0}, y index={4}, col sep=comma] {./Results/PoissonStatic2D/eig_errors_p2_I7_QuarterAnnulus.csv};
\addlegendentry{$q=3$};
\addplot[color5, mark=diamond*, very thick] table [x index = {0}, y index={5}, col sep=comma] {./Results/PoissonStatic2D/eig_errors_p2_I7_QuarterAnnulus.csv};
\addlegendentry{$q=4$};
\addplot[color6, mark=diamond*, very thick] table [x index = {0}, y index={6}, col sep=comma] {./Results/PoissonStatic2D/eig_errors_p2_I7_QuarterAnnulus.csv};
\addlegendentry{$q=5$};
\logLogSlopeTriangle{0.7}{0.2}{0.13}{2}{4}{};
\legend{}; %
\end{loglogaxis}
\end{tikzpicture} \end{scaletikzpicturetowidth}
\end{minipage}
\hfill
\begin{minipage}{0.32\textwidth}
\begin{scaletikzpicturetowidth}{\textwidth}
\begin{tikzpicture}[scale=\tikzscale,font=\large]
\begin{loglogaxis}[
xlabel={Degrees of freedom},
ylabel={$|\lambda^{(8)} - \tilde{\lambda}_h^{(8)}|$},
xmajorgrids,
ymajorgrids,
legend style={fill=white, fill opacity=0.6, draw opacity=1, text opacity=1},
legend pos=north east
]
\addplot[black, mark=*, very thick, mark options={scale=1.5}, fill opacity=0.6] table [x index = {0}, y index={1}, col sep=comma] {./Results/PoissonStatic2D/eig_errors_p2_I8_QuarterAnnulus.csv};
\addlegendentry{$M=1$};
\addplot[color2, mark=diamond*, very thick] table [x index = {0}, y index={2}, col sep=comma] {./Results/PoissonStatic2D/eig_errors_p2_I8_QuarterAnnulus.csv};
\addlegendentry{$q=1$};
\addplot[color3, mark=diamond*, very thick] table [x index = {0}, y index={3}, col sep=comma] {./Results/PoissonStatic2D/eig_errors_p2_I8_QuarterAnnulus.csv};
\addlegendentry{$q=2$};
\addplot[color4, mark=diamond*, very thick] table [x index = {0}, y index={4}, col sep=comma] {./Results/PoissonStatic2D/eig_errors_p2_I8_QuarterAnnulus.csv};
\addlegendentry{$q=3$};
\addplot[color5, mark=diamond*, very thick] table [x index = {0}, y index={5}, col sep=comma] {./Results/PoissonStatic2D/eig_errors_p2_I8_QuarterAnnulus.csv};
\addlegendentry{$q=4$};
\addplot[color6, mark=diamond*, very thick] table [x index = {0}, y index={6}, col sep=comma] {./Results/PoissonStatic2D/eig_errors_p2_I8_QuarterAnnulus.csv};
\addlegendentry{$q=5$};
\logLogSlopeTriangle{0.7}{0.2}{0.13}{2}{4}{};
\legend{}; %
\end{loglogaxis}
\end{tikzpicture} \end{scaletikzpicturetowidth}
\end{minipage}
\hfill
\begin{minipage}{0.32\textwidth}
\begin{scaletikzpicturetowidth}{\textwidth}
\begin{tikzpicture}[scale=\tikzscale,font=\large]
\begin{loglogaxis}[
xlabel={Degrees of freedom},
ylabel={$|\lambda^{(9)} - \tilde{\lambda}_h^{(9)}|$},
xmajorgrids,
ymajorgrids,
legend style={fill=white, fill opacity=0.6, draw opacity=1, text opacity=1},
legend pos=north east
]
\addplot[black, mark=*, very thick, mark options={scale=1.5}, fill opacity=0.6] table [x index = {0}, y index={1}, col sep=comma] {./Results/PoissonStatic2D/eig_errors_p2_I9_QuarterAnnulus.csv};
\addlegendentry{$M=1$};
\addplot[color2, mark=diamond*, very thick] table [x index = {0}, y index={2}, col sep=comma] {./Results/PoissonStatic2D/eig_errors_p2_I9_QuarterAnnulus.csv};
\addlegendentry{$q=1$};
\addplot[color3, mark=diamond*, very thick] table [x index = {0}, y index={3}, col sep=comma] {./Results/PoissonStatic2D/eig_errors_p2_I9_QuarterAnnulus.csv};
\addlegendentry{$q=2$};
\addplot[color4, mark=diamond*, very thick] table [x index = {0}, y index={4}, col sep=comma] {./Results/PoissonStatic2D/eig_errors_p2_I9_QuarterAnnulus.csv};
\addlegendentry{$q=3$};
\addplot[color5, mark=diamond*, very thick] table [x index = {0}, y index={5}, col sep=comma] {./Results/PoissonStatic2D/eig_errors_p2_I9_QuarterAnnulus.csv};
\addlegendentry{$q=4$};
\addplot[color6, mark=diamond*, very thick] table [x index = {0}, y index={6}, col sep=comma] {./Results/PoissonStatic2D/eig_errors_p2_I9_QuarterAnnulus.csv};
\addlegendentry{$q=5$};
\logLogSlopeTriangle{0.7}{0.2}{0.13}{2}{4}{};
\end{loglogaxis}
\end{tikzpicture} \end{scaletikzpicturetowidth}
\end{minipage}
\caption{\label{fig:eigenvalue_convergence_p2} Eigenvalue convergence plots. Quarter annulus geometry with $p=2$ and the constant sampling length $M=5$.
Note that the ``exact'' solutions, $\lambda^{(k)}$, were taken from precomputed values from a high-order discretization on a much finer mesh.
}
\end{figure}
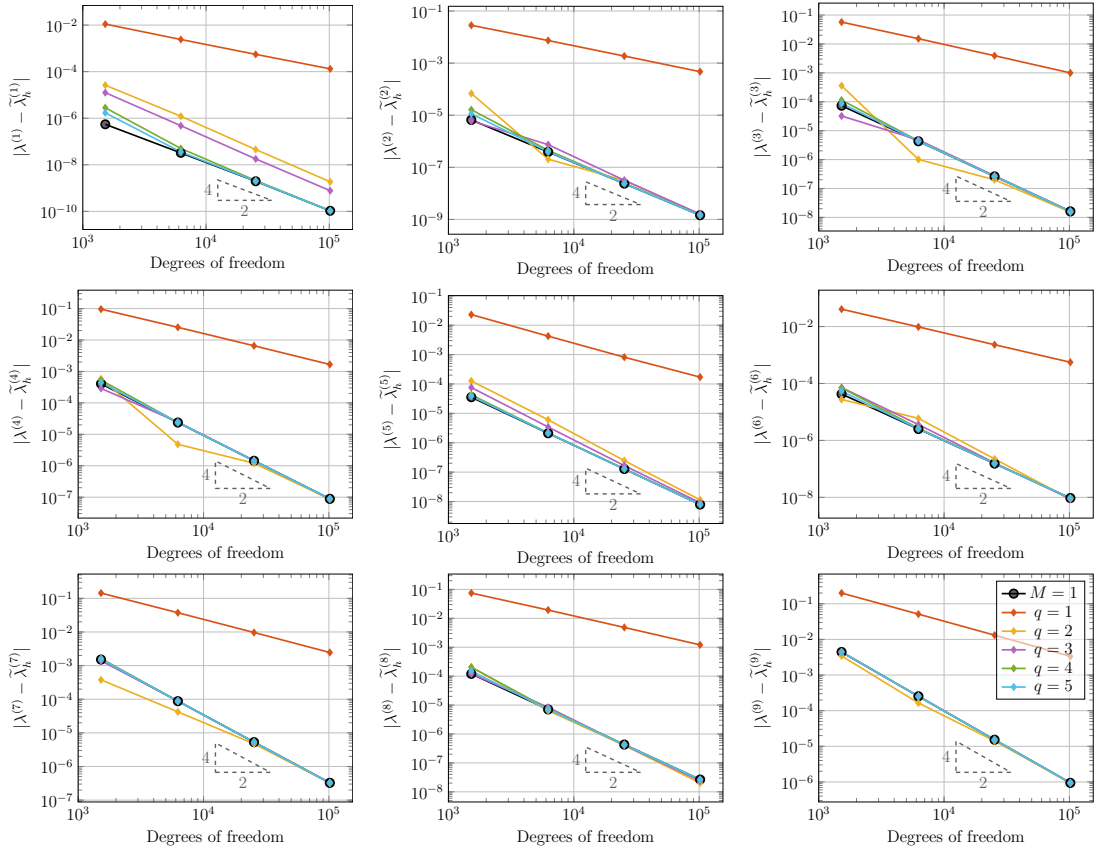

\begin{figure}%
\centering
\begin{minipage}{0.45\textwidth}
\begin{scaletikzpicturetowidth}{\textwidth}
\begin{tikzpicture}[scale=\tikzscale,font=\large]
\begin{axis}[
xlabel={k/N},
ylabel={$\frac{\omega_h^{(k)}-\tilde{\omega}_h^{(k)}}{\omega_h^{(k)}}$},
xmajorgrids,
ymajorgrids,
title={$p=2$},
legend style={fill=white, fill opacity=0.6, draw opacity=1, text opacity=1},
legend pos=north west
]
\addplot[color3, very thick] table [x index = {0}, y index={1}, col sep=comma] {./Results/PoissonStatic2D/omega_diff_p2.csv};
\addlegendentry{$q=2$};
\addplot[color4, very thick] table [x index = {0}, y index={2}, col sep=comma] {./Results/PoissonStatic2D/omega_diff_p2.csv};
\addlegendentry{$q=3$};
\addplot[color5, very thick] table [x index = {0}, y index={3}, col sep=comma] {./Results/PoissonStatic2D/omega_diff_p2.csv};
\addlegendentry{$q=4$};
\addplot[color6, very thick] table [x index = {0}, y index={4}, col sep=comma] {./Results/PoissonStatic2D/omega_diff_p2.csv};
\addlegendentry{$q=5$};
\legend{}; %
\end{axis}
\end{tikzpicture} \end{scaletikzpicturetowidth}
\end{minipage}
\qquad
\begin{minipage}{0.45\textwidth}
\begin{scaletikzpicturetowidth}{\textwidth}
\begin{tikzpicture}[scale=\tikzscale,font=\large]
\begin{axis}[
xlabel={k/N},
ylabel={$\frac{\omega_h^{(k)}-\tilde{\omega}_h^{(k)}}{\omega_h^{(k)}}$},
xmajorgrids,
ymajorgrids,
title={$p=3$},
legend style={fill=white, fill opacity=0.6, draw opacity=1, text opacity=1},
legend pos=north west
]
\addplot[color3, very thick] table [x index = {0}, y index={1}, col sep=comma] {./Results/PoissonStatic2D/omega_diff_p3.csv};
\addlegendentry{$q=2$};
\addplot[color4, very thick] table [x index = {0}, y index={2}, col sep=comma] {./Results/PoissonStatic2D/omega_diff_p3.csv};
\addlegendentry{$q=3$};
\addplot[color5, very thick] table [x index = {0}, y index={3}, col sep=comma] {./Results/PoissonStatic2D/omega_diff_p3.csv};
\addlegendentry{$q=4$};
\addplot[color6, very thick] table [x index = {0}, y index={4}, col sep=comma] {./Results/PoissonStatic2D/omega_diff_p3.csv};
\addlegendentry{$q=5$};
\end{axis}
\end{tikzpicture} \end{scaletikzpicturetowidth}
\end{minipage}
\caption{\label{fig:MembraneSpectum_p2}Relative differences in the computed natural frequencies from the IGA and the surrogate IGA method for $50\times50$ control points and $M=5$.}
\end{figure}
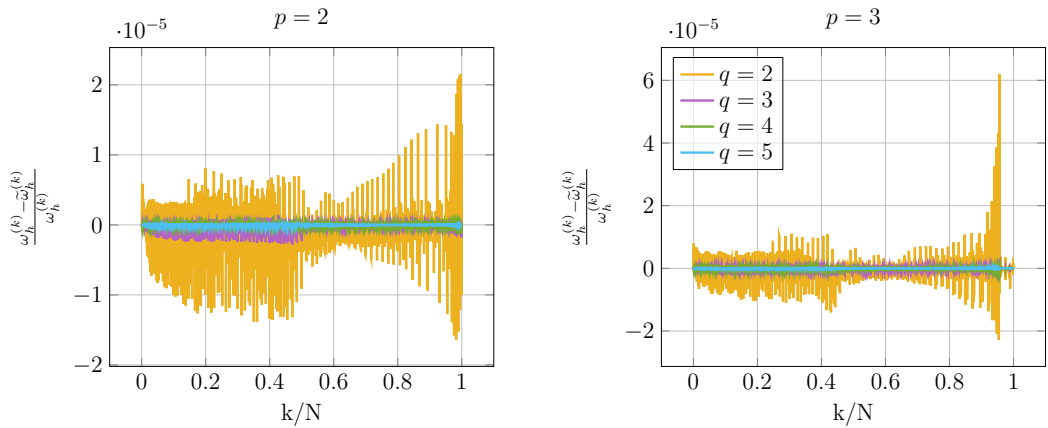

Our numerical experiments for problem~\cref{eq:EigenvalueWeakForm} involve only the two-dimensional quarter annulus domain $\Omega$, depicted in \Cref{fig:EigenvalueGeometry}.
This domain was chosen instead of \Cref{fig:PoissonGeometries2D} so that $u^{(k)}\in H^3(\Omega)$, for each $k$.
Thus, when $p=2$, \Cref{thm:APrioriEigenvalues} concludes that $|\lambda_h^{(k)} - \tilde{\lambda}_h^{(k)}| \leq \mcO(h^4 + H^{q+1})$.
\Cref{fig:eigenvalue_convergence_p2}, which shows the convergence of the first nine eigenvalues, for $p=2$, verifies this result.
Recall \Cref{rem:parity}.
One again witnesses the parity present in the $L^2$ error of the Poisson problems above.
That is, we actually observe the stronger conclusion $|\lambda_h^{(k)} - \tilde{\lambda}_h^{(k)}| = \mcO(h^4 + H^{q+2})$ when $q$ is even.

A close inspection of \Cref{fig:eigenvalue_convergence_p2} appears to indicate that the accuracy of the surrogate solution improves as the eigenvalues grow.
This observation is in line with the previous numerical results for Poisson's equation, which showed nearly indistinguishable solutions for the ``high-frequency'' manufactured solution $u(\bmx) = \sin\left(20\,\pi \,x_1\right)\,\sin\left(20\,\pi \,x_2\right)$ (cf. \Cref{fig:varying_hs_p2,fig:varying_hs_p3,fig:varying_hs_p2_3D}).
Naturally, we should compare the accuracy of all eigenvalues computed with the standard IGA method to those coming from the surrogate IGA method.
This is done, in part, in \Cref{fig:MembraneSpectum_p2} for both $p=2,3$.
Here, it is more meaningful to use the natural frequencies $(\omega^{(j)})^2 = \lambda^{(j)}$.
Notice that the differences are extremely small across the entire range of computed frequencies.

\section{Plate bending under a transverse load} %
\label{sec:the_biharmonic_equation}

Another clear advantage of IGA is the simplicity of discretizing high-order PDEs.
In this section, we briefly demonstrate that the same features hold true for surrogate IGA methods.
As a proof-of-concept, consider the simple Poisson--Kirchoff isotropic plate bending model.
Given a function $f\in L^2(\Omega)$, the corresponding weak form is the following:
\begin{equation}
	\text{Find } u\in H^2(\Omega)\cap H^1_0(\Omega)
	\text{ satisfying }
	\quad
	a(u,v) = F(v)
	\quad
	\text{for all }
	v\in H^2(\Omega)\cap H^1_0(\Omega)
	\,,
\label{eq:PlateWeakForm}
\end{equation}
where $a(u,v) = \int_\Omega \Delta u \, \Delta v \dd x$ and $F(v) = \int_\Omega f\sspace v \dd x$.

\subsection{A higher-dimensional kernel} %
\label{sub:a_higher_dimension_kernel}

With the same principles as used for Poisson's equation, one may easily design a surrogate IGA method for~\cref{eq:PlateWeakForm}.
In our approach, the corresponding surrogate stiffness matrix $\tilde{\sfA}$ was also defined using~\cref{eq:DefinitionOfSurrogateMatrixSymmetricKernel}.
Notice that this definition does not preserve the entire kernel found in the true IGA stiffness matrix $\sfA$.
For instance, one may easily verify that all linear functions lie in the kernel of $a(\cdot,\cdot)$.
Therefore, $\sfA\sfc^{(1)} = \sfA\sfc^{(2)} = 0$, where $\sfc^{(1)},\sfc^{(2)}$ are the $x_1$- and $x_2$-coefficients of the control points, respectively.
In our experiments, this property was only recovered in the limits $h\to 0$ or $H\to h$.

\begin{figure}%
\begin{minipage}{\textwidth}
\centering
\begin{minipage}{0.32\textwidth}
\begin{scaletikzpicturetowidth}{\textwidth}
\begin{tikzpicture}[scale=\tikzscale,font=\large]
\begin{loglogaxis}[
xlabel={Degrees of freedom},
ylabel={Relative $H^2$ error},
xmajorgrids,
ymajorgrids,
legend style={fill=white, fill opacity=0.6, draw opacity=1, text opacity=1},
legend pos=north east
]
\addplot[black, mark=*, very thick, mark options={scale=1.5}, fill opacity=0.6] table [x index = {0}, y index={1}, col sep=comma] {./Results/BiharmonicStatic/h2errors_p3.csv};
\addlegendentry{$M=1$};
\addplot[color3, mark=diamond*, very thick] table [x index = {0}, y index={3}, col sep=comma] {./Results/BiharmonicStatic/h2errors_p3.csv};
\addlegendentry{$q=2$};
\addplot[color4, mark=diamond*, very thick] table [x index = {0}, y index={4}, col sep=comma] {./Results/BiharmonicStatic/h2errors_p3.csv};
\addlegendentry{$q=3$};
\addplot[color5, mark=diamond*, very thick] table [x index = {0}, y index={5}, col sep=comma] {./Results/BiharmonicStatic/h2errors_p3.csv};
\addlegendentry{$q=4$};
\addplot[color6, mark=diamond*, very thick] table [x index = {0}, y index={6}, col sep=comma] {./Results/BiharmonicStatic/h2errors_p3.csv};
\addlegendentry{$q=5$};
\logLogSlopeTriangle{0.7}{0.2}{0.17}{2}{2}{};
\legend{}; %
\end{loglogaxis}
\end{tikzpicture} \end{scaletikzpicturetowidth}
\end{minipage}
\hfill
\begin{minipage}{0.32\textwidth}
\begin{scaletikzpicturetowidth}{\textwidth}
\begin{tikzpicture}[scale=\tikzscale,font=\large]
\begin{loglogaxis}[
xlabel={Degrees of freedom},
ylabel={Relative $H^1$ error},
xmajorgrids,
ymajorgrids,
legend style={fill=white, fill opacity=0.6, draw opacity=1, text opacity=1},
legend pos=north east
]
\addplot[black, mark=*, very thick, mark options={scale=1.5}, fill opacity=0.6] table [x index = {0}, y index={1}, col sep=comma] {./Results/BiharmonicStatic/h1errors_p3.csv};
\addlegendentry{$M=1$};
\addplot[color3, mark=diamond*, very thick] table [x index = {0}, y index={3}, col sep=comma] {./Results/BiharmonicStatic/h1errors_p3.csv};
\addlegendentry{$q=2$};
\addplot[color4, mark=diamond*, very thick] table [x index = {0}, y index={4}, col sep=comma] {./Results/BiharmonicStatic/h1errors_p3.csv};
\addlegendentry{$q=3$};
\addplot[color5, mark=diamond*, very thick] table [x index = {0}, y index={5}, col sep=comma] {./Results/BiharmonicStatic/h1errors_p3.csv};
\addlegendentry{$q=4$};
\addplot[color6, mark=diamond*, very thick] table [x index = {0}, y index={6}, col sep=comma] {./Results/BiharmonicStatic/h1errors_p3.csv};
\addlegendentry{$q=5$};
\logLogSlopeTriangle{0.7}{0.2}{0.15}{2}{3}{};
\legend{}; %
\end{loglogaxis}
\end{tikzpicture} \end{scaletikzpicturetowidth}
\end{minipage}
\hfill
\begin{minipage}{0.32\textwidth}
\begin{scaletikzpicturetowidth}{\textwidth}
\begin{tikzpicture}[scale=\tikzscale,font=\large]
\begin{loglogaxis}[
xlabel={Degrees of freedom},
ylabel={Relative $L^2$ error},
xmajorgrids,
ymajorgrids,
legend style={fill=white, fill opacity=0.6, draw opacity=1, text opacity=1},
legend pos=north east
]
\addplot[black, mark=*, very thick, mark options={scale=1.5}, fill opacity=0.6] table [x index = {0}, y index={1}, col sep=comma] {./Results/BiharmonicStatic/l2errors_p3.csv};
\addlegendentry{$M=1$};
\addplot[color3, mark=diamond*, very thick] table [x index = {0}, y index={3}, col sep=comma] {./Results/BiharmonicStatic/l2errors_p3.csv};
\addlegendentry{$q=2$};
\addplot[color4, mark=diamond*, very thick] table [x index = {0}, y index={4}, col sep=comma] {./Results/BiharmonicStatic/l2errors_p3.csv};
\addlegendentry{$q=3$};
\addplot[color5, mark=diamond*, very thick] table [x index = {0}, y index={5}, col sep=comma] {./Results/BiharmonicStatic/l2errors_p3.csv};
\addlegendentry{$q=4$};
\addplot[color6, mark=diamond*, very thick] table [x index = {0}, y index={6}, col sep=comma] {./Results/BiharmonicStatic/l2errors_p3.csv};
\addlegendentry{$q=5$};
\logLogSlopeTriangle{0.7}{0.2}{0.14}{2}{4}{};
\end{loglogaxis}
\end{tikzpicture} \end{scaletikzpicturetowidth}
\end{minipage}
\end{minipage}
\caption{\label{fig:plate_varying_hs_p3}Plate bending problem~\cref{eq:PlateWeakForm}. Relative errors for $p=3$, $M=5$, and the manufactured solution $u(\bmx) = \sin\left(\pi \,x_1\right)\,\mathrm{sinh}\left(\pi \,x_2\right)$}
\end{figure}
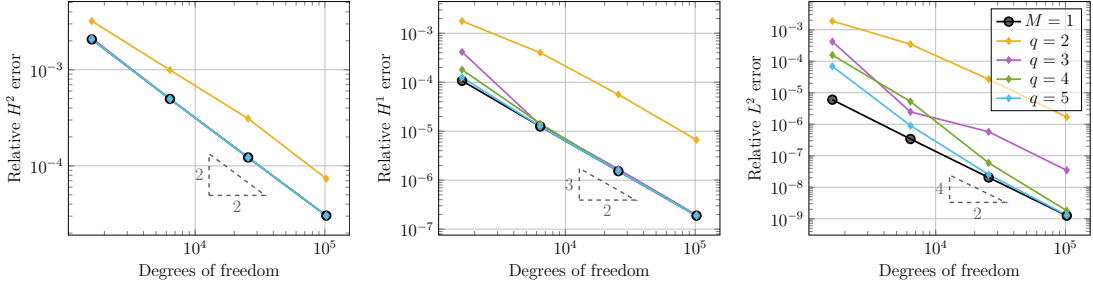

\subsection{Numerical experiments} %
\label{sub:Biharmonic_numerical_experiments}
Let $\Omega$ be the quarter annulus domain depicted in \Cref{fig:EigenvalueGeometry}.
\Cref{fig:plate_varying_hs_p3} shows the convergence of the errors, in the $H^2$, $H^1$, and $L^2$ norms, corresponding to this geometry $\Omega$ and the manufactured solution $u(\bmx) = \sin\left(\pi \,x_1\right)\,\mathrm{sinh}\left(\pi \,x_2\right)$.
Even though the kernel is not preserved, the numerical results we witnessed are similar to those documented for the $p=3$ experiments in the Poisson setting (see top row of \Cref{fig:varying_hs_p3}).
For instance, notice that the surrogate error with $q=2$ is parallel to the reference IGA error ($M=1$), in both the $H^1(\Omega)$ and $L^2(\Omega)$ norms.

\begin{remark}
Although we will not provide any rigorous analysis, if we recall \Cref{rem:StructureScalingLoss}, the similarity between our Poisson results and those above may still appear somewhat surprising.
Indeed, since only the zero row sum property is inherited in the surrogate $\tilde{\sfA}$ when using \cref{eq:DefinitionOfSurrogateMatrixSymmetricKernel}, we can not improve on the upper bound in \Cref{lem:FiniteDifference}.
Had the entire kernel of $\sfA$ been preserved in $\tilde{\sfA}$, we conjecture that an optimal form of this bound would involve an $h^{4-n}$ scaling factor.
Such a factor should permit a surrogate solution $\tilde{u}_h$ of two $h$-orders higher accuracy.
\end{remark}

\section{Stokes' flow} %
\label{sec:stokes_equation}
In this section, we consider a surrogate IGA discretization of Stokes' flow in a domain $\Omega \subset \R^n$.
Given a viscosity $\mu\in \R_{>0}$, a function $\bff \in \big[L^2(\Omega)\big]^n$, and a velocity field on the boundary $\bfg\in \big[H^{\onehalf}(\bdry\Omega)\big]^n$, $\int_{\partial\Omega} \bfg \cdot \bmn \dd s = 0$, the corresponding weak form is the following:
\begin{equation}
\label{eq:StokesWeakForm}
\left\{
\begin{alignedat}{3}
	&\text{Find } \bfu\in [H^1(\Omega)]^n \text{ and } p\in L^2(\Omega)\sspace/\sspace\sspace \R\
	\text{ satisfying }
	\tr\, \bfu = \bfg\,
	\text{ and }\\
	&a(\bfu,\bfv) + b(p,\bfv) + b(q,\bfu) = F(\bfv)
	\quad
	\text{for all }
	\bfv\in [H^1_0(\Omega)]^n \text{ and } q\in L^2(\Omega)
	\,,
	\end{alignedat}
\right.
\end{equation}
where $a(\bfu,\bfv) = \int_\Omega \mu\sspace \nabla \bfu \cdot \nabla \bfv \dd x$, $b(p,\bfv) = \int_\Omega p\, \nabla\cdot \bfv \dd x$, and $F(\bfv) = \int_\Omega \bff \cdot \bfv \dd x$.
In this scenario, the pressure is not unique up to a constant, therefore we enforce the pressure to have zero mean value, i.e., $\int_\Omega p \dd x = 0$.

\subsection{Surrogate divergence matrices} %
Since no symmetry can be exploited, the surrogate divergence matrices $\sfB$ are constructed by employing definition~\cref{eq:DefinitionOfSurrogateMatrix}.
Similarly, as in the mass term arising in \Cref{sec:transverse_vibrations_of_an_isotropic_membrane}, we have a surrogate reproduction property for the divergence form $b(q,\bfu)$, when the geometry map is described by a polynomial.
This property is formalized in the following corollary of \Cref{prop:PolynomialReproduction}.
\begin{corollary}
\label{lem:stokesreproduction}
Assume that the domain mapping $\bvarphi\colon \hat{\Omega} \rightarrow \Omega$ is defined through a polynomial of order $p$, i.e., $\bvarphi \in \big[\mcQ_p(\hat{\Omega})\big]^n$.
Let $\sfB$ be the coefficient matrix arising from the discretization of $b(\cdot,\cdot)$ and $\tilde{\sfB}$ the corresponding surrogate matrix.
If $q \geq (n-1)p$, it holds that $\sfB = \tilde{\sfB}$.
\end{corollary}
\begin{proof}
In the following, we only consider the cases $n=2$ and $n=3$.
Let $\bmJ(\hat{\bmx})$ be the Jacobian of $\bvarphi(\hat{\bmx})$.
Assuming a gradient preserving transformation to the reference domain, the divergence $\nabla\cdot \bfu$ in the physical domain is transformed to $\tr(\bmJ^{-1}\hat{\nabla}\hat{\bfu})$, where $\tr$ is the trace.
Using the property $\bmJ^{-1} = \det(\bmJ)^{-1}\adj(\bmJ)$, where $\adj(\bmJ)$ is the adjugate of $\bmJ$, the bilinear form $b(\cdot,\cdot)$ may be written as
\begin{align}
b(q,\bfu) &= \int_\Omega p\, \nabla\cdot \bfv \dd \bmx = \int_{\hat{\Omega}} \hat{p}\,  \tr\left(\bmJ^{-1}\hat{\nabla}\hat{\bfu}\right) \det(\bmJ)\dd \hat{\bmx} = \int_{\hat{\Omega}} \hat{p}\, \tr\left(\adj(\bmJ)\hat{\nabla}\hat{\bfu}\right) \dd \hat{\bmx}.
\end{align}
It remains to show that $G(\hat{\bmx}, \hat{p}(\hat{\bmy}), \hat{\bfu}(\hat{\bmy})) = \hat{p}(\hat{\bmy})\, \tr\left(\adj(\bmJ(\hat{\bmx}))\hat{\nabla}\hat{\bfu}(\hat{\bmy})\right)$ is a polynomial of degree $(n-1)p$ in the $\hat{\bmx}$-variable.
Applying \cref{prop:PolynomialReproduction} yields the desired reproduction property.
The trace operator $\tr$ is linear, thus it suffices to analyze the entries of $\adj(\bmJ)$.
In 2D, the components of $\adj(\bmJ)$ and $\bmJ$ only differ by their position and sign.
Since each component of $\bmJ$ is an element of $\mcQ_p(\hat{\Omega})$, we conclude that each component of $\adj(\bmJ)$ is also in $\mcQ_p(\hat{\Omega})$.
In 3D, the components of $\adj(\bmJ)$ are made up of determinants of $2 \times 2$ sub-matrices of $\bmJ$.
Taking the trace yields
\begin{equation}
\tr\left(\adj(\bmJ)\right) = \det\begin{pmatrix}
\bmJ_{22} & \bmJ_{23}\\
\bmJ_{32} & \bmJ_{33}
\end{pmatrix} + \det\begin{pmatrix}
\bmJ_{11} & \bmJ_{13}\\
\bmJ_{31} & \bmJ_{33}
\end{pmatrix} + \det\begin{pmatrix}
\bmJ_{11} & \bmJ_{12}\\
\bmJ_{21} & \bmJ_{22}
\end{pmatrix}.
\end{equation}
For the first summand, we have $\bmJ_{22} \in \mcP_{p} \otimes \mcP_{p-1} \otimes \mcP_{p}$, $\sfJ_{23} \in \mcP_{p} \otimes \mcP_{p} \otimes \mcP_{p-1}$, $\bmJ_{32} \in \mcP_{p} \otimes \mcP_{p-1} \otimes \mcP_{p}$, and $\bmJ_{33} \in \mcP_{p} \otimes \mcP_{p} \otimes \mcP_{p-1}$.
From this it follows that $\bmJ_{22} \cdot \bmJ_{33} \in \mcP_{2p} \otimes \mcP_{2p-1} \otimes \mcP_{2p-1}$ and $\bmJ_{23} \cdot \bmJ_{32} \in \mcP_{2p} \otimes \mcP_{2p-1} \otimes \mcP_{2p-1}$.
This means that $\bmJ_{22} \cdot \bmJ_{33} - \bmJ_{23} \cdot \bmJ_{32} \in \mcQ_{2p}(\hat{\Omega})$.
Applying the same arguments to the other summands finally yields that $\tr\left(\adj(\bmJ)\right) \in \mcQ_{2p}(\hat{\Omega})$.
\end{proof}
In order to discretize \cref{eq:StokesWeakForm}, an inf-sup stable space pair is required.
For this purpose, we choose the isogeometric \emph{subgrid} element as described in \cite{bressan2013isogeometric}.
In this discretization, the velocity field is defined on a subgrid of the pressure where each pressure element is subdivided into $2^n$ elements.
This allows for using a velocity space of order $p+1$ with $C^{p}(\hat{\Omega})$ regularity and a pressure space of order $p$ with $C^{p-1}(\hat{\Omega})$ regularity.

We do not provide \emph{a priori} error estimates for the Stokes problem, but in the case that the divergence matrix is reproduced one would follow similar arguments as presented for the Poisson problem.
In the scenario where the divergence matrix is not reproduced, further work is necessary.
However, the results in the next subsection suggest that the reproduction is not required in order to obtain an optimal order of convergence. 

\subsection{Numerical experiments} %
Our computational study of Stokes' flow is comprised of two separate experiments.
In the first experiment, we provide a smooth manufactured solution in order to investigate convergence rates.
In the second example, we consider a lid-driven cavity benchmark problem.
Due to the discontinuous boundary conditions, the solution of this problem has singularities at two corners of the domain.
Both examples are computed on the domain shown in \Cref{fig:StokesGeometry} which was constructed by a Coons patch with a cubic boundary parameterization.
We discretize the problem using the aforementioned subgrid element with a third order velocity and second order pressure.
The viscosity is set to $\mu = 1$ for all scenarios.

In the first example, the manufactured solution is chosen to be
\begin{align}
\label{eq:StokesSolution}
\bfu(x,y) = \left[\begin{matrix}\frac{\sin{\left (x \right )} \cos{\left (y \right )}}{x + 1}\\
\frac{\left(\left(x + 1\right) \cos{\left (x \right )} - \sin{\left (x \right )}\right) \sin{\left (y \right )}}{\left(x + 1\right)^{2}}\end{matrix}\right],\quad
p(x,y) = y \sinh{\left (x \right )} + C_{p}.
\end{align}
Note that this solution satisfies $\nabla\cdot \bfu = 0$ and the constant $C_{p} \in \R$ is chosen such that the pressure mean is zero.
The Dirichlet boundary condition $\bfg$ and the right hand side $\bff$ are set accordingly to match the manufactured solution.
In \Cref{fig:stokesconvergence}, we present convergence plots for the velocity and pressure separately for different surrogate orders $q$ and fixed $M = 5$.
For reference, we also include the standard discretization with $M = 1$ in these plots.
We observe the expected convergence orders for all $q \geq 2$.
In agreement with \Cref{lem:stokesreproduction}, the divergence matrices were perfectly reproduced for every $q \geq 3$.

\begin{figure}
\centering
\includegraphics[trim=0.5cm 2.4cm 1cm 3cm,clip=true,width=0.6\textwidth]{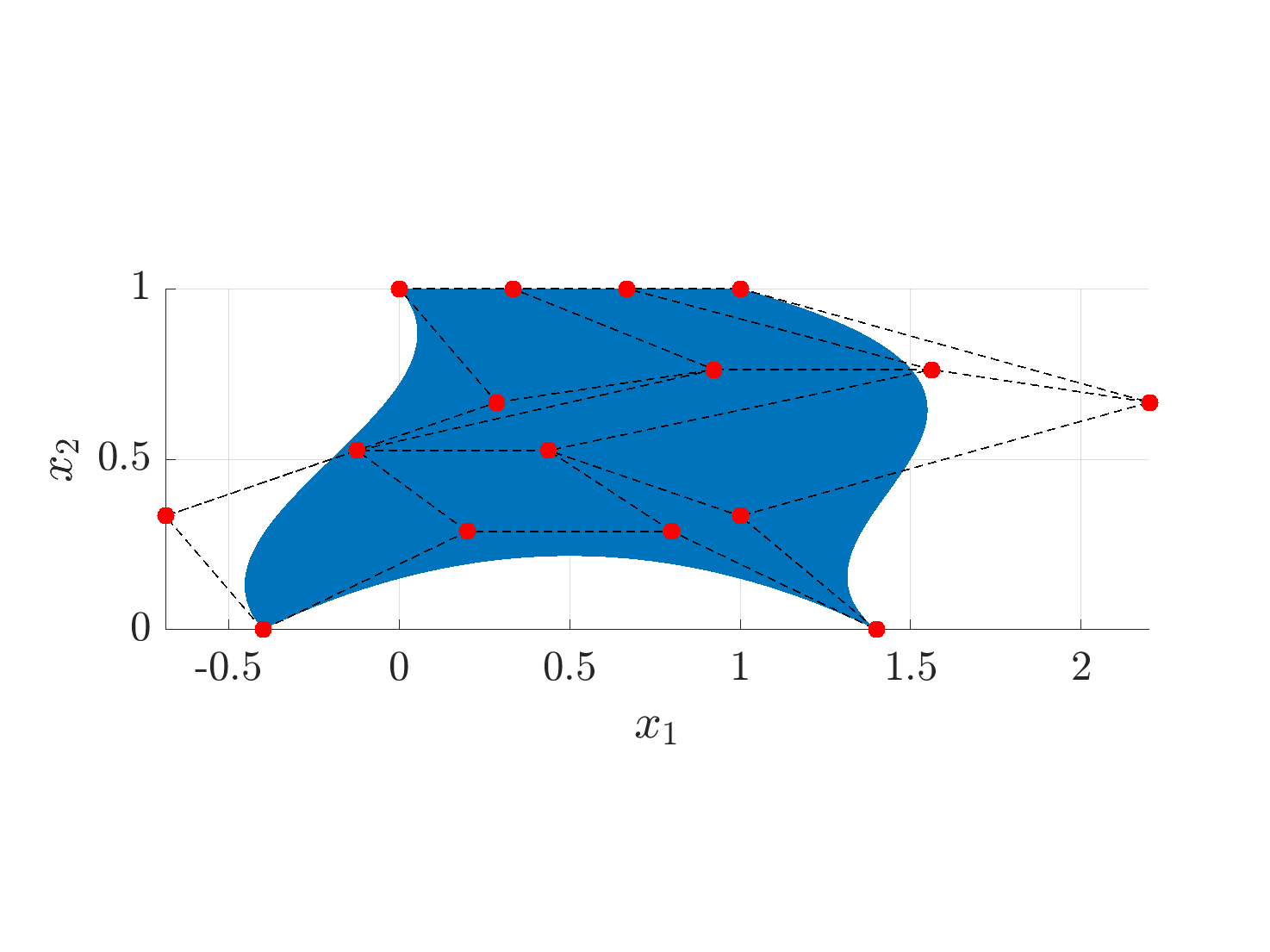}
\caption{\label{fig:StokesGeometry}The computational domain $\Omega$ used in the Stokes flow problem \cref{eq:StokesWeakForm}.}
\end{figure}
\begin{figure}%
\centering
\begin{minipage}{0.32\textwidth}
\begin{scaletikzpicturetowidth}{\textwidth}
\begin{tikzpicture}[scale=\tikzscale,font=\large]
\begin{loglogaxis}[
xlabel={Degrees of freedom},
ylabel={Relative velocity $L^2$ error},
xmajorgrids,
ymajorgrids,
title={Velocity},
legend style={fill=white, fill opacity=0.6, draw opacity=1, text opacity=1},
legend pos=north east
]
\addplot[black, mark=*, very thick, mark options={scale=1.5}, fill opacity=0.6] table [x index = {0}, y index={1}, col sep=comma] {./Results/StokesStatic2D/l2errors_v_p2.csv};
\addlegendentry{$M=1$};
\addplot[color2, mark=diamond*, very thick] table [x index = {0}, y index={2}, col sep=comma] {./Results/StokesStatic2D/l2errors_v_p2.csv};
\addlegendentry{$q=1$};
\addplot[color3, mark=diamond*, very thick] table [x index = {0}, y index={3}, col sep=comma] {./Results/StokesStatic2D/l2errors_v_p2.csv};
\addlegendentry{$q=2$};
\addplot[color4, mark=diamond*, very thick] table [x index = {0}, y index={4}, col sep=comma] {./Results/StokesStatic2D/l2errors_v_p2.csv};
\addlegendentry{$q=3$};
\addplot[color5, mark=diamond*, very thick] table [x index = {0}, y index={5}, col sep=comma] {./Results/StokesStatic2D/l2errors_v_p2.csv};
\addlegendentry{$q=4$};
\addplot[color6, mark=diamond*, very thick] table [x index = {0}, y index={6}, col sep=comma] {./Results/StokesStatic2D/l2errors_v_p2.csv};
\addlegendentry{$q=5$};
\logLogSlopeTriangle{0.7}{0.2}{0.11}{2}{4}{};
\legend{}; %
\end{loglogaxis}
\end{tikzpicture}
 \end{scaletikzpicturetowidth}
\end{minipage}
\qquad
\begin{minipage}{0.32\textwidth}
\begin{scaletikzpicturetowidth}{\textwidth}
\begin{tikzpicture}[scale=\tikzscale,font=\large]
\begin{loglogaxis}[
xlabel={Degrees of freedom},
ylabel={Relative pressure $L^2$ error},
xmajorgrids,
ymajorgrids,
title={Pressure},
legend style={fill=white, fill opacity=0.6, draw opacity=1, text opacity=1},
legend pos=north east
]
\addplot[black, mark=*, very thick, mark options={scale=1.5}, fill opacity=0.6] table [x index = {0}, y index={1}, col sep=comma] {./Results/StokesStatic2D/l2errors_p_p2.csv};
\addlegendentry{$M=1$};
\addplot[color2, mark=diamond*, very thick] table [x index = {0}, y index={2}, col sep=comma] {./Results/StokesStatic2D/l2errors_p_p2.csv};
\addlegendentry{$q=1$};
\addplot[color3, mark=diamond*, very thick] table [x index = {0}, y index={3}, col sep=comma] {./Results/StokesStatic2D/l2errors_p_p2.csv};
\addlegendentry{$q=2$};
\addplot[color4, mark=diamond*, very thick] table [x index = {0}, y index={4}, col sep=comma] {./Results/StokesStatic2D/l2errors_p_p2.csv};
\addlegendentry{$q=3$};
\addplot[color5, mark=diamond*, very thick] table [x index = {0}, y index={5}, col sep=comma] {./Results/StokesStatic2D/l2errors_p_p2.csv};
\addlegendentry{$q=4$};
\addplot[color6, mark=diamond*, very thick] table [x index = {0}, y index={6}, col sep=comma] {./Results/StokesStatic2D/l2errors_p_p2.csv};
\addlegendentry{$q=5$};
\logLogSlopeTriangle{0.7}{0.2}{0.11}{2}{3}{};
\end{loglogaxis}
\end{tikzpicture}
 \end{scaletikzpicturetowidth}
\end{minipage}
\caption{\label{fig:stokesconvergence}Stokes' flow problem \cref{eq:StokesWeakForm} with manufactured solution \cref{eq:StokesSolution}. The velocity is discretized with $p=3$ and the pressure with $p=2$. Relative $L^2$ velocity and pressure errors for $M = 5$.}
\end{figure}
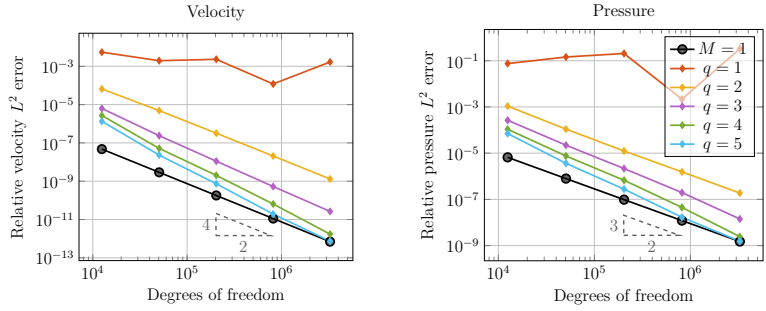

In the second example, we consider a lid-driven cavity benchmark on the domain \Cref{fig:StokesGeometry} where the fluid is driven on the top edge by constant velocity $\bfg = (1,0)^\top$ and we assume no-slip boundary conditions $\bfg = \bm{0}$ on the remaining parts of the boundary.
The degrees of freedom corresponding to the nodal basis functions in the top left and top right corner are set to zero.
Furthermore, the volume forces are neglected, i.e., $\bff = \bm{0}$.
In \Cref{fig:stokes_streamlines_ref}, we show the velocity streamlines which were computed using a standard IGA approach on a mesh with $320 \times 320$ control points.
The effect of different surrogate approaches on the velocity streamlines may be observed in \Cref{fig:stokes_streamlines_q1_m10,fig:stokes_streamlines_q2_m10,fig:stokes_streamlines_q3_m10,fig:stokes_streamlines_q1_m50,fig:stokes_streamlines_q2_m50,fig:stokes_streamlines_q3_m50,fig:stokes_streamlines_q1_m100,fig:stokes_streamlines_q2_m100,fig:stokes_streamlines_q3_m100}.
In the case $q = 3$, where the divergence matrices are \changed{in fact} reproduced, the streamlines show the same behavior as in the standard approach even for $M = 100$.
For other values of $q$ and $M$, the streamline behavior is different, but the streamlines are getting closer to the reference solution the larger $q$ and the smaller $M$ becomes.
We note that actually using an interpolation order of $q=1,2,3$ in computation is still probably not recommended for standard practice.
For instance, assembly using $q=5$ took roughly the same time as either of these lower-order choices and, in this case, the surrogate solution $(\tilde{\bfu}_h,\tilde{p}_h)$ should be expected to be even more accurate.

\begin{figure}
\centering
\begin{subfigure}[c]{\textwidth}
\centering
\qquad
\begin{minipage}{0.7\textwidth}
\includegraphics[width=\textwidth]{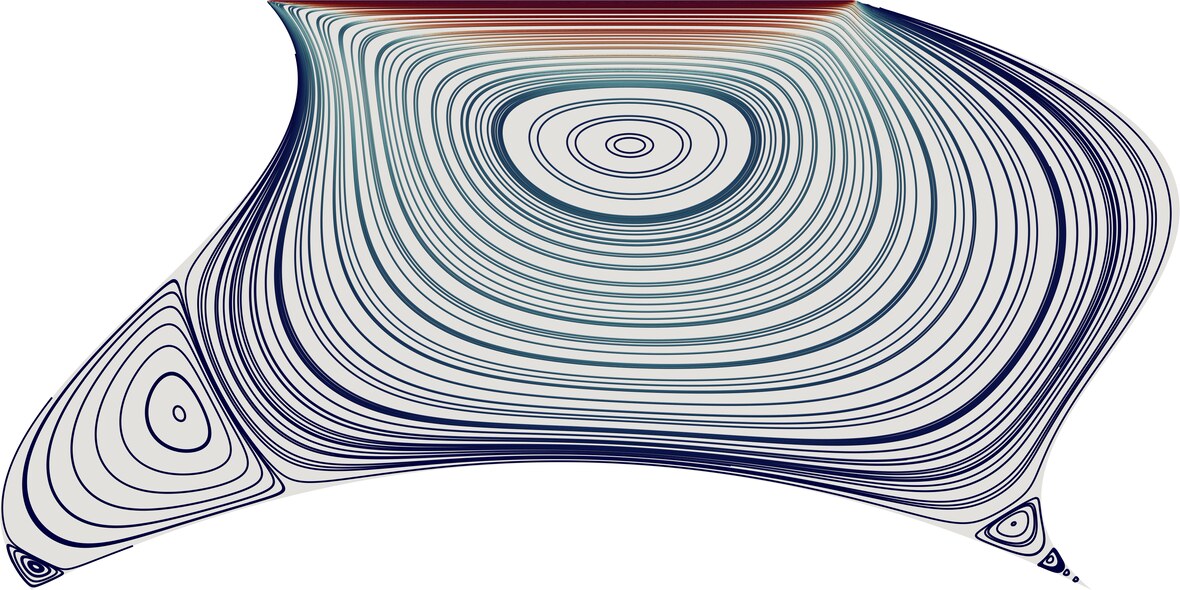}
\end{minipage}
~~
\begin{minipage}{0.12\textwidth}
\begin{tikzpicture}
\pgfmathsetlengthmacro\MajorTickLength{
  \pgfkeysvalueof{/pgfplots/major tick length} * 0.5
}
\begin{axis}[
xmin=0, xmax=0.02,
ymin=0, ymax=1,
axis on top,
scaled x ticks=false,
scaled y ticks=false,
xtick=\empty,
xticklabels=\empty,
yticklabel pos=right,
y tick label style={
  /pgf/number format/.cd,
    fixed,
    fixed zerofill,
    precision=1,
  /tikz/.cd  
},
extra y tick style={
    font=\large,
    tick style=transparent, %
    yticklabel pos=left,
    y tick label style={
        /pgf/number format/.cd,
            std,
            precision=3,
      /tikz/.cd
    }
},
ylabel={\Large $|u|$},
ylabel style={rotate=-90},
width=1.82cm,
height=5.5cm,
major tick length=\MajorTickLength
]
\addplot graphics [
includegraphics cmd=\pgfimage,
xmin=\pgfkeysvalueof{/pgfplots/xmin}, 
xmax=\pgfkeysvalueof{/pgfplots/xmax}, 
ymin=\pgfkeysvalueof{/pgfplots/ymin}, 
ymax=\pgfkeysvalueof{/pgfplots/ymax}
] {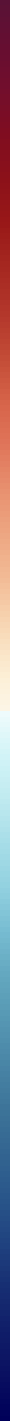};
\end{axis}
\end{tikzpicture}
 \end{minipage}
\caption{\label{fig:stokes_streamlines_ref}Reference solution with $M = 1$}
\end{subfigure}

\begin{subfigure}[c]{0.33\textwidth}
\includegraphics[width=\textwidth]{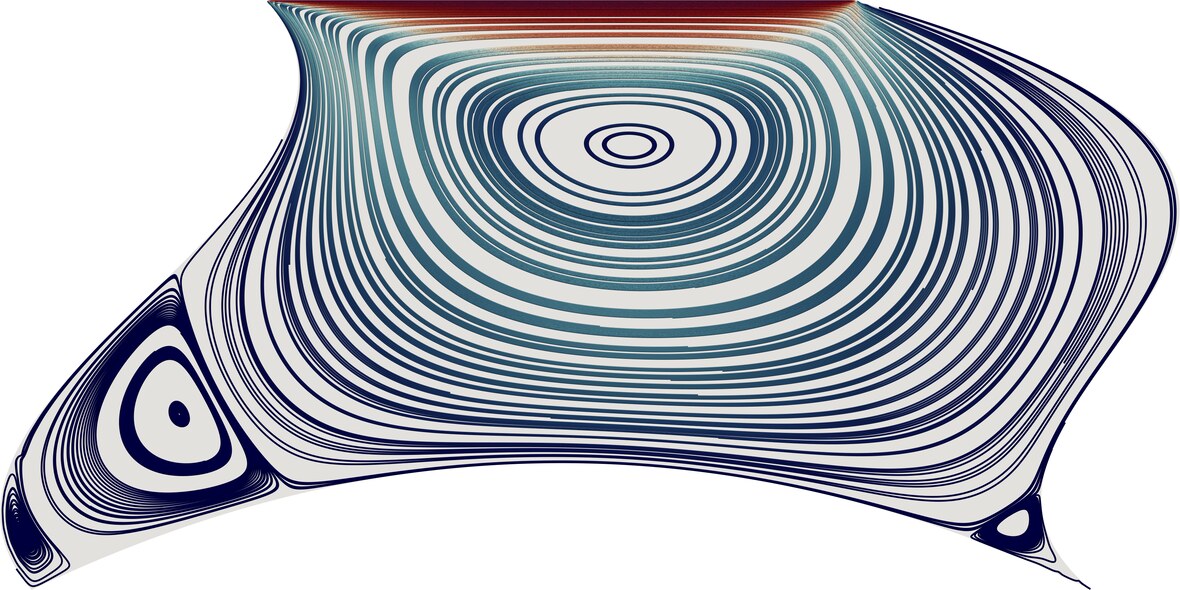}
\caption{\label{fig:stokes_streamlines_q1_m10}$M = 10$, $q = 1$}
\end{subfigure}%
\begin{subfigure}[c]{0.33\textwidth}
\includegraphics[width=\textwidth]{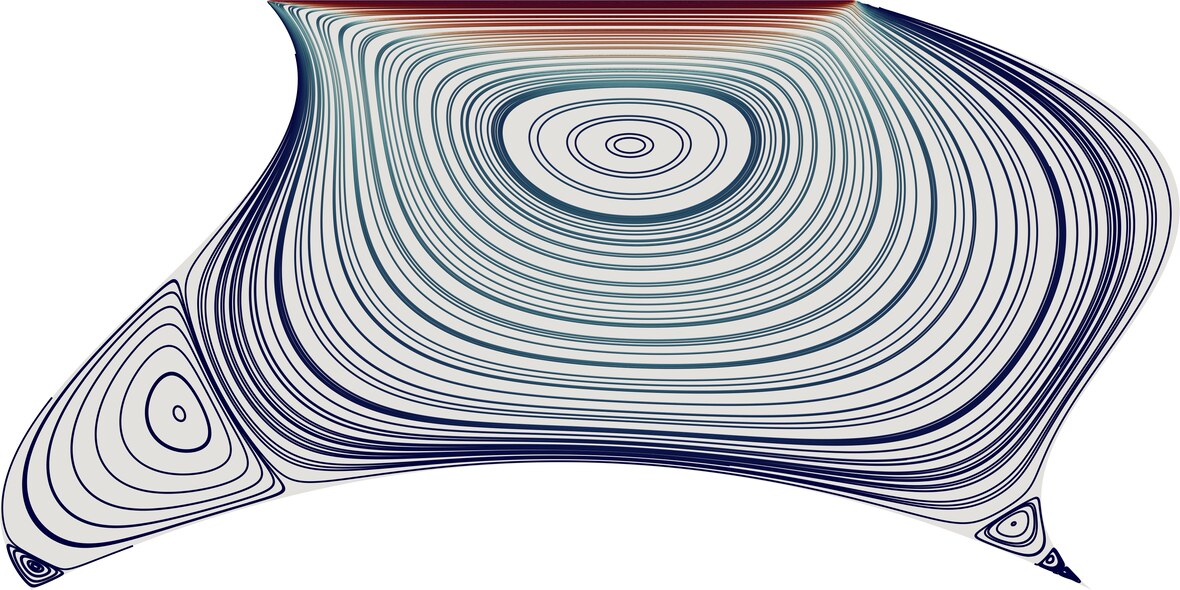}
\caption{\label{fig:stokes_streamlines_q2_m10}$M = 10$, $q = 2$}
\end{subfigure}%
\begin{subfigure}[c]{0.33\textwidth}
\includegraphics[width=\textwidth]{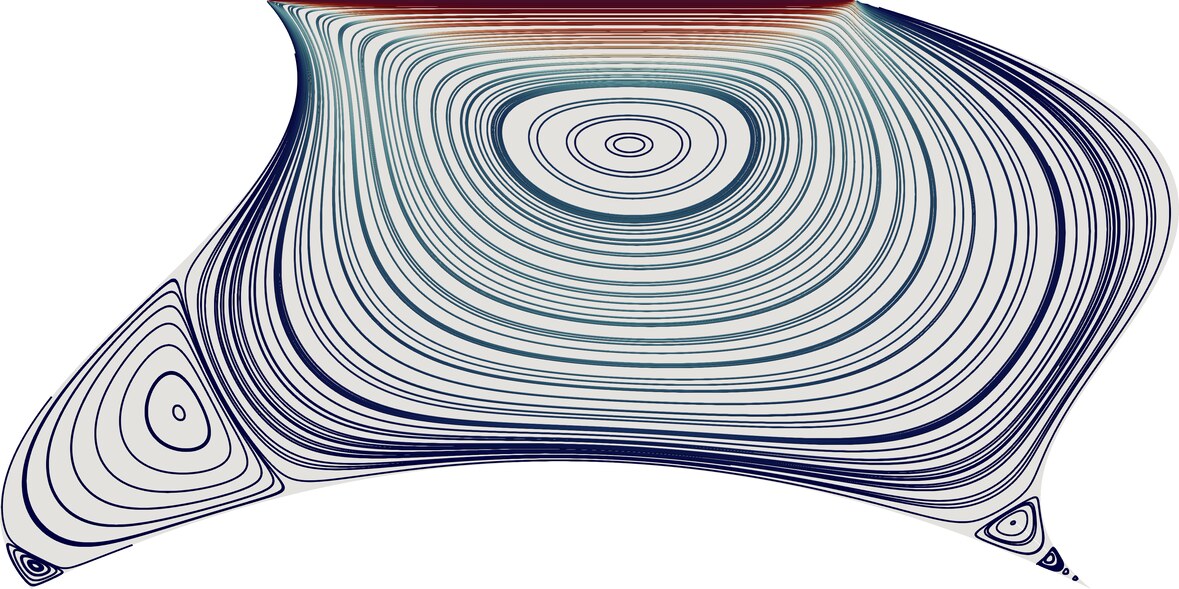}
\caption{\label{fig:stokes_streamlines_q3_m10}$M = 10$, $q = 3$}
\end{subfigure}

\begin{subfigure}[c]{0.33\textwidth}
	\includegraphics[width=\textwidth]{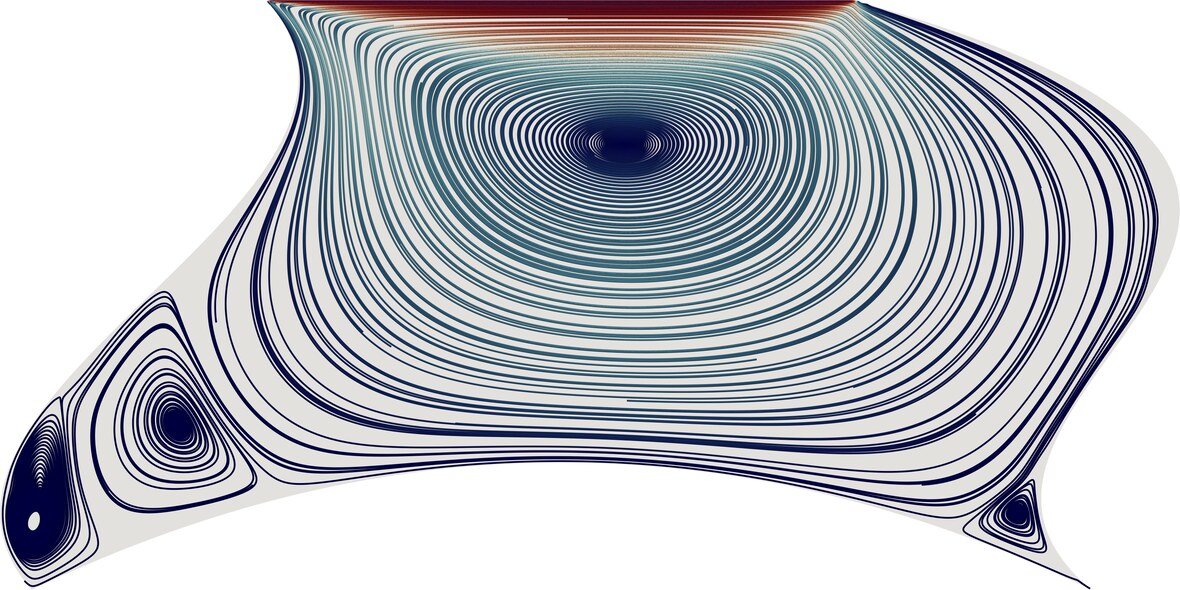}
	\caption{\label{fig:stokes_streamlines_q1_m50}$M = 50$, $q = 1$}
\end{subfigure}%
\begin{subfigure}[c]{0.33\textwidth}
	\includegraphics[width=\textwidth]{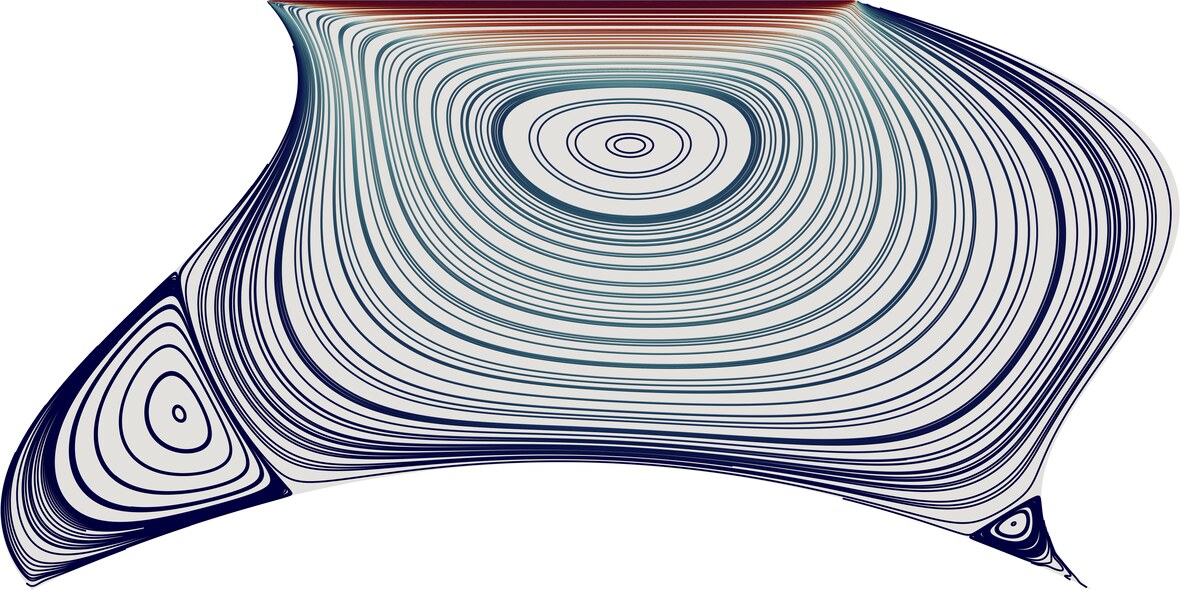}
	\caption{\label{fig:stokes_streamlines_q2_m50}$M = 50$, $q = 2$}
\end{subfigure}%
\begin{subfigure}[c]{0.33\textwidth}
	\includegraphics[width=\textwidth]{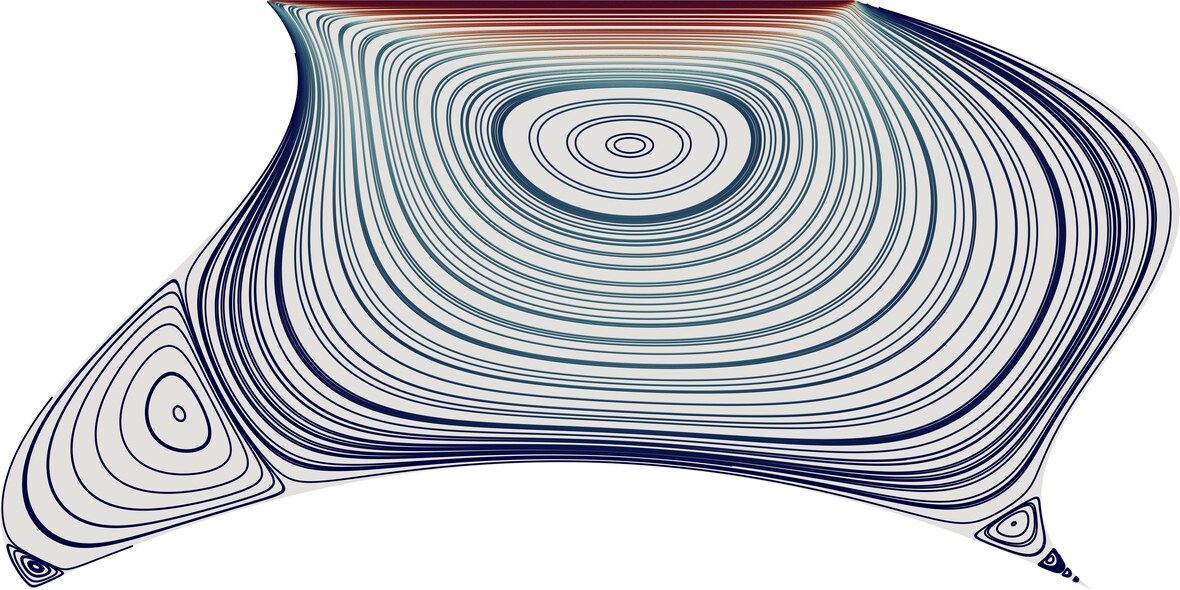}
	\caption{\label{fig:stokes_streamlines_q3_m50}$M = 50$, $q = 3$}
\end{subfigure}

\begin{subfigure}[c]{0.33\textwidth}
	\includegraphics[width=\textwidth]{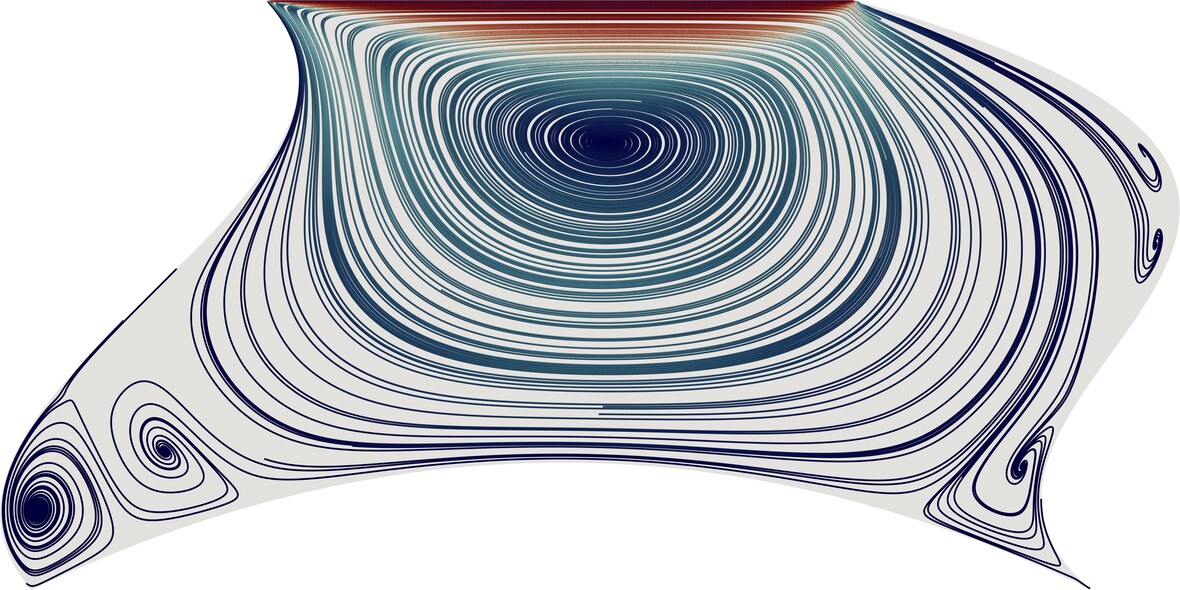}
	\caption{\label{fig:stokes_streamlines_q1_m100}$M = 100$, $q = 1$}
\end{subfigure}%
\begin{subfigure}[c]{0.33\textwidth}
	\includegraphics[width=\textwidth]{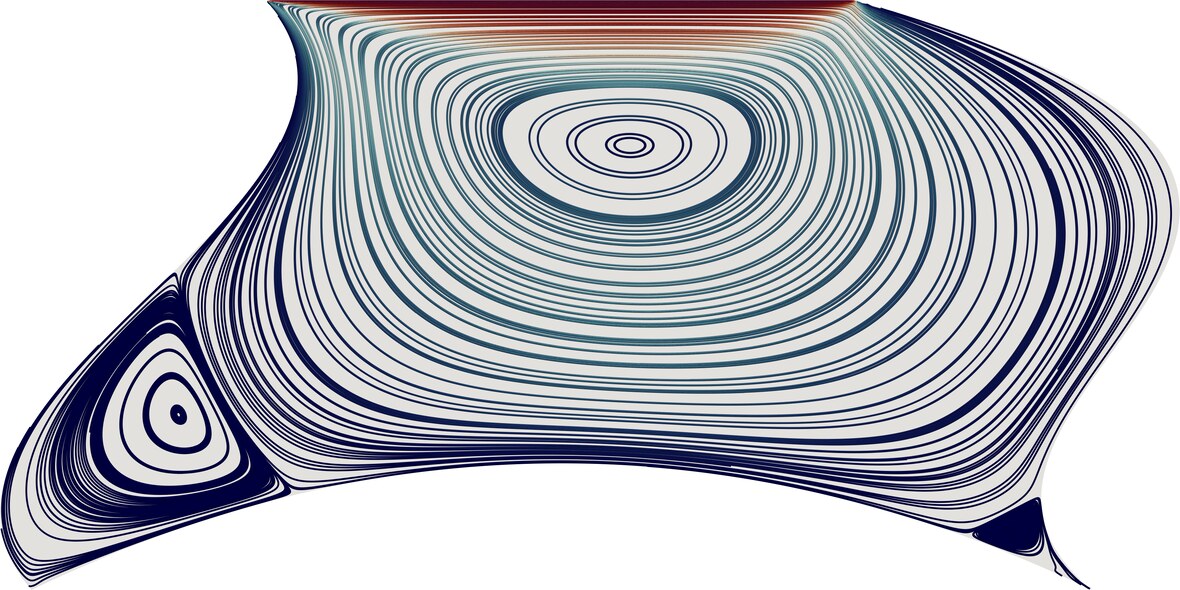}
	\caption{\label{fig:stokes_streamlines_q2_m100}$M = 100$, $q = 2$}
\end{subfigure}%
\begin{subfigure}[c]{0.33\textwidth}
	\includegraphics[width=\textwidth]{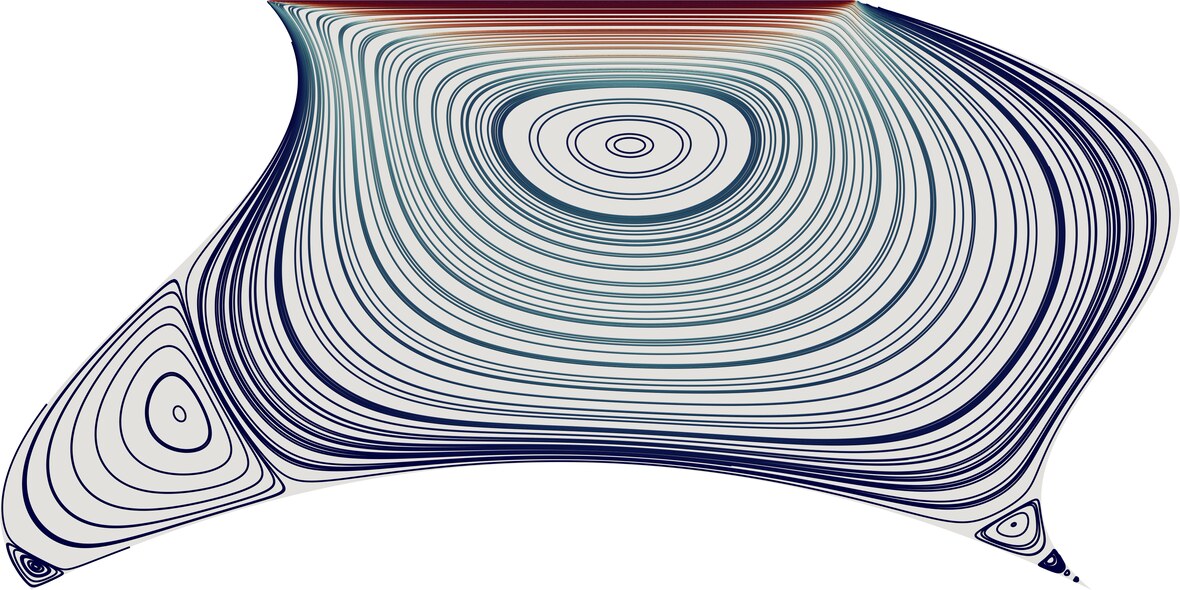}
	\caption{\label{fig:stokes_streamlines_q3_m100}$M = 100$, $q = 3$}
\end{subfigure}
\caption{\label{fig:stokes_streamlines_surr}Lid-driven cavity benchmark velocity streamlines. The velocity is discretized with $p=3$ and the pressure with $p=2$ using surrogate approaches with varying $M$ and $q$. The mesh is discretized by $320 \times 320$ control points. Note that actually using an interpolation order of $q=1,2,3$ is still probably not recommended for standard practice since more accurate results should be expected with $q=5$ while taking roughly the same time.}
\end{figure}
\afterpage{\FloatBarrier}

\appendix
\section{Marsden's identity} %
\label{app:marsden_s_identity}

The purpose of this appendix is to substantiate some of the claims made in \Cref{sub:basis_structure_nurbs}, as well as provide a complete proof of \Cref{thm:RegularityOfStencilFunctions}.
In turn, we adopt all of the notation and assumptions introduced in \Cref{sub:nurbs}.
We begin by stating \Cref{thm:Marsden}, the proof of which can be found in \cite{de1986b}; see also \cite{deBoor1993box,dahmen1985convergence,DAHMEN1983217}.

\begin{theorem}[Marsden]
\label{thm:Marsden}
Let $\psi_{k}(\hat{y}) = (\xi_{k+1}-\hat{y})\cdots(\xi_{k+p}-\hat{y})$.
For any $\hat{x},\hat{y}\in[0,1]$,
\begin{equation}
	(\hat{x}-\hat{y})^p
	=
	\sum_{k=1}^{m}
	{b}_{k}(\hat{x})\psi_{k}(\hat{y})
	\,.
\end{equation}
\end{theorem}

This theorem allows us to conclude that the expression~\cref{eq:StencilFunctionNURBS} is valid and, moreover, $w\in \mcQ_p(\tilde{\Omega}_{\bm{0}})$.
This is shown in two steps.
\begin{corollary}
\label{cor:Marsden}
Let $\Psi(\tilde{\bmy}) = \prod_{i=1}^n\prod_{j=0}^{p-1} \big(\frac{j}{m-p}-\frac{p-1}{2(m-p)}-\tilde{y}_i\big)$ and $\bmp = (p,\ldots,p)\in\bbN^n$.
Then, for any $\tilde{\bmx},\tilde{\bmy}\in\tilde{\Omega}_{\bm{0}}$,
\begin{equation}
	(\tilde{\bmx}-\tilde{\bmy})^{\bmp}
	=
	\sum_{\tilde{\bmx}_i\in\tilde{\bbX}}
	\hat{B}(\tilde{\bmx}-\tilde{\bmx}_i)\Psi(\tilde{\bmy}-\tilde{\bmx}_i)
	\,.
\label{eq:Marsden_nD}
\end{equation}
\end{corollary}
\begin{proof}
	Observe that, for each $k = p+1,\ldots,m-p$, it holds that $\tilde{x}^{(k)} + h = \tilde{x}^{(k+1)} = \xi_{k+1} + (p+1)\cdot \nicefrac{h}{2}$.
	Recall that $k = p+1,\ldots,m-p$ are exactly the indices of the cardinal B-splines $b_k(\hat{x}) = b(\hat{x}-\tilde{x}^{(k)})$.
	Therefore, by \Cref{thm:Marsden}, we see that
	\begin{equation}
		(\tilde{x}_i-\tilde{y}_i)^p
		=
		\sum_{k=p+1}^{m-p}
		b(\tilde{x}_i-\tilde{x}^{(k)})\prod_{j=0}^{p-1}\Big(\tilde{x}^{(k)} + \Big(j-\frac{p-1}{2}\Big)\cdot h-\tilde{y}_i\Big)
		\,,
	\end{equation}
	for each $i=1,\ldots,n$.
	The result now follows from the definitions of $\hat{B}$ and $\Psi$.
\end{proof}

\Cref{cor:Marsden} can be used to write out an elegant expression for any polynomial in $\mcQ_p(\tilde{\Omega}_{\bm{0}})$, in terms of cardinal B-splines.
Indeed, let $\balpha = (\alpha_1,\ldots,\alpha_n)$ be a multi-index, let $f$ be an arbitrary polynomial in $\mcQ_p(\tilde{\Omega}_{\bm{0}})$, and let $D^{\balpha}$ be the $\balpha$-derivative operator, in the variable $\tilde{\bmy}$.
By Taylor's Theorem, it holds that
\begin{equation}
	f(\tilde{\bmx})
	=
	\sum_{i=1}^n
	\sum_{\alpha_i\leq p}
	\frac{D^{\balpha}f(\tilde{\bmy})}{\balpha!}(\tilde{\bmx}-\tilde{\bmy})^{\balpha}
	\,,
\label{eq:AAA}
\end{equation}
for every $\tilde{\bmx},\tilde{\bmy} \in \tilde{\Omega}_{\bm{0}}$.
Next, applying $D^{\bmp-\balpha}$ to both sides of~\cref{eq:Marsden_nD}, we find that
\begin{equation}
	\frac{(-1)^{|\bmp-\balpha|}\bmp!}{\balpha!} (\tilde{\bmx}-\tilde{\bmy})^{\balpha}
	=
	\sum_{\tilde{\bmx}_i\in\tilde{\bbX}}
	\hat{B}(\tilde{\bmx}-\tilde{\bmx}_i) D^{\bmp-\balpha}\Psi(\tilde{\bmy}-\tilde{\bmx}_i)
	\,.
\label{eq:BBB}
\end{equation}
Together,~\cref{eq:AAA,eq:BBB} imply $f(\tilde{\bmx}) = \sum_{\tilde{\bmx}_i\in\tilde{\bbX}} \hat{B}(\tilde{\bmx}-\tilde{\bmx}_i)(Lf)(\tilde{\bmx}_i)$, where
\begin{equation}
	(Lf)(\tilde{\bmx})
	=
	\sum_{i=1}^n
	\sum_{\alpha_i\leq p}
	\frac{(-D)^{\bmp-\balpha}\Psi(\tilde{\bmy}-\tilde{\bmx})}{\bmp!}
	D^{\balpha}f(\tilde{\bmy})
	\,.
\end{equation}

We may now state our second corollary of \Cref{thm:Marsden}.
\begin{corollary}
\label{eq:MarsdenCor3}
Let $W(\hat{\bmx}) = \sum_{j=1}^N w_j \hat{B}_j(\hat{\bmx})$.
If $W\in \mcQ_p(\hat{\Omega})$, then there exists a polynomial $w\in \mcQ_p(\tilde{\Omega}_{\bm{0}})$ such that $w_i = w(\tilde{\bmx}_i)$, for each $\tilde{\bmx}_i \in \tilde{\bbX}$.
Moreover, there exists a constant $C$, depending only on $p$, such that
\begin{equation}
	\|w\|_{W^{r,\infty}(\tilde{\Omega}_{\bm{0}})}
	\leq
	C
	\|W\|_{W^{r,\infty}(\hat{\Omega})}
	\qquad
	\text{for all }
	r \geq p
	\,.
\label{eq:BoundsOnWeightFunctions}
\end{equation}
\end{corollary}
\begin{proof}
Clearly, $w = LW \in \mcQ_p(\tilde{\Omega}_{\bm{0}})$.
Therefore, one readily determines that $\|w\|_{L^\infty(\tilde{\Omega}_{\bm{0}})} \leq C \|W\|_{W^{p,\infty}(\tilde{\Omega}_{\bm{0}})}$, for some constant $C$, depending only on $p$.
Due to the equivalence of norms on finite dimensional vector spaces (note that the dimension of $\mcQ_p(\tilde{\Omega}_{\bm{0}})$ depends only on $p$) and the fact $\tilde{\Omega}_{\bm{0}}\subset \overline{\hat{\Omega}}$, we immediately arrive at~\cref{eq:BoundsOnWeightFunctions}.
\end{proof}

We may now complete the proof of \Cref{thm:RegularityOfStencilFunctions}.

\begin{proof}[Proof of \Cref{thm:RegularityOfStencilFunctions}]
Let $\tilde{\bmx} \in \tilde{\Omega}$ be arbitrary.
Then, for any $\bdelta\in\scD$, we have that $\tilde{\bmx} + \bdelta \in \tilde{\Omega}_{\bm{0}}$.
Moreover, for every multi-index $|\balpha| = r$, the product rule can be used to show that
\begin{align}
	D^\balpha
	\Phi_\bdelta(\tilde{\bmx})
	&=
	D^\balpha\bigg[
	w(\tilde{\bmx})w(\tilde{\bmx}+\bdelta)
	\int_{\hat{\omega}_{\bdelta}} \hat{\nabla} \bigg(\frac{\hat{B}(\hat{\bmy})}{W(\tilde{\bmx} + \hat{\bmy})}\bigg)^\top {K}(\tilde{\bmx} + \hat{\bmy})\, \hat{\nabla} \bigg(\frac{\hat{B}_\bdelta(\hat{\bmy})}{W(\tilde{\bmx} + \hat{\bmy})}\bigg) \dd \hat{\bmy}
	\bigg]
	\\
	&\leq
	C\cdot
	\|w\|_{W^{r,\infty}(\tilde{\Omega}_{\bm{0}})}^2
	\sspace
	\|K\|_{W^{r,\infty}(\hat{\Omega})}
	\sspace
	\|W\|_{W^{r,\infty}(\hat{\Omega})}^2
	\sspace
	\|\hat{\nabla} W\|_{W^{r,\infty}(\hat{\Omega})}^2
	\\
	&
	~\cdot
	\big(
	\|\hat{\nabla} \hat{B}\cdot \hat{\nabla} \hat{B}_\bdelta\|_{L^1(\hat{\omega}_\bdelta)}
	+
	\|\hat{\nabla} \hat{B}\cdot \hat{B}_\bdelta\|_{L^1(\hat{\omega}_\bdelta)}
	+
	\|\hat{\nabla} \hat{B}_\bdelta\cdot \hat{B}\|_{L^1(\hat{\omega}_\bdelta)}
	+
	\|\hat{B}\cdot \hat{B}_\bdelta\|_{L^1(\hat{\omega}_\bdelta)}
	\big)
	\,,
\end{align}
for some $C$ depending only on $\balpha$.
Since both functions $K$, and $W$ are determined by the choice of $\Omega$, \Cref{eq:MarsdenCor3} and a scaling argument show that $\|D^{\balpha}\Phi_\bdelta\|_{L^\infty(\tilde{\Omega})}\leq C h^{n-2}$, where $C$ now depends on $p$, $\balpha$, and \changed{$\bvarphi$}.
\Cref{lem:SplineBAE} now completes the proof.~
\end{proof}

\section*{Acknowledgments}
This project has received funding from the European Union's Horizon 2020 research and innovation programme under grant agreement No 800898.
This work was also partly supported by the German Research Foundation through
the Priority Programme 1648 ``Software for Exascale Computing'' (SPPEXA) and by grant WO671/11-1.

\phantomsection%
\bibliographystyle{abbrv}
\bibliography{main}

\begin{thebibliography}{10}

\bibitem{githubdrzisga}
Fork of the {GeoPDEs} git repository including the surrogate branch.
\newblock \url{https://github.com/drzisga/geopdes}, 2019.
\newblock [Online; accessed March 13th 2019].

\bibitem{ANTOLIN2015817}
P.~Antolin, A.~Buffa, F.~Calabrò, M.~Martinelli, and G.~Sangalli.
\newblock Efficient matrix computation for tensor-product isogeometric
  analysis: {T}he use of sum factorization.
\newblock {\em Computer Methods in Applied Mechanics and Engineering},
  285:817--828, 2015.

\bibitem{auricchio2012simple}
F.~Auricchio, F.~Calabro, T.~J. Hughes, A.~Reali, and G.~Sangalli.
\newblock A simple algorithm for obtaining nearly optimal quadrature rules for
  {NURBS}-based isogeometric analysis.
\newblock {\em Computer Methods in Applied Mechanics and Engineering},
  249:15--27, 2012.

\bibitem{bauer2018large}
S.~Bauer, M.~Huber, S.~Ghelichkhan, M.~Mohr, U.~R{\"u}de, and B.~Wohlmuth.
\newblock Large-scale simulation of mantle convection based on a new
  matrix-free approach.
\newblock {\em Journal of Computational Science}, 31:60--76, 2019.

\bibitem{bauer2018new}
S.~Bauer, M.~Huber, M.~Mohr, U.~R{\"u}de, and B.~Wohlmuth.
\newblock A new matrix-free approach for large-scale geodynamic simulations and
  its performance.
\newblock In {\em International Conference on Computational Science}, pages
  17--30. Springer, 2018.

\bibitem{bauer2017two}
S.~Bauer, M.~Mohr, U.~R{\"u}de, J.~Weism{\"u}ller, M.~Wittmann, and
  B.~Wohlmuth.
\newblock A two-scale approach for efficient on-the-fly operator assembly in
  massively parallel high performance multigrid codes.
\newblock {\em Applied Numerical Mathematics}, 122:14--38, 2017.

\bibitem{bressan2013isogeometric}
A.~Bressan and G.~Sangalli.
\newblock Isogeometric discretizations of the stokes problem: stability
  analysis by the macroelement technique.
\newblock {\em IMA Journal of Numerical Analysis}, 33(2):629--651, 2013.

\bibitem{bressan2018sum}
A.~Bressan and S.~Takacs.
\newblock Sum-factorization techniques in isogeometric analysis.
\newblock {\em arXiv preprint arXiv:1809.05471}, 2018.

\bibitem{bungartz2004sparse}
H.-J. Bungartz and M.~Griebel.
\newblock Sparse grids.
\newblock {\em Acta numerica}, 13:147--269, 2004.

\bibitem{butler2012posteriori}
T.~Butler, P.~Constantine, and T.~Wildey.
\newblock A posteriori error analysis of parameterized linear systems using
  spectral methods.
\newblock {\em SIAM Journal on Matrix Analysis and Applications},
  33(1):195--209, 2012.

\bibitem{Calabro2017}
F.~Calabrò, G.~Sangalli, and M.~Tani.
\newblock Fast formation of isogeometric {G}alerkin matrices by weighted
  quadrature.
\newblock {\em Computer Methods in Applied Mechanics and Engineering},
  316:606--622, 2017.
\newblock Special Issue on Isogeometric Analysis: Progress and Challenges.

\bibitem{cottrell2009isogeometric}
J.~A. Cottrell, T.~J. Hughes, and Y.~Bazilevs.
\newblock {\em Isogeometric analysis: toward integration of {CAD} and {FEA}}.
\newblock John Wiley \& Sons, 2009.

\bibitem{da2014mathematical}
L.~B. Da~Veiga, A.~Buffa, G.~Sangalli, and R.~V{\'a}zquez.
\newblock Mathematical analysis of variational isogeometric methods.
\newblock {\em Acta Numerica}, 23:157--287, 2014.

\bibitem{dahmen1980multidimensional}
W.~Dahmen, R.~De~Vore, and K.~Scherer.
\newblock Multidimensional spline approximation.
\newblock {\em SIAM Journal on Numerical Analysis}, 17(3):380--402, 1980.

\bibitem{dahmen1985convergence}
W.~Dahmen, N.~Dyn, and D.~Levin.
\newblock On the convergence rates of subdivision algorithms for box spline
  surfaces.
\newblock {\em Constructive Approximation}, 1(1):305--322, 1985.

\bibitem{DAHMEN1983217}
W.~Dahmen and C.~A. Micchelli.
\newblock Translates of multivariate splines.
\newblock {\em Linear Algebra and its Applications}, 52:217--234, 1983.

\bibitem{davis2007methods}
P.~J. Davis and P.~Rabinowitz.
\newblock {\em Methods of numerical integration}.
\newblock Courier Corporation, 2007.

\bibitem{de1986b}
C.~De~Boor.
\newblock B(asic)-spline basics.
\newblock Technical report, UW\textendash{}Madison Mathematics Research Center,
  1986.

\bibitem{deBoor1993box}
C.~De~Boor, K.~H{\"o}llig, and S.~Riemenschneider.
\newblock {\em Box splines}, volume~98.
\newblock Springer Science \& Business Media, 1993.

\bibitem{de2011geopdes}
C.~de~Falco, A.~Reali, and R.~V{\'a}zquez.
\newblock {GeoPDEs}: a research tool for isogeometric analysis of {PDE}s.
\newblock {\em Advances in Engineering Software}, 42(12):1020--1034, 2011.

\bibitem{drzisga2018surrogate}
D.~Drzisga, B.~Keith, and B.~Wohlmuth.
\newblock The surrogate matrix methodology: a priori error estimation.
\newblock {\em arXiv preprint arXiv:1902.07333}, 2019.

\bibitem{drzisga2019igasurrogateimpl}
D.~Drzisga, B.~Keith, and B.~Wohlmuth.
\newblock The surrogate matrix methodology: A reference implementation for
  low-cost assembly in isogeometric analysis.
\newblock {\em Submitted}, 2019.

\bibitem{Fahrendorf2018}
F.~Fahrendorf, L.~D. Lorenzis, and H.~Gomez.
\newblock Reduced integration at superconvergent points in isogeometric
  analysis.
\newblock {\em Computer Methods in Applied Mechanics and Engineering},
  328:390--410, 2018.

\bibitem{hiemstra2017optimal}
R.~R. Hiemstra, F.~Calabro, D.~Schillinger, and T.~J. Hughes.
\newblock Optimal and reduced quadrature rules for tensor product and
  hierarchically refined splines in isogeometric analysis.
\newblock {\em Computer Methods in Applied Mechanics and Engineering},
  316:966--1004, 2017.

\bibitem{hiemstra2019fast}
R.~R. Hiemstra, G.~Sangalli, M.~Tani, F.~Calabro, and T.~J.~R. Hughes.
\newblock Fast formation and assembly of finite element matrices with
  application to isogeometric linear elasticity.
\newblock ICES Report 19-03, The University of Texas at Austin, 2019.

\bibitem{hofreither2018black}
C.~Hofreither.
\newblock A black-box low-rank approximation algorithm for fast matrix assembly
  in isogeometric analysis.
\newblock {\em Computer Methods in Applied Mechanics and Engineering},
  333:311--330, 2018.

\bibitem{hollig2003finite}
K.~Hollig.
\newblock {\em Finite element methods with {B}-splines}, volume~26.
\newblock Siam, 2003.

\bibitem{hughes2005isogeometric}
T.~J. Hughes, J.~A. Cottrell, and Y.~Bazilevs.
\newblock Isogeometric analysis: {CAD}, finite elements, {NURBS}, exact
  geometry and mesh refinement.
\newblock {\em Computer methods in applied mechanics and engineering},
  194(39-41):4135--4195, 2005.

\bibitem{hughes2010efficient}
T.~J. Hughes, A.~Reali, and G.~Sangalli.
\newblock Efficient quadrature for {NURBS}-based isogeometric analysis.
\newblock {\em Computer methods in applied mechanics and engineering},
  199(5-8):301--313, 2010.

\bibitem{hughes2018mathematics}
T.~J. Hughes and G.~Sangalli.
\newblock Mathematics of isogeometric analysis: a conspectus.
\newblock {\em Encyclopedia of Computational Mechanics Second Edition}, pages
  1--40, 2018.

\bibitem{scipy}
E.~Jones, T.~Oliphant, P.~Peterson, et~al.
\newblock {SciPy}: Open source scientific tools for {Python}.
\newblock \url{http://www.scipy.org}, 2001--.
\newblock [Online; accessed February 11th 2019].

\bibitem{Karatarakis2014}
A.~Karatarakis, P.~Karakitsios, and M.~Papadrakakis.
\newblock {GPU} accelerated computation of the isogeometric analysis stiffness
  matrix.
\newblock {\em Computer Methods in Applied Mechanics and Engineering},
  269:334--355, 2014.

\bibitem{Mantzaflaris2014}
A.~Mantzaflaris and B.~J{\"u}ttler.
\newblock Exploring matrix generation strategies in isogeometric analysis.
\newblock In M.~Floater, T.~Lyche, M.-L. Mazure, K.~M{\o}rken, and L.~L.
  Schumaker, editors, {\em Mathematical Methods for Curves and Surfaces}, pages
  364--382, Berlin, Heidelberg, 2014. Springer Berlin Heidelberg.

\bibitem{mantzaflaris2015integration}
A.~Mantzaflaris and B.~J{\"u}ttler.
\newblock Integration by interpolation and look-up for {G}alerkin-based
  isogeometric analysis.
\newblock {\em Computer Methods in Applied Mechanics and Engineering},
  284:373--400, 2015.

\bibitem{mantzaflaris2017low}
A.~Mantzaflaris, B.~J{\"u}ttler, B.~N. Khoromskij, and U.~Langer.
\newblock Low rank tensor methods in {G}alerkin-based isogeometric analysis.
\newblock {\em Computer Methods in Applied Mechanics and Engineering},
  316:1062--1085, 2017.

\bibitem{MATTIS201836}
S.~A. Mattis and B.~Wohlmuth.
\newblock Goal-oriented adaptive surrogate construction for stochastic
  inversion.
\newblock {\em Computer Methods in Applied Mechanics and Engineering},
  339:36--60, 2018.

\bibitem{piegl2012nurbs}
L.~Piegl and W.~Tiller.
\newblock {\em The NURBS book}.
\newblock Springer Science \& Business Media, 2012.

\bibitem{sangalli2018matrix}
G.~Sangalli and M.~Tani.
\newblock Matrix-free weighted quadrature for a computationally efficient
  isogeometric $k$-method.
\newblock {\em Computer Methods in Applied Mechanics and Engineering},
  338:117--133, 2018.

\bibitem{schoenberg1946contributionsA}
I.~J. Schoenberg.
\newblock Contributions to the problem of approximation of equidistant data by
  analytic functions. {P}art {A}. on the problem of smoothing or graduation.
  {A} first class of analytic approximation formulae.
\newblock {\em Quarterly of Applied Mathematics}, 4:45--99, 1946.

\bibitem{schoenberg1946contributionsB}
I.~J. Schoenberg.
\newblock Contributions to the problem of approximation of equidistant data by
  analytic functions. {P}art {B}. {O}n the problem of osculatory interpolation.
  {A} second class of analytic approximation formulae.
\newblock {\em Quarterly of Applied Mathematics}, 4:112--141, 1946.

\bibitem{schoenberg1973cardinal}
I.~J. Schoenberg.
\newblock {\em Cardinal spline interpolation}, volume~12.
\newblock Siam, 1973.

\bibitem{vazquez2016new}
R.~V{\'a}zquez.
\newblock A new design for the implementation of isogeometric analysis in
  {O}ctave and {M}atlab: {GeoPDEs 3.0}.
\newblock {\em Computers \& Mathematics with Applications}, 72(3):523--554,
  2016.

\end{thebibliography}

\end{document}